\documentclass{amsart}%

\usepackage{amssymb}
\usepackage{amsmath}
\usepackage{amsthm}
\usepackage{amscd,mdwlist, mathtools}
\usepackage[margin=1in]{geometry}
\usepackage[utf8]{inputenc}
\usepackage{hyperref}
\usepackage{xcolor}
\usepackage{cite}
\usepackage{dsfont} 
\usepackage{enumerate}
\usepackage{graphicx}
\usepackage{csquotes}
\usepackage{float}
\usepackage{tikz}
\usetikzlibrary{lindenmayersystems}
	\pgfdeclarelindenmayersystem{Sierpinski triangle}{
		\symbol{X}{\pgflsystemdrawforward}
		\symbol{Y}{\pgflsystemdrawforward}
		\rule{X -> X-Y+X+Y-X}
		\rule{Y -> YY}
	}%

\usepackage{circuitikz}

\theoremstyle{plain}
\newtheorem{theorem}{Theorem}[section]
\newtheorem{proposition}[theorem]{Proposition}
\newtheorem{corollary}[theorem]{Corollary}
\newtheorem{lemma}[theorem]{Lemma}

\theoremstyle{remark}
\newtheorem{remark}[theorem]{Remark}
\newtheorem{examples}[theorem]{Examples}
\newtheorem{assumption}[theorem]{Assumption}

\DeclareMathOperator*{\esssup}{ess\,sup}

\DeclareMathOperator{\supp}{supp}

\DeclareMathOperator{\lin}{span}
\DeclareMathOperator{\loc}{loc}

\DeclareMathOperator{\diverg}{div}

\DeclareMathOperator{\sgn}{sgn}

\newcommand{\norm}[1]{\left\| #1 \right\|}

\def\XXint#1#2#3{{\setbox0=\hbox{$#1{#2#3}{\int}$}
		\vcenter{\hbox{$#2#3$}}\kern-.5\wd0}}

\begin{document}

\title{On equations of continuity and transport type on metric graphs and fractals}

\author{Michael Hinz$^1$}
\address{$^1$ Fakult\"{a}t f\"{u}r Mathematik, Universit\"{a}t Bielefeld, Postfach 100131, 33501 Bielefeld,
Germany}
\email{mhinz@math.uni-bielefeld.de}
\author{Waldemar Schefer$^2$}
\address{$^2$ Fakult\"{a}t f\"{u}r Mathematik, Universit\"{a}t Bielefeld, Postfach 100131, 33501 Bielefeld,
Germany}
\thanks{$^1$, $^2$ Research supported by the DFG International Research Training Group 2235: \enquote{Searching for the regular in the irregular: Analysis of singular and random systems}. $^2$ Research supported in part by the Bielefeld Young Researchers Fund.}
\email{wschefer@math.uni-bielefeld.de}

\date{\today}
\begin{abstract}
We study first order equations of continuity and transport type on metric spaces of martingale dimension one, including finite metric graphs, p.c.f. self-similar sets and classical Sierpi\'nski carpets. On such spaces solutions of the continuity equation in the weak sense are generally non-unique. We use semigroup theory to prove a well-posedness result for divergence free vector fields and under suitable loop and boundary conditions. It is the first well-posedness result for first order equations with scalar valued solutions on fractal spaces. A key tool is the concept of boundary quadruples recently introduced by Arendt, Chalendar and Eymard. To exploit it, we prove a new domain characterization for the relevant first order operator and a novel integration by parts formula, which takes into account the given vector field and the loop structure of the space. We provide additional results on duality and on metric graph approximations in the case of periodic boundary conditions.
\tableofcontents
\end{abstract}

\keywords{Continuity equation, Transport equation, Fractals, Metric Graphs, Dirichlet forms, Dissipative operators}
\subjclass[2010]{28A80, 31C25, 35F10, 35F16, 35R02, 47A07, 47B44, 47D06}

\maketitle

\section{Introduction}

The main purpose of this article is to establish a first well-posedness result for the continuity equation on fractal spaces. Given a space $X$ on which vector fields $b$ and a divergence operator $\diverg$ can be introduced, the Cauchy problem for the corresponding (homogeneous) continuity equation is of the form
\begin{equation}\label{E:ce}
\begin{cases}
\partial_tu + \diverg(ub) = 0 &\text{in } (0,T) \times X,\\
u(0)=\mathring{u} &\text{on } X.
\end{cases}
\end{equation}
Here $T\in (0,+\infty]$ and $\mathring{u}$ is a suitable initial condition. A priori $b$ may depend on space and time. Equation (\ref{E:ce}) is the prototype of a scalar conservation (or balance) law and has numerous applications in mathematics, physics and other sciences. It is closely related to the transport equation 
\begin{equation}\label{E:te}
\begin{cases}
\partial_tw + b\cdot\nabla w = 0 &\text{in } (0,T) \times X,\\
w(0)=\mathring{w} &\text{on } X,
\end{cases}
\end{equation}
where $\nabla$ denotes the gradient. For divergence free $b$ problems (\ref{E:ce}) and (\ref{E:te}) coincide. In this article we discuss solutions $u$ and $w$ that are real valued functions. 

For $X=\mathbb{R}^n$ and a constant vector field $b\in\mathbb{R}^n$ the equation in (\ref{E:ce}) and (\ref{E:te}) is one of the most basic partial differential equations, cf. \cite[Section 2.1]{Evans}. For $C^1$-vector fields $b$ the ordinary differential equation associated with $b$ is well-posed by the Cauchy-Lipschitz theorem, and using this fact the well-posedness of (\ref{E:ce}) and (\ref{E:te}) in the weak sense can be obtained, cf. \cite[Chapter I]{Crippadiss}. The well-posedness of (\ref{E:ce}) in the weak sense on $X=\mathbb{R}^n$ for, roughly speaking, spatially Sobolev regular $b$ with bounded divergence was proved in \cite{DPL89} and used to construct a kind of flow for the corresponding ordinary differential equation. In \cite{A04} this program was extended to integrable vector fields $b$ of
spatial BV regularity and with bounded divergence, and the concept of regular Lagrangian flows was introduced; see also \cite{A08}. In \cite{AT14} the well-posedness of (\ref{E:ce}) in the weak sense on metric measure spaces $X$, endowed with a strongly local regular Dirichlet form, was proved for suitable vector fields $b$. Based on this, regular Lagrangian flows on such spaces were constructed.  The assumptions in \cite{AT14} are satisfied for many important classes of examples such as weighted Riemannian manifolds with lower Ricci curvature bounds, abstract Wiener spaces and $\mathrm{RCD}(K,\infty)$-spaces. The existence result for solutions of (\ref{E:ce}) in the weak sense proved in \cite{AT14} also holds on metric graphs and fractal spaces, see Theorem \ref{T:ATex} below for a simple special case. However, the uniqueness of solutions in the weak sense generally fails on
metric graphs and fractals, see Remarks \ref{R:announcenonweakwellpos} and \ref{R:noweakwellpos} below. In fact, the uniqueness proof in \cite[Sections 5 and 6]{AT14} relies on, roughly speaking, Bakry-Emery type curvature conditions, which in this form do not hold on general metric graphs and fractals, \cite{BK16, Kaj13}.

Under additional vertex conditions a well-posedness result for (\ref{E:ce}) on metric graphs with edge-wise constant $b$ and solutions in the sense of semigroup theory was proved in \cite{KramarSikolya05}. Related models and further references can be found in \cite{Banasiak16, EngelKramar22}, related applications in \cite{Bressan14} and general background on semigroup methods for networks in \cite{Mugnolo}. On metric graphs first order differential operators and vector fields can be defined in a straightforward manner, and an edge-wise integration by parts identity is available. There is no immediate way to transfer these tools to fractal spaces, except perhaps for finite unions of simple fractal curves, which are not substantially different from metric graphs.

An approach to (\ref{E:ce}) with measure valued solutions was studied in \cite{Bogachev}. Under suitable conditions on $b$, cf. \cite[Theorems 1.1. and 1.4 and Proposition 2.8]{Bogachev}, the uniqueness of solutions of the continuity equation was proved. However, this approach uses continuous test vector fields, and since analogs of (Sub-)Riemannian metrics on fractals are only measurable but not continuous, the notion of continuous vector fields on fractals is not clear. In \cite{Camilli17} a well-posedness result for measure-valued solutions of continuity equations on metric graphs was proved by piecing together solutions on edges, subject to suitable vertex conditions. It is typical for fractal spaces that they cannot be pieced together from macroscopic non-fractal components.

On certain established classes of fractal spaces counterparts of second order elliptic and parabolic equations are well understood. See for instance \cite{Ba98, Ki01, Ki03, Ku93, Str06} and \cite{BB89,BBKT10}, further references are listed in Examples \ref{Ex:qforms} (ii) below. The main tool is Dirichlet form theory, \cite{BH91,FOT94}, which also provides an abstract first order derivation operator $f\mapsto \partial f$, see \cite{CS03,CS09,IRT12}. Counterparts of linear and semilinear second order equations with first order terms on fractals were studied in \cite{HinzMeinert20, HinzMeinert22, LQ19}. Except for basic results on (vector valued) Euler type equations, \cite{HT15}, nothing is known about first order partial differential equations on fractals. 

For most of our results the spaces we discuss are assumed to have martingale dimension one, examples include metric graphs, p.c.f. self-similar sets and classical Sierpi\'nski carpets. In the context of fractal spaces the notion of martingale dimension,
also called (Kusuoka-Hino) index, goes back to \cite{Ku93} and was studied in great detail in \cite{Hino08, Hino10, Hino13}. It is a feature of, roughly speaking, the Dirichlet form on the space. There are examples of spaces with martingale dimension one and arbitrarily large Hausdorff dimension, \cite{Laakso00}. In the martingale dimension one case we can make use of a Poincar\'e duality $\star_\omega$ with respect to a minimal energy-dominant (cf. Section \ref{S:Bilinear} below) $1$-form $\omega$ and introduce first order differential operators $f\mapsto \star_\omega^{-1}\partial f$ taking functions into functions. These operators may be viewed as directional derivatives (vector field actions) and, up to details, coincide with  special cases of the derivations in \cite[Definition 3.1]{AT14} for $b=\omega$. They also generalize special cases of the differential operators $f\mapsto \frac{df}{dg}$ introduced in \cite[Theorem 5.4]{Hino10}. In the metric graph case they coincide with familiar edge-wise defined first order differential operators. 

We first study (\ref{E:ce}) and (\ref{E:te}) in the context of extensions of the operator $-\star_\omega^{-1}\partial$ on $L^2(X,\nu_\omega)$, where $\nu_\omega$ denotes the energy measure of $\omega$, see \cite{BH91, FOT94, Ku93, LeJan} and \cite{HRT13, HR16}. The first result we prove is Theorem \ref{T:duality} on the duality of (\ref{E:ce}) and (\ref{E:te}) in the martingale dimension one case. There we consider solutions of (\ref{E:ce}) in the weak sense. This result may be seen as a generalization of the well-known duality of these equations on $X=\mathbb{R}$, \cite{BouchutJames98}. The vector field $b$ is assumed to be bounded in space and uniformly square integrable in time. 

Motivated by the possible non-uniqueness of solutions in the weak sense on fractals we then study (\ref{E:ce}) and (\ref{E:te}) from a semigroup perspective, aiming to find additional conditions under which well-posedness could be shown. We eventually restrict our attention to vector fields $b$ that are stationary, divergence free, possibly with respect to a nonempty boundary $B$, and minimal energy-dominant; see Sections \ref{S:Bilinear} and \ref{S:BC} for details. On the real line or the circle this limits the scope of (\ref{E:ce}) to the constant coefficient case. However, on fractal spaces 
the situation is drastically different, because the kernel of the divergence is generally infinite-dimensional and therefore the variety of resulting equations (\ref{E:ce}) is huge. The restriction on $b$ to be energy-dominant can be relaxed if the space $X$ is reduced accordingly, Remark \ref{R:medmreduction}. 

Our main assumptions are stated in terms of nonnegative definite symmetric bilinear forms $(\mathcal{E},\mathcal{C})$ that are regular and local in a certain sense and satisfy the Markov property, see Assumptions \ref{A:basic}. This is sufficient to introduce related first order structures and to discuss divergence free minimal energy-dominant $b$.  Only after such a vector field $b$ is fixed, we choose an adapted volume measure to discuss  (\ref{E:ce}), namely the energy measure $\nu_b$ of $b$. Since $b$ is divergence free, the operator $f\mapsto -\star_b^{-1}\partial_B f$ is skew-symmetric on $L^2(X,\nu_b)$. Here $\partial_B$ denotes the restriction of $\partial$ to functions in $\mathcal{C}$ vanishing on $B$. It is not difficult to see that in this setting $f\mapsto \star_b^{-1}\partial_B f$ is the correct candidate to replace $f\mapsto \diverg(f b)$ and $f\mapsto b\cdot \nabla f$ in (\ref{E:ce}) and (\ref{E:te}), see formulas (\ref{E:botpartial}) and (\ref{E:opscoincide}) and Remark \ref{R:inverse} below. Now one would hope that adequate additional conditions might determine $m$-dissipative extensions of $-\star_b^{-1}\partial_B$. By the Lumer-Phillips theorem such extensions generate strongly continuous contraction semigroups on $L^2(X,\nu_b)$. The resulting Cauchy problems could then be viewed as implementations of (\ref{E:ce}) and (\ref{E:te}), augmented by suitable additional conditions. In general the domains of $m$-dissipative extensions of $-\star_b^{-1}\partial_B$ are strictly larger than the domain $\mathcal{C}$ of the bilinear form $\mathcal{E}$. The form $(\mathcal{E},\mathcal{C})$ determines the first order structure, but the dynamics of the mentioned Cauchy problem is governed by the chosen extension of $-\star_b^{-1}\partial_B$.

A key tool to characterize $m$-dissipative extensions is the concept of boundary quadruples recently introduced in \cite{Arendtetal23} by Arendt, Chalendar and Eymard. Boundary quadruples may be regarded as abstract functional analytic integration by parts identities. They exist for any densely defined skew-symmetric operator on a real or complex Hilbert space, \cite[Examples 3.6 and 3.7]{Arendtetal23}.
Our main contribution is a specific construction of boundary quadruples for the operators $-\star_b^{-1}\partial_B$. These boundary quadruples reflect the loop structure of $X$, the boundary $B$ and the vector field $b$ in a meaningful way. Here we provide this construction for, roughly speaking, compact resistance spaces in the sense of \cite{Ki03}. The (implicit) assumption on $X$ to be compact is made for convenience, it can be dropped if function spaces and arguments are modified accordingly. 

We prepare the construction of the boundary quadruples by proving two novel key results. The first is Theorem \ref{T:keydecomp}, which provides a domain characterization for the adjoint $(-\star_b^{-1}\partial_B)^\ast$, in Theorem \ref{T:keydecomp} denoted by $\partial^\bot_B$. From this operator any $m$-dissipative extension of $-\star_b^{-1}\partial_B$ can be obtained by restriction. Theorem \ref{T:keydecomp} provides a representation for functions from the domain of $\partial^\bot_B$, unique up to constants. It describes the difference between an element of this domain and elements of the generally smaller domain $\mathcal{C}$. Using this representation, we then prove the second novel key result, Theorem \ref{T:IbP}. This theorem is a generalization of the ordinary integration by parts identity to spaces of martingale dimension one, formulated for the operators $\partial^\bot_B$. The deviation of $\partial^\bot_B$ from being skew-adjoint is balanced using terms involving the boundary $B$ and, in addition, terms that evaluate the loop structure of $X$ in view of the Poincar\'e duality $\star_b$.

We then assume that the boundary $B$ is finite and, using Theorem \ref{T:IbP}, construct boundary quadruples. For the convenience of the reader this is done in two steps, first for the situation of vector fields with zero normal part at $B$ in Theorem \ref{T:simplequadruple} and then for the general case in Theorem \ref{T:generalquadruple}. Having boundary quadruples at our disposal, we can then exploit \cite[Theorem 3.10]{Arendtetal23} where $m$-dissipative extensions were characterized by linear contractions between boundary spaces; the result is quoted in Theorem \ref{T:Arendt} below. This gives the desired well-posedness result, Theorem \ref{T:wellposed}, for the Cauchy problem (\ref{E:ce}), augmented by an abstract additional condition involving the chosen linear contraction. 

It is desirable to translate this abstract additional condition into an understandable restriction on the dynamics of (\ref{E:ce}). In the classical case of an interval it is simply a boundary condition, but in general the choice of the linear contraction describes a set of boundary and loop conditions. Depending on the structure of the space $X$, the variety of such contractions and resulting conditions can be huge. We provide examples for particularly simple cases of such conditions, in particular, periodic cases, for which we also state additional observations. Classical and metric graph examples are listed to give orientation and to illustrate the meaning of certain quantities. To explain new effects on fractal spaces we discuss examples involving the Sierpi\'nski gasket. Although for metric graphs
we clearly do not claim to have the first well-posedness result, the integration by parts formula in Theorem \ref{T:IbP} and its use in connection with boundary quadruples seem to be new. 

Most of the described results are based on \cite{ScheferThesis}. A fine analysis of the operator $\partial^\bot_B$ in the Sierpi\'nski gasket case will be provided in a follow-up, \cite{Schefer24}.

It is interesting to point out a special case of (\ref{E:ce}). Suppose that the boundary $B$ is nonempty and finite and that $h$ is a minimal energy-dominant harmonic function on $X\setminus B$ which equals zero on a proper subset of $B$ and one on the complement of this subset in $B$. Such $h$ is a simple case of a capacitor (condenser) potential, see \cite[D\'efinition 3]{BeurlingDeny58} or \cite[Chapter VI, Theorem 6.5]{Landkof72}. Its gradient field $\partial h$ is divergence free with respect to $B$. If $V$ is the classical solution of $\partial_tV+\partial_yV=0$ on $[0,1]$ with periodic boundary conditions, then its pull-back $(t,x)\mapsto V(t,h(x))$ to $X$ under the capacitor potential $h$ is a solution of (\ref{E:ce}) on $X$ with the cylindrical initial condition $x\mapsto V(0,h(x))$. This mechanism could be seen as a replacement for the pull-back under the inverse of the parametrization in the case of a simple curve. Without further considerations one cannot say anything about the uniqueness of the solution $(t,x)\mapsto V(t,h(x))$. In Theorem \ref{T:cylindrical} we use the results of the earlier sections to prove that for a certain (explicitely given) linear contraction this is actually the corresponding unique solution of the associated Cauchy problem. In the case of the Sierpi\'nski gasket we also observe that the corresponding unique solutions on approximating metric graphs converge uniformly to $(t,x)\mapsto V(t,h(x))$, Corollary \ref{C:cylindrical}. 

In Section \ref{S:Bilinear} we state our main hypotheses on the bilinear form $(\mathcal{E},\mathcal{C})$ and collect basic facts on the associated first order structures. The existence result for solutions in the weak sense from \cite{AT14} is briefly revisited in Section \ref{S:Dirichletweak}, Poincar\'e duality and the duality of (\ref{E:ce}) and (\ref{E:te}) are discussed in Section \ref{S:Poinc}. Dissipative operators based on stationary vector fields are considered in Section \ref{S:BC}, skew symmetry and the definition of boundary quadruples from \cite{Arendtetal23} in Section \ref{S:Skew}. In Section \ref{S:domains} we collect facts on normal derivatives and on the domains of the operators $\partial^\bot$ in examples. The domain characterization and the integration by parts identities are stated and proved in Section \ref{S:IbP}. Boundary quadruples are constructed in Section \ref{S:quadruples}. Section \ref{S:wellpos} contains the well-posedness result, along with concrete examples. Cylindrical initial conditions and metric graph approximations in the case of periodic boundary conditions are disussed in Section \ref{S:cylindrical}, domains and stationarity in Section \ref{S:statio}. Several technical results and calculations are provided in Appendices.

\section{Bilinear forms and first order structures}\label{S:Bilinear}

Given $a,b\in\mathbb{R}$, we write $a\wedge b$ and $a\vee b$ for their minimum and maximum, respectively. For functions we understand $\wedge$ and $\vee$ in the pointwise sense.

Let $(X,\varrho)$ be a locally compact separable metric space. By $C_c(X)$ we denote the space of compactly supported continuous functions on $X$, endowed with the supremum norm $\|\cdot\|_{\sup}$. Let $\mathcal{C}$ be a subalgebra of $C_c(X)$ and $\mathcal{E}:\mathcal{C}\times\mathcal{C}\to\mathbb{R}$ a nonnegative definite symmetric bilinear form. Our basic assumption is as follows:

\begin{assumption}\mbox{}\label{A:basic}
\begin{enumerate}
\item[(i)] The algebra $\mathcal{C}$ is dense in $C_c(X)$. 
\item[(ii)] The algebra $\mathcal{C}$ is separable with respect to the seminorm $\mathcal{E}^{1/2}$.
\item[(iii)] For any $u\in \mathcal{C}$ we have $(u\wedge 1)\vee 0\in \mathcal{C}$ and $\mathcal{E}((u\wedge 1)\vee 0)\leq \mathcal{E}(u)$.
\item[(iv)] For any $f\in \mathcal{C}$ the linear functional $g\mapsto \mathcal{E}(fg,f)-\frac12\mathcal{E}(f^2,g)$ is positive on $\mathcal{C}$ and bounded with respect to $\|\cdot\|_{\sup}$. 
\item[(v)] The bilinear form $(\mathcal{E},\mathcal{C})$ is \emph{strongly local} in the sense that $\mathcal{E}(f,g)=0$ for all $f,g\in\mathcal{C}$ such that $g$ is constant on a neighborhood of $\supp f$.
\end{enumerate}
\end{assumption}

This point of view is similar to \cite{Allain} in the sense that no reference measure is given and no closability properties of $(\mathcal{E},\mathcal{C})$ are assumed.

\begin{examples}\label{Ex:qforms}\mbox{}
\begin{enumerate}
\item[(i)] Suppose that $\nu$ is a nonnegative Radon measure on $X$ with full support, $(\mathcal{E},\mathcal{D}(\mathcal{E}))$ is a strongly local regular Dirichlet form on $L^2(X,\nu)$ in the sense of \cite{FOT94} and $\mathcal{C}\subset \mathcal{D}(\mathcal{E})\cap C_c(X)$ is a subalgebra of $C_c(X)$, dense both in $C_c(X)$ and in $\mathcal{D}(\mathcal{E})$, and such that $u\in \mathcal{C}$ implies $(u\wedge 1)\vee 0\in \mathcal{C}$.  Then Assumption \ref{A:basic} is satisfied.
\item[(ii)] Suppose that $(\mathcal{E},\mathcal{C})$ is a resistance form on a set $X$, \cite[Definition 2.8]{Ki03}, the metric $\varrho$ is the associated resistance metric, $(X,\varrho)$ is compact and $(\mathcal{E},\mathcal{C})$ is strongly local as in (v). Then Assumption \ref{A:basic} is satisfied. Examples are the usual Dirichlet integrals on the unit interval $[0,1]$, on the unit circle $S^1$ and on finite metric graphs with Kirchhoff vertex conditions, \cite{BK16, BerkolaikoKuchment, ExnerPost09, GK21, Kostrykin99, Kostrykin00, Kuchment04, Kuchment05, Kurasov, Mugnolo, Post12}. More complex examples are resistance forms on nested fractals, \cite{Li}, p.c.f. self-similar sets, \cite{BP88, Ki89, Ki93, Ki93b, Ku89, Ku93, Ki03, Ki01}, finitely ramified fractals, \cite{Ba98, T08}, fractafolds, \cite{Str03}, certain Sierpi\'nski carpets, \cite{BB89, BBKT10}, random fractals \cite{Hambly92, Hambly97}, graph-directed and $V$-variable fractals, \cite{FreibergHamblyHutchinson17, HamblyNyberg03}, Laaks\o-spaces, \cite{Laakso00,BE04,Steinhurst}, diamond fractals, \cite{Alonso18,HamblyKumagai10}, stretched gaskets \cite{AlonsoFreibergKigami18} and certain Julia sets, \cite{RT10}. Very recent results on resistance forms on new classes of fractals can be found in \cite{CaoQiu, Ki23}.
\end{enumerate}
\end{examples}

The Riesz representation theorem and polarization imply that for any $f_1,f_2\in \mathcal{C}$ there is a unique finite signed Radon measure $\nu_{f_1,f_2}$ on $X$ satisfying
\begin{equation}\label{E:defenergymeasure}
\int_Xg\:d\nu_{f_1,f_2} = \frac12\{\mathcal{E}(f_1g,f_2)+\mathcal{E}(f_1,f_2g)-\mathcal{E}(f_1f_2,g)\},\quad g\in\mathcal{C};
\end{equation}
it is called the \emph{mutual energy measure of $f_1$ and $f_2$}, see for instance \cite{FOT94, Hino08, Hino10, Ku93, LeJan}. For any $f\in \mathcal{C}$ the \emph{energy measure} $\nu_{f}:=\nu_{f,f}$ \emph{of $f$} is nonnegative.

\begin{remark}\label{R:atomfree}
In the situation of Examples \ref{Ex:qforms} (i) and (ii) all (mutual) energy measures are atom free, \cite[Theorems 5.2.3 and 7.1.1]{BH91}.
\end{remark}

We recall basic facts on the first order structure associated with a bilinear form $(\mathcal{E},\mathcal{C})$. Let $C_{a}(X\times X)$ denote the space of all continuous antisymmetric real valued functions on $X\times X$ and write
\begin{equation}\label{E:action}
(\overline{g}v)(x,y):=\overline{g}(x,y)v(x,y),\quad x,y\in X, 
\end{equation}
for any $v\in C_a(X\times X)$ and $g\in C_c(X)$, where $\overline{g}(x,y):=\frac12(g(x)+g(y))$.
Definition (\ref{E:action}) defines an action of $C_c(X)$ on $C_a(X\times X)$, turning it into a module. By $d_u:\mathcal{C}\to C_a(X\times X)$ we denote the universal derivation, 
\begin{equation}\label{E:universalder}
d_uf(x,y):=f(x)-f(y),\quad x,y\in X,
\end{equation} 
and by $\Omega_a^1(X)$ the submodule of $C_a(X\times X)$ consisting of finite linear combinations of the form $\sum_i \overline{g}_i d_u f_i$ with $g_i\in C_c(X)$ and  $f_i \in \mathcal{C}$. For $f,g\in \mathcal{C}$ we have $d_u(fg)=\overline{f} d_ug +\overline{g} d_uf$. 

The bilinear extension of 
\begin{equation}\label{E:scalarprod}
\left\langle \overline{g}_1 d_uf_1, \overline{g}_2 d_uf_2\right\rangle_{\mathcal{H}}:=\int_X g_1g_2\:d\nu_{f_1,f_2},\quad f_1,f_2\in \mathcal{C},\quad g_1,g_2 \in C_c(X),
\end{equation}
defines a symmetric nonnegative definite bilinear form $\left\langle\cdot,\cdot\right\rangle_{\mathcal{H}}$ on $\Omega_a^1(X)$, see for instance \cite[Remark 2.1]{HRT13}.

\begin{examples} In the situation of Examples \ref{Ex:qforms} (ii) there is a sequence $(V_n)_n$ of finite subsets $V_n\subset X$ with $V_n\subset V_{n+1}$, $n\geq 0$, and such that $\bigcup_{n\geq 0} V_n$ is dense in $(X,\varrho)$. For any such sequence $(V_n)_n$ we have 
$\mathcal{E}(u)=\lim_{n}\mathcal{E}_{V_{n}}(u)$ for all $ u\in\mathcal{F}$, \cite[Proposition 2.10 and Theorem 2.14]{Ki03}. Each $\mathcal{E}_{V_n}$ is of the form 
\begin{equation}\label{E:approxbydiscreteforms}
\mathcal{E}_{V_n}(u)=\frac12\sum_{x\in V_n}\sum_{y\in V_n}c(n;x,y)(u(x)-u(y))^2,\quad u\in\mathcal{F},
\end{equation}
with constants $c(n;x,y)\geq 0$ symmetric in $x$ and $y$. As a consequence,
\[\left\langle \overline{g}_1 d_uf_1, \overline{g}_2 d_uf_2\right\rangle_{\mathcal{H}}=\lim_{n\to\infty} \frac{1}{2}\sum_{x\in V_n}\sum_{y\in V_n} c(n;x,y) \overline{g_1}(x,y)\overline{g_2}(x,y)d_uf_1(x,y)d_uf_2(x,y)\]
for all  $f_1,f_2\in \mathcal{C}$ and $g_1,g_2 \in C_c(X)$, see \cite[Lemma 3.1]{HinzMeinert20}.
\end{examples}

Writing $\left\|\cdot\right\|_{\mathcal{H}}=\sqrt{\left\langle\cdot,\cdot\right\rangle_{\mathcal{H}}}$ for the Hilbert seminorm on $\Omega_a^1(X)$ induced by (\ref{E:scalarprod}) and factoring out elements of zero seminorm gives the space $\Omega^1_a(X)/\ker \left\|\cdot\right\|_{\mathcal{H}}$, and completing the latter with respect to $\left\|\cdot\right\|_{\mathcal{H}}$ gives a Hilbert space $(\mathcal{H},\left\langle\cdot,\cdot\right\rangle_{\mathcal{H}})$. This construction was already mentioned in \cite[Chapter V, Exercise 5.9]{BH91}. In a general Dirichlet form context it was introduced in \cite{CS03}. Given $\sum_i \overline{g}_i d_uf_i\in  \Omega_a^1(X)$, we write $\big[\sum_i \overline{g}_i d_uf_i\big]_\mathcal{H}$ to denote its equivalence class in $\Omega^1_a(X)/\ker \left\|\cdot\right\|_{\mathcal{H}}$. 

The action (\ref{E:action}) induces an action of $C_c(X)$ on $\mathcal{H}$: For $v\in \mathcal{H}$, $g\in C_c(X)$ and $(v_n)_n\subset \Omega_a^1(X)$ with $\lim_n [v_n]_\mathcal{H}=v$ in $\mathcal{H}$, we define $g v \in \mathcal{H}$ by 
$g v:=\lim_n [\overline{g} v_n]_\mathcal{H}$. By (\ref{E:action}) and (\ref{E:scalarprod}) this definition is correct, and
\begin{equation}\label{E:boundedaction}
\left\|g v\right\|_{\mathcal{H}}\leq \left\|g\right\|_{\sup}\left\|v\right\|_{\mathcal{H}},\quad v\in \mathcal{H},\quad g\in C_c(X).
\end{equation}
Given $f\in\mathcal{C}$, we denote the $\mathcal{H}$-equivalence class of the universal derivation $d_uf$ as in (\ref{E:universalder}) by $\partial f:=[d_uf]_\mathcal{H}$. The map $f\mapsto \partial f$ defines a linear operator 
\begin{equation}\label{E:partial}
\partial: \mathcal{C}\to\mathcal{H}
\end{equation}
which satisfies the identity 
\begin{equation}\label{E:iso}
\left\|\partial f\right\|_{\mathcal{H}}^2=\mathcal{E}(f)
\end{equation} 
for any $f\in \mathcal{C}$ and the Leibniz rule $\partial(fg)=f\partial g+g\partial f$
for any $f,g\in \mathcal{C}$. The space of finite linear combinations $\sum_i g_i\partial f_i$ with $f_i,g_i\in \mathcal{C}$ is dense in $\mathcal{H}$.

Using $f\mapsto \|f\|_{\mathcal{C}}:=\mathcal{E}(f)^{1/2}+\|f\|_{\sup}$ the space $\mathcal{C}$ can be made into a normed vector space. By (\ref{E:iso}) the linear operator $\partial$ in (\ref{E:partial}) is bounded with respect to this norm. Given $f,g\in \mathcal{C}$, we have 
\[\mathcal{E}(fg)^{1/2}\leq \|f\|_{\sup}\mathcal{E}(g)^{1/2}+\|g\|_{\sup}\mathcal{E}(f)^{1/2}\]
by (\ref{E:iso}), the Leibniz rule and (\ref{E:boundedaction}). As a consequence we also have $\|fg\|_{\mathcal{C}}\leq 2\|f\|_{\mathcal{C}}\|g\|_{\mathcal{C}}$. See \cite[Section 3]{HRT13}.

Let $\mathcal{C}^\ast$ denote the topological dual of $\mathcal{C}$. The dual of $\partial$ is the bounded linear operator $\partial^\ast:\mathcal{H}\to \mathcal{C}^\ast$ determined by
\begin{equation}\label{E:dual}
(\partial^\ast v)(\varphi)=\left\langle v,\partial \varphi\right\rangle_{\mathcal{H}},\quad \varphi\in\mathcal{C}.
\end{equation}
It may be viewed as a substitute for the coderivation operator in a distributional sense. The kernel $\ker \partial^\ast$ of $\partial^\ast$ is a closed linear subspace of $\mathcal{H}$. The space $\mathcal{H}$ admits the orthogonal decomposition 
\begin{equation}\label{E:Hodge}
\mathcal{H}=\overline{\partial(\mathcal{C})}\oplus \ker \partial^\ast;
\end{equation}
here $\overline{\partial(\mathcal{C})}$ is the closure of the image $\partial(\mathcal{C})$ of $\mathcal{C}$ under $\partial$.

Setting $\nu_{g_1\partial f_1,g_2\partial f_2}:=g_1g_2\nu_{f_1,f_2}$ for $g_1,g_2\in C_c(X)$ and $f_1,f_2\in\mathcal{C}$, extending bilinearly and using the mentioned density, one obtains a uniquely determined positive definite bilinear mapping $(v,w)\mapsto \nu_{v,w}$ from $\mathcal{H}\times\mathcal{H}$ into the space of 
finite signed Radon measures; it satisfies $\nu_{v,w}(X)=\left\langle v,w\right\rangle_{\mathcal{H}}$ for all $v,w\in \mathcal{H}$, see \cite[Lemma 2.1]{HRT13}. For any $v\in \mathcal{H}$ the measure $\nu_v:=\nu_{v,v}$ is nonnegative. Note that $\nu_{\partial f}=\nu_f$, $f\in \mathcal{C}$. Given $v\in \mathcal{H}$, the linear map $g\mapsto gv$ from $C_c(X)$ into $\mathcal{H}$ extends to a unique isometry from $L^2(X,\nu_v)$ into $\mathcal{H}$, note that
\begin{equation}\label{E:isointo}
\|gv\|_{\mathcal{H}}^2=\int_Xg^2d\nu_v, \quad g\in L^2(X,\nu_v).
\end{equation}

Essentially following \cite[Definition 2.1]{Hino10}, we call a $\sigma$-finite Borel measure $\nu$ on $X$ \emph{energy-dominant} for $(\mathcal{E},\mathcal{C})$ if for any $f\in \mathcal{C}$ the energy measure $\nu_f$ of $f$ is absolutely continuous with respect to $\nu$, that is, $\nu_f\ll\nu$. In this case there is a density $\frac{d\nu_f}{d\nu}\in L^1(X,\nu)$. An energy-dominant measure $\nu$ is called \emph{minimal}, \cite[Definition 2.1]{Hino10}, if for any other energy-dominant measure $\nu'$ we have $\nu\ll\nu'$. If $f\in \mathcal{C}$ or $v\in\mathcal{H}$ are such that $\nu_f$ or $\nu_v$ are minimal energy-dominant, then we agree to call $f$ or $v$ \emph{minimal energy-dominant}, respectively. The set of energy-dominant $f\in \mathcal{C}$ and $v\in \mathcal{H}$ is dense in $\mathcal{C}$ and $\mathcal{H}$, respectively. Under slightly different assumptions this was first proved in \cite[Proposition 2.7]{Hino10}, a formulation under Assumption \ref{A:basic} is provided in Proposition \ref{P:plenty}. Any two minimal energy-dominant 
measures are mutually absolutely continuous. Abusing language, we say that a Borel set $B\subset X$ \emph{is an $(\mathcal{E},\mathcal{C})$-null set} if $\nu(B)=0$ for some (hence all) minimal energy-dominant 
measure(s) for $(\mathcal{E},\mathcal{C})$.

Suppose that $\nu$ is a minimal energy-dominant measure. Varying \cite[Definition 2.9]{Hino10} slightly, we call a Borel function $p(\cdot):X\to\mathbb{N}\cup\{+\infty\}$ a \emph{pointwise index} of $(\mathcal{E},\mathcal{C})$ if for any $N\in \mathbb{N}$ and $f_1,...,f_N\in\mathcal{C}$ and $\nu$-a.e. $x\in X$ the rank of the matrix $\big(\frac{d\nu_{f_i,f_j}}{d\nu}\big)_{i,j=1}^N$ is bounded by $p(x)$ and $p(\cdot)$ is minimal in the sense that $p(x)\leq p'(x)$ at $\nu$-a.e. $x\in X$ for any other such function $p'(\cdot)$. The \emph{index} or \emph{martingale dimension} of $(\mathcal{E},\mathcal{C})$ is then defined to be the number $p:=\nu - \esssup_{x\in X}p(x)$. A justification for this double terminology was given in \cite[Theorem 3.4]{Hino10}. Here we prefer to use the name \enquote{martingale dimension}, as it avoids possible confusion with indices in an operator theoretic sense. Pointwise indices always exist and are greater than zero at $\nu$-a.e. $x\in X$ and, as a consequence, the martingale dimension always exists and is greater than zero unless $\mathcal{E}$ is the zero form. Under slightly different hypotheses this has been proved in \cite[Proposition 2.10 and Proposition 2.11]{Hino10}, the proofs of these statements remain valid under our assumptions. The $\nu$-equivalence class of pointwise indices is unique and, as a consequence, the martingale dimension $p$ is unique. In particular, these notions do not depend on the choice of the minimal energy-dominant measure.

If $\nu$ is a minimal energy-dominant measure, then there is a measurable field $(\mathcal{H}_x)_{x\in X}$ of Hilbert spaces $(\mathcal{H}_x,\left\langle\cdot,\cdot\right\rangle_{\mathcal{H}_x})$ such that its direct integral $L^2(X,(\mathcal{H}_x)_{x\in X},\nu)$ with respect to $\nu$,  \cite{Dixmier81, Takesaki79}, is isometrically isomorphic to $\mathcal{H}$ and 
\begin{equation}\label{E:scalarprodHx}
\left\langle (\partial f_1)_x,(\partial f_2)_x\right\rangle_{\mathcal{H}_x}=\frac{d\nu_{f_1,f_2}}{d\nu}(x)\quad\text{for $\nu$-a.e. $x\in X$}
\end{equation} 
and all $f_1,f_2\in \mathcal{C}$. See \cite[Chapter III]{Eberle99}. In this notation we identify the Hilbert spaces $L^2(X,(\mathcal{H}_x)_{x\in X},\nu)$ and $\mathcal{H}$ and regard an element $v$ of $\mathcal{H}$ as a $\nu$-equivalence class of square integrable measurable sections $(v_x)_{x\in X}$ of $(\mathcal{H}_x)_{x\in X}$. The measurable field $(\mathcal{H}_x)_{x\in X}$ is unique up to $\nu$-essential isometry; see \cite[Section 2]{HRT13} for details. We observe that 
\[\left\langle v,w\right\rangle_\mathcal{H}=\int_X \left\langle v_x,w_x\right\rangle_{\mathcal{H}_x}\nu(dx)\] 
for any $v,w\in\mathcal{H}$ and that $(g v)_x=g(x)v_x$ for any $g\in C_c(X)$ and $v\in\mathcal{H}$. We write $L^\infty(X,(\mathcal{H}_x)_{x\in X},\nu)$ for the space of $\nu$-equivalence classes of $\nu$-essentially bounded measurable sections $v=(v_x)_{x\in X}$ on $X$, endowed with the norm $\left\|v\right\|_{L^\infty(X,(\mathcal{H}_x)_{x\in X}, \nu)}:=\nu-\esssup_{x\in X} \left\|v_x\right\|_{\mathcal{H}_x}$. If $v\in L^\infty(X,(\mathcal{H}_x)_{x\in X},\nu)$ and $g\in L^2(X,\nu)$ then we define $g v$ as the $\nu$-equivalence class of $(\widetilde{g}(x)\widetilde{v}_x)_{x\in X}$, where $\widetilde{g}$ and $\widetilde{v}$ are arbitrary representatives of $g$ and $v$, respectively. The class $gv$ is an element of $\mathcal{H}$ and
\begin{equation}\label{E:boundedfields}
\left\|g v\right\|_{\mathcal{H}}^2=\int_X g(x)^2\left\|v_x\right\|_{\mathcal{H}_x}^2\nu(dx)\leq \left\|v\right\|_{L^\infty(X,(\mathcal{H}_x)_{x\in X},\nu)}^2\left\|g\right\|_{L^2(X,\nu)}^2.
\end{equation}
If in addition $f\in L^\infty(X,\nu)$, then $f(gv)=(fg)v$. If a different minimal energy-dominant measure $\nu'$ is used in place of $\nu$, then the particular scalar products in (\ref{E:scalarprodHx}) change by the conformal factor $\frac{d\nu}{d\nu'}$, but 
the vector spaces $\mathcal{H}_x$ remain unchanged at $\nu$-a.e. $x\in X$. Any pointwise index $p$ of $(\mathcal{E},\mathcal{C})$ is $\nu$-a.e. equal to the dimension of the vector spaces $\mathcal{H}_x$, that is, $p(x)=\dim \mathcal{H}_x$ for $\nu$-a.e. $x\in X$.

\section{Dirichlet forms and solutions in the weak sense}\label{S:Dirichletweak}

Let $T>0$ be fixed. The definitions provided in the preceding section suggest to formally rephrase (\ref{E:ce}) as
\begin{equation}\label{E:cemod}
\begin{cases}
\partial_tu = \partial^\ast(ub)  &\text{on } (0,T) \times X,\\
u(0)=\mathring{u} &\text{on } X.
\end{cases}
\end{equation}

Let $\nu$ be a nonnegative Radon measure on $X$. In this section we make use of the following assumption.
\begin{assumption}\mbox{}\label{A:Dirichletform}
\begin{enumerate}
\item[(i)] The form $(\mathcal{E},\mathcal{C})$ is as in Assumption \ref{A:basic} and  its closure on $L^2(X,\nu)$ is a Dirichlet form $(\mathcal{E},\mathcal{D}(\mathcal{E}))$.
\item[(ii)] The measure $\nu$ is minimal energy-dominant for $(\mathcal{E},\mathcal{D}(\mathcal{E}))$ in the sense of \cite{Hino10}.
\item[(iii)] We have $\frac{d\nu_f}{d\nu}\in L^\infty(X,\nu)$ for all $f\in \mathcal{C}$.
\end{enumerate}
\end{assumption}

\begin{remark}\label{R:timechange}\mbox{}
\begin{enumerate}
\item[(i)] Under Assumption \ref{A:Dirichletform} (i) the statement of Assumption \ref{A:Dirichletform} (ii) implies that $(\mathcal{E},\mathcal{D}(\mathcal{E}))$ admits a carr\'e du champ, \cite[Definition I.4.1.2]{BH91}.
\item[(ii)] Under Assumption \ref{A:Dirichletform} (i) a $\sigma$-finite Borel measure $\nu$ on $X$ is minimal energy-dominant for $(\mathcal{E},\mathcal{C})$ if and only if it is minimal energy-dominant for $(\mathcal{E},\mathcal{D}(\mathcal{E}))$ in the sense of \cite[Definition 2.1]{Hino10}; this follows using the density of $\mathcal{C}$ in $\mathcal{D}(\mathcal{E})$ and standard estimates, \cite[formula (2.1)]{Hino10}. Since the matrix rank is lower semicontinuous, it also follows that the pointwise index and the martingale dimension of $(\mathcal{E},\mathcal{C})$ equal the pointwise index and the martingale dimension of $(\mathcal{E},\mathcal{D}(\mathcal{E}))$ in the sense of \cite[Definition 2.9]{Hino10}, respectively.
\item[(iii)] Suppose that $(\mathcal{E},\mathcal{D}(\mathcal{E}))$ is a strongly local regular Dirichlet form in the sense of \cite{FOT94}. Then one can construct a finite Radon measure $\widetilde{\nu}$ and find an algebra $\mathcal{C}\subset \mathcal{D}(\mathcal{E})\cap C_c(X)$, dense in $C_c(X)$, such that the closure $(\widetilde{\mathcal{E}},\mathcal{D}(\widetilde{\mathcal{E}}))$ of $(\mathcal{E},\mathcal{C})$ on $L^2(X,\widetilde{\nu})$ satisfies Assumption \ref{A:Dirichletform}, cf. \cite[Section 2]{HinzKelleherTeplyaev15}. This is a variant of well-known results on time change, see for instance \cite{FOT94, Hinz16,HRT13,KN91}.
Given a resistance form $(\mathcal{E},\mathcal{C})$ as in Examples \ref{Ex:qforms} (ii), one can proceed similarly to create a Dirichlet form satisfying Assumption \ref{A:Dirichletform}.
\end{enumerate}
\end{remark}

If Assumption \ref{A:Dirichletform} (i) is satisfied, then the domain $\mathcal{D}(\mathcal{E})$, endowed with the norm
\[\|f\|_{\mathcal{D}(\mathcal{E})}:=\Big(\mathcal{E}(f)+\|f\|_{L^2(X,\nu)}^2\Big)^{1/2},\]
is a Hilbert space. In general the vector space $\mathcal{D}(\mathcal{E})$ depends on the choice of $\nu$. Moreover, the operator (\ref{E:partial}) extends to a closed unbounded linear operator $\partial$ from $L^2(X,\nu)$ into $\mathcal{H}$ with dense domain $\mathcal{D}(\mathcal{E})$, and (\ref{E:iso}) extends to all $f\in \mathcal{D}(\mathcal{E})$. 

\begin{remark}\label{R:Dform} Suppose that Assumption \ref{A:Dirichletform} (i) holds. It is well known that by approximation energy measures $\nu_f$, nonnegative Radon, but not necessarily finite, can be defined for each element $f$ of the larger space $\mathcal{D}_{\loc}(\mathcal{E})$ of all $f\in L^2_{\loc}(X,\nu)$ such that for any compact $K\subset X$ there is some $g\in\mathcal{D}(\mathcal{E})$ with $f=g$ $\nu$-a.e. on $K$. See \cite[Appendix]{Sturm94}. The operator $\partial$ can be extended to $\mathcal{D}_{\loc}(\mathcal{E})$, it takes values in a locally compact space $\mathcal{H}_{\loc}\supset \mathcal{H}$ which is a local variant of $\mathcal{H}$, see \cite[Section 6]{HinzKelleherTeplyaev15}. Energy measures $\nu_v$, nonnegative Radon, but not necessarily finite, can then be defined for all $v\in \mathcal{H}_{\loc}$.
\end{remark}

If Assumption \ref{A:Dirichletform} (i) and (ii) are satisfied, we can use a special case of \cite[Definition 4.2]{AT14} to give a rigorous meaning to (\ref{E:cemod}): Suppose that $b \in L^2(0, T; L^\infty(X,(\mathcal{H}_x)_{x\in X},\nu))$ and $\mathring{u}\in L^2(X,\nu)$. A function $u\in L^2(0,T;L^2(X,\nu)) $ is called a \emph{solution} of (\ref{E:cemod}) \emph{in the weak sense} if 
\begin{equation}\label{E:weaksol}
-\int_0^T\psi'(t)\int_X u(t)\varphi\: d\nu dt = \int_0^T\psi(t)\langle u(t)b(t), \partial \varphi \rangle_{\mathcal H} dt + \psi(0)\int_X \mathring{u}\varphi\: d\nu
\end{equation}
for every $\psi\in C_c^1([0, T))$ and $\varphi\in \mathcal{C}$. 

\begin{remark}\label{R:derivativeL1}
By (\ref{E:boundedfields}) the function $t\mapsto u(t)b(t)$ in (\ref{E:weaksol}) is an element of $L^1(0,T;\mathcal{H})$. It follows that for any $\varphi\in \mathcal{C}$ the function $t\mapsto \langle u(t)b(t), \partial \varphi \rangle_{\mathcal H}$ is in $L^1(0,T)$. 
\end{remark}
The existence of solutions in the weak sense, bounded with respect to time, was proved in \cite[Theorem 4.3]{AT14}. We quote the following special case of this result.

\begin{theorem}\label{T:ATex} 
Let Assumption \ref{A:Dirichletform} be in force. Let $b\in L^\infty(0, T; L^\infty(X,(\mathcal{H}_x)_{x\in X},\nu))$ be such that $\partial^\ast b\in L^\infty(0, T;L^\infty(X,\nu))$ and let $\mathring{u}\in L^2(X, \nu)$. Then there is a solution $u\in L^\infty(0, T;L^2(X, \nu))$ of \eqref{E:cemod} in the weak sense. For any $\varphi\in \mathcal{C}$ the function $t\mapsto \left\langle u(t),\varphi\right\rangle_{L^2(X,\nu)}$ is continuous on $[0,T)$.
\end{theorem} 

The later developments in this article are partially motivated by the following.

\begin{remark}\label{R:announcenonweakwellpos}
For Dirichlet forms on metric graphs and fractals solutions of (\ref{E:cemod}) in the weak sense will generally not be unique.
An example for the non-uniqueness of solutions in the weak sense is provided in Remark \ref{R:noweakwellpos} (ii).
The method used in \cite{AT14} to prove uniqueness relies on gradient estimates, \cite[Section 6]{AT14}, which in this form do not hold for such spaces. See for instance \cite[Section 5.4]{BK16} and \cite[Section 8.2]{Kaj13}.
\end{remark}

We prepare some more notation for later use. Under Assumption \ref{A:Dirichletform} (i) the adjoint of the operator $(\partial,\mathcal{D}(\mathcal{E}))$ is a closed unbounded and densely defined linear operator $(\partial^\ast,\mathcal{D}(\partial^\ast))$, now from $\mathcal{H}$ into $L^2(X,\nu)$. Its domain is 
\begin{equation}\label{E:partialstar}
\mathcal{D}(\partial^\ast)=\{v\in\mathcal{H}: \text{there is some $g\in L^2(X,\nu)$ such that $\left\langle v, \partial f\right\rangle_{\mathcal{H}}=\left\langle g, f\right\rangle_{L^2(X,\nu)}$ for all $f\in \mathcal{D}(\mathcal{E})$}\},
\end{equation}
and for any $v\in\mathcal{D}(\partial^\ast)$ this $g\in L^2(X,\nu)$ is uniquely determined; one writes $\partial^\ast v:=g$. The use of the symbol $\partial^\ast$ is consistent with our former use of this symbol for the dual in distributional sense: Since $\nu$ is fixed, an element $f$ of the space $L^1_{\loc}(X,\nu)$ may be regarded as the element of $\mathcal{C}^\ast$ defined by
\begin{equation}\label{E:interpret}
f(\varphi):=\int_X f\varphi\:d\nu, \quad \varphi\in \mathcal{C}. 
\end{equation}
If $\omega\in \mathcal{D}(\partial^\ast)$ and $\partial^\ast \omega\in L^2(X,\nu)$ is regarded as an element of $\mathcal{C}^\ast$ in this way, it follows that $\partial^\ast \omega(\varphi)=\left\langle \partial^\ast \omega,\varphi\right\rangle_{L^2(X,\nu)}$ for all $\varphi\in \mathcal{C}$.

The space $\mathcal{D}(\partial^\ast)$, endowed with the norm 
\begin{equation}\label{E:divnorm}
\|\omega\|_{\mathcal{D}(\partial^\ast)}:=\Big(\|\partial^\ast \omega\|^2_{L^2(X,\nu)}+\|\omega\|_{\mathcal{H}}^2\Big)^{1/2},
\end{equation}
is a Hilbert space. 

Clearly $(\partial, \mathcal{D}(\mathcal{E}))$ and $(\partial^\ast, \mathcal{D}(\partial^\ast))$ depend heavily on the choice of $\nu$. However, to keep notation short, we suppress this dependency from notation. The choice of $\nu$ will be clear from the context.

\section{Poincar\'e duality in the martingale dimension one case}\label{S:Poinc}

We recall the concept of Poincar\'e duality for bilinear forms of martingale dimension one. Let Assumption \ref{A:basic} be satisfied and let $\omega\in\mathcal{H}$. Recall that by (\ref{E:isointo}) the linear operator 
\begin{equation}\label{E:multiplier}
g\mapsto g\omega,\quad g\in L^2(X,\nu_\omega),
\end{equation}
is an isometry from $L^2(X,\nu_\omega)$ into $\mathcal{H}$. The following observation is a variant of \cite[Corollary 3.9]{Hino10} and closely related to \cite[Proposition 4.5]{BK16}.

\begin{proposition}\label{P:Poinc}
Let Assumption \ref{A:basic} be satisfied, assume that $\mathcal{E}(f)>0$ for some $f\in\mathcal{C}$ and let $\omega\in\mathcal{H}$. Then the isometry (\ref{E:multiplier}) is onto if and only if the martingale dimension of $(\mathcal{E},\mathcal{C})$ is one and $\omega$ is minimal energy-dominant.
\end{proposition}

\begin{proof}
Suppose that (\ref{E:multiplier}) is onto. Then for each $f\in \mathcal{C}$ there is some $g\in L^2(X,\nu_\omega)$ such that $\partial f=g\omega$. Given such $f$ and a Borel set $A$, we find that 
\[\nu_f(A)=\|\mathbf{1}_A\partial f\|_{\mathcal{H}}^2=\|\mathbf{1}_Ag\omega\|_{\mathcal{H}}^2=\|\mathbf{1}_Ag\|_{L^2(X,\nu_\omega)}^2,\]
and if $\nu_\omega(A)=0$, then also $\nu_f(A)=0$.  This implies that $\nu_\omega$ is energy-dominant for $(\mathcal{E},\mathcal{C})$. It is also minimal, see Lemma \ref{L:coincide}. To see that the martingale dimension $p$ of $(\mathcal{E},\mathcal{C})$ is one, we can follow an argument in the proof of \cite[Theorem 3.4]{Hino10}: Let $f_1,...,f_N\in\mathcal{C}$ and let $g_1,...,g_N\in L^2(X,\nu_\omega)$ be such that $\partial f_i=g_i\omega$. Then $\nu_{f_i,f_j}=g_ig_j\nu_\omega$ for any $i$ and $j$ and therefore the matrix $\big(\frac{d\nu_{f_i,f_j}}{d\nu}(x)\big)_{i,j=1}^N$ equals $(g_1(x),...,g_N(x))^T(g_1(x),...,g_N(x))$ at $\nu_\omega$-a.e. $x\in X$, which shows that its rank is less or equal to one at each such $x$. This implies that $p\leq 1$, and together with \cite[Proposition 2.11]{Hino10} it follows that $p=1$.

Suppose now that the martingale dimension $p$ of $(\mathcal{E},\mathcal{C})$ is one and that $\omega$ is minimal energy-dominant. Let $(\mathcal{H}_x)_{x\in X}$ be a measurable field of Hilbert spaces such that (\ref{E:scalarprodHx}) holds with $\nu_\omega$ in place of $\nu$. Then 
\begin{equation}\label{E:densityone}
\|\omega_x\|_{\mathcal{H}_x}=1\quad \text{for $\nu_\omega$-a.e. $x\in X$,}
\end{equation}
and $\omega_x$ spans $\mathcal{H}_x$ at each such $x$. For arbitrary $v\in\mathcal{H}$ we therefore have 
\begin{equation}\label{E:expandv}
v_x=\big\langle \widetilde{v}_x,\widetilde{\omega}_x\big\rangle_{\mathcal{H}_x}\omega_x 
\end{equation}
at $\nu_\omega$-a.e. $x\in X$, where $\widetilde{v}$ and $\widetilde{\omega}$ are arbitrary representatives of $v$ and $\omega$, respectively. The $\nu_\omega$-equivalence class $g$ of $x\mapsto \big\langle\widetilde{v}_x,\widetilde{\omega}_x\big\rangle_{\mathcal{H}_x}$ is in $L^2(X,\nu_\omega)$, note that
\[\norm{g}_{L^2(X,\nu_\omega)}^2 = \int_X \langle g(x)\omega_x, g(x)\omega_x\rangle_{\mathcal H_x} \nu_\omega(dx) = \int_X \langle v_x, v_x\rangle_{ \mathcal H_x} \nu_\omega(dx) = \norm{v}_{\mathcal H}^2,\]
and obviously $v=g\omega$.
\end{proof}

Now suppose that the martingale dimension of $(\mathcal{E},\mathcal{C})$ is one. Let $\omega\in \mathcal{H}$ be minimal energy-dominant. The linear operator $\star_\omega:L^2(X,\nu_\omega)\to\mathcal{H}$, defined by 
\[\star_\omega g:=g\omega,\quad g\in L^2(X,\nu_\omega),\]
was called a \emph{Hodge star operator} in \cite[Definition 4.4]{BK16}. The inverse $\star^{-1}_\omega:\mathcal{H}\to L^2(X,\nu_\omega)$ of $\star_\omega$ assigns $\star_\omega^{-1} v=g$ to $v\in \mathcal{H}$ having the unique representation $v=g\omega$ with $g\in L^2(X,\nu_\omega)$. The element $\omega=\star_\omega \mathbf{1}$ itself may be seen as an abstract \emph{volume form}.

For $f\in L^\infty(X,\nu_\omega)$ and $v\in\mathcal{H}$ we have
\begin{equation}\label{E:pullstar}
\star_\omega^{-1} (fv)=f\star_\omega^{-1} v.
\end{equation}
Since $v\in L^\infty(X,(\mathcal{H}_x)_{x\in X},\nu_\omega)$ if and only if $\star_\omega^{-1} v\in L^\infty(X,\nu_\omega)$, the same identity also holds for $f\in L^2(X,\nu_\omega)$ and $v\in L^\infty(X,(\mathcal{H}_x)_{x\in X},\nu_\omega)$.

\begin{remark}
Although these facts will not be used later on, we point out that $f\star_\omega g=g\star_\omega f$ if one of $f,g$ is in $L^\infty(X,\nu_\omega)$ and the other in $L^2(X,\nu_\omega)$; likewise $(\star_\omega^{-1} v)w=(\star_\omega^{-1} w)v$ if one of $v,w$ is in $L^\infty(X,(\mathcal{H}_x)_{x\in X},\nu_\omega)$ and the other in $\mathcal{H}$. An element $v\in L^\infty(X,(\mathcal{H}_x)_{x\in X},\nu_\omega)$ is zero if and only if $(\star_\omega^{-1} v)v=0$. This is in line with the classical case, cf. \cite[Sections II.2.6 and II.2.7]{Goldberg}.
\end{remark}

\begin{remark}\label{R:localstar} Recall Remark \ref{R:Dform}. We say that $\omega\in \mathcal{H}_{\loc}$ is minimal energy-dominant if its energy measure $\nu_\omega$ is. For such $\omega\in \mathcal{H}_{\loc}$ we can still define $\star_\omega:L^2(X,\nu_\omega)\to\mathcal{H}$. Also this operator is onto.
\end{remark}

\begin{examples}\label{Ex:indexone} \mbox{}
\begin{enumerate}
\item[(i)] Examples for bilinear forms $(\mathcal{E},\mathcal{C})$ satisfying Assumption \ref{A:basic} and
having martingale dimension one are the Dirichlet integral $\mathcal{E}(f,g)=\int_\mathbb{R} f'(x)g'(x)dx$
on $\mathbb{R}$ with domain $\mathcal{C}=H^1(\mathbb{R})\cap C_c(\mathbb{R})$ and the Dirichlet integral 
$\mathcal{E}(f,g)=\int_0^1 f'(x)g'(x)dx$
on $[0,1]$ with domain $\mathcal{C}=H^1(0,1)$. The operator $\partial$ is the usual derivative $d$. For the line $\mathbb{R}$ the constant form  $dx$ is in $\mathcal{H}_{\loc}$, for the interval $[0,1]$ it is in $\mathcal{H}$. In both cases it is minimal energy-dominant, and its energy measure $\nu_{dx}$  is the one-dimensional Lebesgue measure $\mathcal{L}^1$ or its restriction to $[0,1]$. We have $d f=f'dx$ for $f$ from $H^1(\mathbb{R})$ or $H^1(0,1)$, respectively; here $f'$ denotes the weak derivative of $f$.
\item[(ii)] Another example is the Dirichlet integral $\mathcal{E}(f,g)=\int_{S^1} f'(x)g'(x)dx$
on the circle $S^1$ with domain 
$\mathcal{C}=H^1(S^1)$. For a fixed orientation of $S^1$, the algebra $H^1(S^1)$ can be identified with the Sobolev space $\{f\in H^1(\mathbb{R}):\ f(k)=f(0),\ k\in\mathbb{Z}\}$ of periodic functions through replacing $x\mapsto f(x)$, $x\in S^1$, by $x\mapsto f(e^{2\pi i x})$, $x\in \mathbb{R}$. In this angular coordinate representation $\partial$ is again just the usual derivative $d$. The constant form $dx$ is minimal energy-dominant and in $\mathcal{H}$, its energy measure $\nu_{dx}$ is the one-dimensional Hausdorff measure $\mathcal{H}^1(dx)$ on $S^1$. We have $d f=f'dx$ for all $f\in H^1(S^1)$, where $f'\in L^2(S^1)$ denotes the weak derivative of $f$. 
\item[(iii)] Let $N\geq 1$ and let $X=\bigcup_{i=1}^N C_i$ be the union of circles $C_i$, glued together at a single common point $p$. Let $\mathcal{C}$ be the Sobolev space $H^1(X)$ of all $f\in C(X)$ such that $f|_{C_i}\in H^1(C_i)$, $i=1,...,N$, and write $f_i'$ for the weak derivative of $f|_{C_i}$. Then the Dirichlet integral $(\mathcal{E},H^1(X))$ on $X$, defined by
\[\mathcal{E}(f,g)=\sum_{i=1}^N\int_{C_i} f_i'(x)g_i'(x)dx,\] 
satisfies Assumption \ref{A:basic} and has martingale dimension one. On each circle $C_i$ individually the derivation $\partial$ agrees with $d$ as in (ii), and we can identify $\partial f$ with $(f_i'dx_i)_{i=1}^N$.
\item[(iv)] Suppose that $X=\bigcup_{e\in E} e$ is a finite metric tree with root $q_0$, leaves $q_1,...,q_N$ and edge set $E$; each edge $e\in E$ being a copy of $[0,1]$. Let 
\begin{equation}\label{E:DirichletdomainKirchhoff}
H^1(X)=\{f\in C(X):\ f|_e\in H^1(e),\ e\in E\}
\end{equation}
and for each $f\in H^1(X)$ and $e\in E$, let $f_e' \in L^2(e)$ be the weak derivative of $f|_e \in H^1(e)$. Then the Dirichlet integral $(\mathcal{E},H^1(X))$ on $X$, where
\begin{equation}\label{E:Dirichletintergralmg}
\mathcal{E}(f,g)=\sum_{e\in E}\int_e f_e'(x)g_e'(x)dx
\end{equation}
satisfies Assumption \ref{A:basic} and has martingale dimension one. For any $f\in H^1(X)$ we can identify $\partial f$ with $(f_e'dx_e)_{e\in E}$, where $dx_e$ denotes the constant form $dx$ on the edge $e$.  
\item[(v)] More generally, let $\mathcal{E}$ be the Dirichlet integral (\ref{E:Dirichletintergralmg}) on a finite metric graph $X$. If the specified domain is the space $H^1(X)$ of continuous edge-wise energy finite functions as in (\ref{E:DirichletdomainKirchhoff}), then one says that \emph{Kirchhoff vertex conditions} are specified. In this case $(\mathcal{E},H^1(X))$ satisfies Assumption \ref{A:basic} and has martingale dimension one. Items (iii) and (iv) are special cases.
\item[(vi)] Structurally more complex examples for bilinear forms $(\mathcal{E},\mathcal{C})$ satisfying Assumption \ref{A:basic} and having martingale dimension one are standard resistance forms on nested fractals, see \cite{Ku89} or \cite[Theorem 4.4.]{Hino08}, or p.c.f. self-similar sets, see \cite[Theorem 4.5]{Hino10} and \cite[Theorem 4.10]{Hino13}. For fractafolds, \cite{Str03}, based on p.c.f. fractals these properties are immediate from these results. That the martingale dimension for resistance forms on finitely ramified sets is one was conjectured in \cite{T08}. For certain Sierpi\'nski carpets, which are not finitely ramified, it was shown in \cite[Theorem 4.16]{Hino13}.  For Laaks\o-spaces, \cite{Laakso00,BE04,Steinhurst}, diamond fractals, \cite{Alonso18,HamblyKumagai10}, stretched gaskets \cite{AlonsoFreibergKigami18} and certain Julia sets, \cite{RT10}, this one-dimensionality is not difficult to see. The Hausdorff dimension may be arbitrarily large,  \cite{Laakso00}.
\end{enumerate}
\end{examples}

Now assume that $(\mathcal{E},\mathcal{C})$ satisfies Assumption \ref{A:basic}, the martingale dimension of $(\mathcal{E},\mathcal{C})$ is one, $\omega\in\mathcal{H}$ is minimal energy-dominant and Assumption \ref{A:Dirichletform} (i) holds with $\nu=\nu_\omega$. Under these assumptions also Assumption \ref{A:Dirichletform} (ii) holds. 

An application of the Poincar\'e duality gives a duality between continuity and transport type equations. To formulate it, we follow \cite[Section 4]{BK16} and consider the densely defined and closed unbounded linear operator $\partial^\bot$ from $L^2(X,\nu_\omega)$ into $L^2(X,\nu_\omega)$ with domain 
\begin{equation}\label{E:botdomain}
\mathcal{D}(\partial^\bot):=\star_\omega^{-1}\mathcal{D}(\partial^\ast)=\{f\in L^2(X,\nu_\omega): \star_\omega f\in \mathcal{D}(\partial^\ast)\}
\end{equation}
and defined by 
\begin{equation}\label{E:botpartial}
\partial^\bot f:=-\partial^\ast \star_\omega f,\quad  f\in \mathcal{D}(\partial^\bot).
\end{equation}
Here $\mathcal{D}(\partial^\ast)$ and $\partial^\ast$ are as in (\ref{E:partialstar}), now with $\nu=\nu_\omega$. Note that $(\partial^\bot, \mathcal{D}(\partial^\bot))$ depends on the choice of $\omega$, which will be clear from the context.

Given $b \in L^2(0, T; L^\infty(X,(\mathcal{H}_x)_{x\in X},\nu_\omega))$ and $\mathring{w}\in \mathcal{D}(\partial^\bot)$, we call a function $w\in C([0,T];L^2(X,\nu_\omega))\cap L^0(0,T;\mathcal{D}(\partial^\bot))$ a \emph{solution} of 
\begin{equation}\label{E:tte}
\begin{cases}
\partial _tw = -(\partial^\bot w) \star_\omega^{-1} b  &\text{on } (0,T) \times X,\\
w(0)=\mathring{w} &\text{on } X
\end{cases}
\end{equation}
\emph{in the integral sense} if $\partial^\bot w\in L^2(0,T;L^2(X,\nu_\omega))$ and $w(t)=\mathring{w}-\int_0^t(\partial^\bot w(s)) \star_\omega^{-1} b(s)ds$, $t\in [0,T]$.

The equation in (\ref{E:tte}) is a \emph{transport type equation}. In view of Remark \ref{R:localstar2}, the following Theorem \ref{T:duality} may be seen as a generalization of the well-known duality between continuity and transport equations on $\mathbb{R}$, see for instance \cite[Section 1 and Lemma 2.2.1]{BouchutJames98} and \cite{Gusev19}. Theorem \ref{T:duality} will not be used in later sections. A proof can be found in Appendix \ref{S:duality}.

\begin{theorem}\label{T:duality}
Suppose that Assumption \ref{A:basic} holds, the martingale dimension of $(\mathcal{E},\mathcal{C})$ is one, $\omega\in\mathcal{H}$ is minimal energy-dominant and Assumption \ref{A:Dirichletform} (i) holds with $\nu=\nu_\omega$. Let $b \in L^2(0, T; L^\infty(X,(\mathcal{H}_x)_{x\in X},\nu_\omega))$, $\mathring{w}\in \mathcal{D}(\partial^\bot)$ and $\mathring{u}:=\partial^\bot \mathring{w}$.
\begin{enumerate}
\item[(i)] If $u$ is a solution of (\ref{E:cemod}) in the weak sense, then $w(t):=\mathring{w}+\int_0^t u(s)\star_\omega^{-1} b(s)ds$, $t\in [0,T]$, is a solution of (\ref{E:tte}) in the integral sense. 
\item[(ii)] If $w$ is a solution of (\ref{E:tte}) in the integral sense, then $u:=\partial^\bot w$
is a solution of (\ref{E:cemod}) in the weak sense.
\end{enumerate}
\end{theorem}

\begin{remark}\label{R:localstar2}
Theorem \ref{T:duality} remains true for $\star_\omega$ with minimal energy-dominant $\omega\in \mathcal{H}_{\loc}$ as in Remark \ref{R:localstar}. One can also prove a variant of Theorem \ref{T:duality} involving a boundary $B\subset X$.
\end{remark}

For the special case of stationary velocity fields $b$ this duality is illustrated in the following examples.

\begin{examples}\label{Ex:nobd} Since $\partial=d$ in these examples, we also write $d^\ast$ and $d^\bot$ for $\partial^\ast$ and $\partial^\bot$.
\begin{enumerate}
\item[(i)] Consider the Dirichlet integral $\mathcal{E}$ on $\mathbb{R}$ as in Examples \ref{Ex:indexone} (ii) with $\nu=\mathcal{L}^1$ and $\mathcal{C}=H^1(\mathbb{R})\cap C_c(\mathbb{R})$. Given $g\in L^2(\mathbb{R})$, the form $gdx$ is in $\mathcal{D}(d^\ast)$ if and only if $g\in H^1(\mathbb{R})$, and in this case, $d^\ast(gdx)=-g'$. Defining $d^\bot$ using the operator $\star_{dx}$, we find that $d^\bot g=g'$ for all $g\in H^1(\mathbb{R})$. Given $b=adx$ with $a\in L^2(\mathbb{R})$, the first equation in (\ref{E:cemod}) becomes $\partial_tu=-\partial_x(ua)$, while the first equation in (\ref{E:tte}) becomes $\partial_tw=-a\partial_xw$. If $a\in \mathbb{R}\setminus\{0\}$, then the two equation coincide and describe a linear motion with constant velocity $a$. 
\item[(ii)]  Consider the Dirichlet integral $(\mathcal{E}, H^1(S^1))$ on $S^1$ as in Examples \ref{Ex:indexone} (ii) and with $\nu=\mathcal{H}^1$. Using the duality $\star_{dx}$ we then have $d^\bot g=g'$ for all $g\in \mathcal{D}(\partial^\bot)=H^1(S^1)$, cf. \cite[Example 4.1]{BK16}. Given $b=adx$ with $a\in L^2(S^1)$, the first lines in (\ref{E:cemod}) and (\ref{E:tte}) again give $\partial_tu=-\partial_x(ua)$ and $\partial_tw=-a\partial_xw$, respectively, in the sense of the angular coordinate representation. For $a\in \mathbb{R}\setminus\{0\}$ the equations coincide and describe a circular motion with constant angular velocity $a$.
\end{enumerate}
\end{examples}

\section{Boundaries, dissipative operators and inverse velocities} \label{S:BC}

In the sequel we will prove well-posedness results for certain stationary velocity fields $b$. To do so, we change track and study the continuity equation from a semigroup perspective. We first  concentrate on the derivation $\partial$ under Poincar\'e duality.

Let Assumptions \ref{A:basic} be in force. To allow boundaries $B\subset X$, we assume the following.

\begin{assumption}\label{A:B}
The set $B\subset X$ is closed and an $(\mathcal{E},\mathcal{C})$-null set. 
\end{assumption}

Under Assumption \ref{A:B} we obviously have $\nu(B)=0$ if $\nu$ is minimal energy-dominant for $(\mathcal{E},\mathcal{C})$. In this case the spaces $L^2(X,\nu)$ and $L^2(X\setminus B,\nu)$ agree and $C_c(X\setminus B)$ is dense in $L^2(X,\nu)$. Moreover, 
the space $\mathcal{H}$ does not change if $\mathcal{C}$ and $C_c(X)$ are replaced by  
\[\mathcal{C}_B:=\{f\in\mathcal{C}:\ f|_B\equiv 0\}\]
and $C_c(X\setminus B)$. This follows from the next observation.

\begin{lemma}
Let Assumptions \ref{A:basic} and \ref{A:B} be in force. Then the space of finite linear combinations $\sum_i g_i\partial_if$ with $f_i\in \mathcal{C}\cap C_c(X\setminus B)$ and $g_i\in C_c(X\setminus B)$ is dense in $\mathcal{H}$.
\end{lemma}

\begin{proof}
It suffices to show that any $g\partial f$ with $g\in C_c(X)$ and $f\in \mathcal{C}$ can be approximated. Let $\varepsilon>0$. Let $U$ be a relatively compact open neighborhood of $B\cap (\supp g\cup \supp f)$, small enough to have $\int_{\overline{U}}g^2d\nu_f<\varepsilon/2$. Now choose $g_\varepsilon\in C_c(X\setminus \overline{U})$ such that $\int_{X\setminus\overline{U}}(g-g_\varepsilon)^2d\nu_f<\varepsilon/2$, note that $g_\varepsilon=0$ on $\overline{U}$. Let $\varphi\in \mathcal{C}$ be such that $0\leq \varphi\leq 1$, $\supp \varphi\subset U$ and $\varphi=1$ on a neighborhood of $B\cap (\supp g\cup \supp f)$. Then $f_\varepsilon:=(1-\varphi)f$ is in $\mathcal{C}_B$, and $f-f_\varepsilon=\varphi f=0$ on $X\setminus \overline{U}$. We find that 
\begin{align}
\|g\partial f-g_\varepsilon \partial f_\varepsilon\|_{\mathcal{H}}&\leq\|\mathbf{1}_{\overline{U}}(g\partial f-g_\varepsilon\partial f_\varepsilon)\|_{\mathcal{H}}+\|\mathbf{1}_{X\setminus \overline{U}}((g-g_\varepsilon)\partial f+g_\varepsilon \partial(f-f_\varepsilon))\|_{\mathcal{H}}\notag\\
&\leq \Big(\int_{\overline{U}}g^2d\nu_f\Big)^{1/2}+\Big(\int_{X\setminus\overline{U}}(g-g_\varepsilon)^2d\nu_f\Big)^{1/2}\notag\\
&<\varepsilon.\notag
\end{align}  
Note that for any relatively compact open set $V\supset \supp g_\varepsilon$ we have $\int_{X\setminus\overline{U}}g_\varepsilon^2 d\nu_{f-f_\varepsilon}
\leq \|g_\varepsilon\|_{\sup}^2\nu_{\varphi f}(V\setminus\overline{U})=0$ by \cite[Corollary 3.2.1]{FOT94}.
\end{proof}

Let Assumptions \ref{A:B} be in force. We write $\partial_B:\mathcal{C}_B\to\mathcal{H}$ for the restriction of $\partial$ to $\mathcal{C}_B$ and $\partial^\ast_B:\mathcal{H}\to \mathcal{C}_B^\ast$ for the coderivation in distributional sense, defined as in (\ref{E:dual}), but with $\mathcal{C}_B$ and $\partial_B$ in place of $\mathcal{C}$ and $\partial$. We write $\|\cdot\|_{\mathcal{C}_B^\ast}$ for the natural norm in the dual space $\mathcal{C}_B^\ast$, analogously defined as before. The kernel $\ker \partial_B^\ast$ of $\partial_B^\ast$ is a closed linear subspace of $\mathcal{H}$, it contains $\ker \partial^\ast$. The decomposition $\mathcal{H}=\overline{\partial_B(\mathcal{C}_B)}\oplus \ker \partial_B^\ast$ is orthogonal. If, as in Examples \ref{Ex:nobd}, the boundary is empty, $B=\emptyset$, then we simply recover the original algebra $\mathcal{C}$, the operators $\partial$ and $\partial^\ast$ and the decomposition (\ref{E:Hodge}). 

\begin{examples}\label{Ex:unitint} \mbox{}
\begin{enumerate}
\item[(i)] For the Dirichlet integral on $X=[0,1]$ with $\mathcal{C}=H^1(0,1)$ as in Examples \ref{Ex:indexone} (i) we may consider $B=\{0,1\}$ and $\mathcal{C}_B=\mathring{H}^1(0,1)=\{f\in H^1(0,1):\ f(0)=f(1)=0\}$.
\item[(ii)] For the Dirichlet integral on $S^1$ as in Examples \ref{Ex:indexone} (ii) and \ref{Ex:nobd} (ii) we have $B=\emptyset$, $\mathcal{C}_\emptyset=H^1(S^1)$, $d_\emptyset=d$ and so on.
\end{enumerate}
\end{examples}

To adopt a pragmatic manner of speaking, we refer to the elements of $\ker \partial^\ast$ as \emph{solenoidal fields} on $X$ and to the elements of $\ker \partial_B^\ast$ as \emph{divergence free vector fields (with respect to $B$)}. For $B=\emptyset$ these notions coincide. A function $h\in \mathcal{C}$ is said to be \emph{subharmonic} in $X\setminus B$ if $\mathcal{E}(h,\varphi)\leq 0$ for all $\varphi\in \mathcal{C}_B$ with $\varphi\geq 0$ and \emph{harmonic} in $X\setminus B$ if both $h$ and $-h$ are subharmonic there. We write $\mathbb{H}_B$ for the space of all functions $f\in \mathcal{C}$ harmonic in $X\setminus B$. Then obviously $\ker \partial_B^\ast=\overline{\partial(\mathbb{H}_B)}\oplus \ker \partial^\ast$. A linear functional $F\in \mathcal{C}_B^\ast$ is called \emph{positive} if $F(\varphi)\geq 0$ for all $\varphi\in \mathcal{C}_B$ with $\varphi\geq 0$. In this case it can be represented by a nonnegative Radon measure on $X\setminus B$. 

\begin{lemma}\label{L:diss}
Suppose that Assumptions \ref{A:basic} and \ref{A:B} are in force. Let $b\in\mathcal{H}$ be such that $\partial_B^\ast b$ is positive. Then $\left\langle \partial f,fb\right\rangle_{\mathcal{H}}\geq 0$ for any $f\in \mathcal{C}_B$.
\end{lemma}

\begin{proof}
For any $f\in \mathcal{C}_B$ we have $f^2\in \mathcal{C}_B$, and by the Leibniz rule
\[\left\langle \partial f,fb\right\rangle_{\mathcal{H}}=\frac12\left\langle \partial (f^2),b\right\rangle_\mathcal{H}=\frac12(\partial^\ast_B b)(f^2)\geq 0.\]
\end{proof}

\begin{examples}\label{Ex:subharm}
Suppose that $h\in \mathcal{C}$ is subharmonic on $X\setminus B$. Then $b:=-\partial h$ is positive, note that
$(\partial^\ast_Bb)(\varphi)=\left\langle b,\partial_B\varphi\right\rangle_{\mathcal{H}}=-\mathcal{E}(h,\varphi)\geq 0$
for all $\varphi\in \mathcal{C}_B$ with $\varphi\geq 0$.
\end{examples}

In the sequel we assume that the martingale dimension of $(\mathcal{E},\mathcal{C})$ is one and that a minimal energy-dominant element $b$ of $\mathcal{H}$ is given. We \emph{choose} the volume measure $\nu=\nu_b$, use the duality $\star_b$ and consider $(-\star_b^{-1}\partial_B,\mathcal{C}_B)$, which is a densely defined unbounded linear operator from $L^2(X,\nu_b)$ into $L^2(X, \nu_b)$.

\begin{remark}\label{R:Hinoop}
In the context of a strongly local regular Dirichlet form of martingale dimension one associated with a regular harmonic structure, \cite{Ki93, Ki01}, on a p.c.f. fractal and for minimal energy-dominant $g\in \mathcal{D}(\mathcal{E})$, the operator $f\mapsto \star_{\partial g}^{-1}\partial f$ with domain $\mathcal{D}(\mathcal{E})$ was introduced in \cite[Theorem 5.4]{Hino10} and denoted by $f\mapsto \frac{df}{dg}$. This theorem gives a $\nu_g$-a.e. analog of the classical interpretation of the first derivative as the slope of an affine function approximating $f$. 
\end{remark}

We observe a consequence of Lemma \ref{L:diss}. Recall that a linear operator $(A,\mathcal{D}(A))$ on a real Hilbert space $H$ is said to be \emph{dissipative} if $\left\langle Af,f\right\rangle_H\leq 0$ for all $f\in \mathcal{D}(A)$. 

\begin{corollary}\label{C:diss}
Suppose that Assumptions \ref{A:basic} and \ref{A:B} are in force and the martingale dimension of $(\mathcal{E},\mathcal{C})$ is one. Let $b\in\mathcal{H}$ be minimal energy-dominant and such that $\partial_B^\ast b$ is positive. Then $(-\star_b^{-1}\partial_B,\mathcal{C}_B)$ is dissipative on $L^2(X,\nu_b)$. 
\end{corollary}

\begin{proof}
For any $f\in \mathcal{C}_B$ Proposition \ref{P:Poinc} and  Lemma \ref{L:diss} give
$\left\langle -\star_b^{-1}\partial_Bf, f\right\rangle_{L^2(X,\nu_b)}=-\left\langle \partial f,fb\right\rangle_{\mathcal{H}}\leq 0$.
\end{proof}

\begin{examples}
Assume that $h\in \mathcal{C}$ is minimal energy-dominant and harmonic in $X\setminus B$. 
\begin{enumerate}
\item[(i)] By Examples \ref{Ex:subharm} we can use $b=\partial h$ in Corollary \ref{C:diss}.
\item[(ii)] In the situation of Examples \ref{Ex:qforms} (ii) we can also use $b=-\partial h^2$ in Corollary \ref{C:diss}: Since $h$ is bounded, we have $h^2\in \mathcal{C}$, and since in this situation $\nu_h(\{h= 0\})=0$ by \cite[Chapter I, Theorem 5.2.3]{BH91} and $\nu_{h^2}=4h^2\nu_h$, the function $h^2$ is minimal energy-dominant. Using (\ref{E:defenergymeasure}) and \cite[Chapter I, Theorem 4.2.1]{BH91} with volume measure $\nu_h$ we see that $(\partial_B^\ast b)(\varphi)=-\mathcal{E}(h^2,\varphi)=2\int_X\varphi\:d\nu_{h}\geq 0$ for all $\varphi\in \mathcal{C}_B$.
\end{enumerate}
\end{examples}

A linear operator $(A,\mathcal{D}(A))$ on $H$ is said to be \emph{$m$-dissipative} if it is dissipative and for some $\lambda>0$ the range $(\lambda-A)(\mathcal{D}(A))$ of $\lambda-A$ equals $H$. The Lumer-Phillips theorem states that $(A,\mathcal{D}(A))$ generates a strongly continuous contraction semigroup on $H$ if and only if $A$ is $m$-dissipative, see \cite[Theorem 3.1]{LP61} or \cite[Theorem 3.15 and Corollary 3.20]{EngelNagel00}. A linear operator $(A,\mathcal{D}(A))$ on $H$ is said to be \emph{maximally dissipative} if it has no proper dissipative extension. By Phillips' theorem a linear operator $(A,\mathcal{D}(A))$ on $H$ is $m$-dissipative if and only if it is densely defined and maximally dissipative, see \cite[Theorem 2.3]{Arendtetal23} or \cite[Corollary on p. 201]{Phillips59}.

Zorn's lemma, applied to all dissipative extensions of $(-\star_b^{-1}\partial_B,\mathcal{C}_B)$, partially ordered by domain inclusion, ensures the existence of a maximally dissipative extension $(A,\mathcal{D}(A))$ of $(-\star_b^{-1}\partial_B,\mathcal{C}_B)$. Since $\mathcal{C}_B$ is dense in $L^2(X,\nu_b)$, this extension $(A,\mathcal{D}(A))$ is $m$-dissipative. This gives the following preliminary abstract result.

\begin{theorem}\label{T:Zorn}
Suppose that Assumptions \ref{A:basic} and \ref{A:B} are in force and the martingale dimension of $(\mathcal{E},\mathcal{C})$ is one. Let $b\in\mathcal{H}$ be minimal energy-dominant and such that $\partial_B^\ast b$ is positive.
\begin{enumerate}
\item[(i)] The operator $(-\star_b^{-1}\partial_B,\mathcal{C}_B)$ has an extension $(A,\mathcal{D}(A))$ which generates a strongly continuous contraction semigroup $(T_t)_{t\geq 0}$ on $L^2(X,\nu_b)$.  
\item[(ii)] For any $\mathring{u}\in \mathcal{D}(A)$ the function $u(t)=T_t\mathring{u}$, $t\geq 0$, is in $C^1([0,+\infty),L^2(X,\nu_b))\cap C([0,+\infty),\mathcal{D}(A))$, and it is the unique solution $u:[0,\infty)\to \mathcal{D}(A)$
of the abstract Cauchy problem 
\begin{equation}\label{E:abstractCauchy}
\begin{cases} \frac{d}{dt} u &=Au,\quad t\in (0,+\infty),\\
 u(0)&=\mathring{u}.\end{cases} 
 \end{equation}
\end{enumerate}
\end{theorem}

The first identity in (ii) may be viewed as an abstract formulation of the transport type equation
\begin{equation}\label{E:unknown}
\partial_tu=-\star_b^{-1}\partial_Bu,\quad t\in (0,+\infty).
\end{equation} 

\begin{examples}\label{Ex:S1case}
For the Dirichlet integral on $S^1$ as in Examples \ref{Ex:indexone} (ii) and \ref{Ex:nobd} (ii) we have $B=\emptyset$. Let $a\in\mathbb{R}\setminus \{0\}$ and $b=adx$. Since the duality $\ast_b$ with respect $b$ is used, we now have $-\star_{b}^{-1}df=-a^{-1}f'$, $f\in H^1(S^1)$, as a quick calculation and comparison with Examples \ref{Ex:nobd} shows. The operator $-\star_{b}^{-1}d$ is densely defined on $L^2(S^1, a^2dx)$ and maximally dissipative. It generates a strongly continuous contraction semigroup $(T_t)_{t\geq 0}$ with action $T_tf(x)=f((x-a^{-1}t)\mod \mathbb{Z})$, $x\in S^1$, $t\geq 0$, see for instance \cite[I.4.18]{EngelNagel00}. For any $\mathring{u}\in H^1(S^1)$, the function $u(t)=T_t\mathring{u}$, $t\geq 0$, is the unique solution of (\ref{E:abstractCauchy}) with $A=-\star_{b}^{-1}d$.
\end{examples}

\begin{remark}\label{R:inverse}
The velocity field $b=adx$ in Examples \ref{Ex:S1case} gives the \enquote{inverse velocity} $a^{-1}$. This is due to the use of $\star_b$ together with the weighted volume measure $\nu_b(dx)=a^2dx$ and, by Corollary \ref{C:opscoincide} (i) below, corresponds to the fact that (\ref{E:unknown}) is dual to the continuity equation in (\ref{E:cemod}) under $\nu_b$. Under this measure the divergence of $fb$ is $-a^{-1}f'$. Rewritten as a divergence with respect to the Hausdorff measure $\mathcal{H}^1$ it is $-af'$. Passing to the dual transport equation under $\mathcal{H}^1$ recovers the \enquote{physical velocity} $a$ as in Examples \ref{Ex:nobd} (ii). To align with physics intuition, we will use the term \enquote{physical velocity} in some examples of Section \ref{S:wellpos} with exactly this meaning.
\end{remark}

Since in general no information on the domain $\mathcal{D}(A)$ is available, Theorem \ref{T:Zorn} is of little practical value.

\section{Skew-symmetry and boundary quadruples}\label{S:Skew}

We show that for divergence free $b$ the operator $\star_b^{-1}\partial_B$ is skew-symmetric and that 
$-\partial^\ast_B \star_b$ extends it. We can then use boundary quadruples, \cite{Arendtetal23}, to characterize domains of $m$-dissipative extensions.

Suppose that Assumptions \ref{A:basic} and \ref{A:B} are in force. Given $v\in \mathcal{C}_B^\ast$ and $f\in \mathcal{C}$, the definition 
\begin{equation}\label{E:actiondual}
(fv)(\varphi):=v(f\varphi),\quad \varphi\in \mathcal{C}_B,
\end{equation}
gives an element $fv$ of $\mathcal{C}_B^\ast$; note that $|(fv)(\varphi)|\leq |v(f\varphi)|\leq \|v\|_{\mathcal{C}_B^\ast}\|f\|_{\mathcal{C}}\|\varphi\|_{\mathcal{C}}$ for all $\varphi\in \mathcal{C}_B$.

\begin{lemma}\label{L:dividentity}
Suppose that Assumptions \ref{A:basic} and \ref{A:B} are in force, the martingale dimension of $(\mathcal{E},\mathcal{C})$ is one and $b\in\mathcal{H}$ is minimal energy-dominant. For all $f\in \mathcal{C}$ we have 
\begin{equation}\label{E:dividentity}
\partial_B^\ast(fb)=f\partial_B^\ast b-\star_b^{-1}\partial f,
\end{equation}
seen as an identity in $\mathcal{C}^\ast_B$. 
\end{lemma}

\begin{proof}
Given $b$ and $f$ as stated and $\varphi\in \mathcal{C}_B$, definition (\ref{E:dual}), the Leibniz rule, definition (\ref{E:actiondual}) and (\ref{E:expandv}) give
\[\partial^\ast_B(fb)(\varphi)=\left\langle fb, \partial\varphi\right\rangle_{\mathcal{H}}=\left\langle b, f\partial\varphi\right\rangle_{\mathcal{H}}=\left\langle b,\partial(\varphi f)-\varphi\partial f\right\rangle_{\mathcal{H}}=(f\partial^\ast_B b)(\varphi)-\int_X\varphi(x)\big\langle (\widetilde{\partial f})_x,\widetilde{b}_x\big\rangle_{\mathcal{H}_x}\nu_b(dx);\]
here $(\mathcal{H}_x)_{x\in X}$ is the measurable field of Hilbert spaces determined by (\ref{E:scalarprodHx}) with $\nu=\nu_b$.
With interpretation (\ref{E:interpret}) we arrive at (\ref{E:dividentity}). 
\end{proof}

Suppose that $\nu$ is a minimal energy-dominant measure. We can then view $(\partial,\mathcal{C})$ and $(\partial_B,\mathcal{C}_B)$ as 
densely defined operators from $L^2(X,\nu)$ into $\mathcal{H}$ and consider their adjoints $(\partial^\ast,\mathcal{D}(\partial^\ast))$ and $(\partial^\ast_B,\mathcal{D}(\partial^\ast_B))$, defined as in (\ref{E:partialstar}) but with 
$\mathcal{D}(\mathcal{E})$ replaced by 
$\mathcal{C}$ and $\mathcal{C}_B$, respectively. 

For $v\in \mathcal{H}$ such that $\varphi\mapsto \partial^\ast v(\varphi)$ is bounded on $\mathcal{C}$ with respect to $\|\cdot\|_{L^2(X,\nu)}$ we can find a unique element $g$ of $L^2(X,\nu)$ such that $\left\langle g,\varphi\right\rangle_{L^2(X,\nu)}=\partial^\ast v(\varphi)$ for all $\varphi\in\mathcal{C}$. But this means that $v\in \mathcal{D}(\partial^\ast)$ and $\partial^\ast v=g$. A similar statement is true for $\partial^\ast_B$.

\begin{examples}\label{Ex:unitintdual}\mbox{}
\begin{enumerate} 
\item[(i)] In the case of the Dirichlet integral on $X=[0,1]$ as in Examples \ref{Ex:unitint} (i) with $\nu=\mathcal{L}^1|_{[0,1]}$ we write again $d$ instead of $\partial$. The domain $\mathcal{D}(d_B^\ast)$  consists of all $fdx$ with $f\in H^1(0,1)$, and $d_B^\ast(fdx)=-f'$. The domain $\mathcal{D}(d^\ast)$ consists of all $fdx$ with $f\in \mathring{H}^1(0,1)$, and $d^\ast(fdx)=-f'$.
\item[(ii)] For the Dirichlet integral on $S^1$ as in Examples \ref{Ex:unitint} (ii) with $\nu=\mathcal{H}^1$ we write $d$ for $\partial$. The domain $\mathcal{D}(d^\ast)$ consists of all $fdx$ with $f\in H^1(S^1)$, and $d^\ast(fdx)=-f'$.
\end{enumerate} 
\end{examples}

From now on we assume that the martingale dimension of $(\mathcal{E},\mathcal{C})$ is one, $b\in\mathcal{H}$ is minimal energy-dominant and $\nu=\nu_b$. By $(\partial^\bot,\mathcal{D}(\partial^\bot))$ we denote the closed linear operator on $L^2(X,\nu_b)$ defined as in formulas (\ref{E:botdomain}) and (\ref{E:botpartial}), now based on the adjoint $(\partial^\ast,\mathcal{D}(\partial^\ast))$ of $(\partial, \mathcal{C})$. We write  $(\partial^\bot_B,\mathcal{D}(\partial^\bot_B))$ for the closed linear operator on $L^2(X,\nu_b)$ obtained similarly with $(\partial^\ast_B,\mathcal{D}(\partial^\ast_B))$ in place of $(\partial^\ast,\mathcal{D}(\partial^\ast))$.

\begin{examples}\label{Ex:unitintdomains}
For  $X=[0,1]$ as in Examples \ref{Ex:unitint} (i) and \ref{Ex:unitintdual} (i) the duality $\star_{dx}$ gives $d_B^\bot f =f'$ for all $f$ from $\mathcal{D}(\partial_B^\bot)=H^1(0,1)$, and $d^\bot f=f'$ for $f\in \mathcal{D}(\partial^\bot)=\mathring{H}^1(0,1)$. See  \cite[Example 4.1]{BK16}.
\end{examples}

\begin{corollary}\label{C:opscoincide}
Suppose that Assumption \ref{A:basic} and \ref{A:B} are in force, the martingale dimension of $(\mathcal{E},\mathcal{C})$ is one and $b\in\mathcal{H}$ is minimal energy-dominant. 
\begin{enumerate}
\item[(i)] We have $\mathcal{D}((\star_b^{-1}\partial_B)^\ast)=\mathcal{D}(\partial^\bot_B)$ and 
\begin{equation}\label{E:target}
(-\star_b^{-1}\partial_B)^\ast f=\partial^\bot_B f
\end{equation}
for each element $f$ of this common domain.
\item[(ii)] Suppose in addition that $b$ is divergence free, $b\in \ker \partial_B^\ast$. Then $\mathcal{C}\subset \mathcal{D}(\partial^\bot_B)$, the identity
\begin{equation}\label{E:opscoincide}
\partial^\bot_B f=\star_b^{-1}\partial_Bf,\quad f\in \mathcal{C},
\end{equation}
holds, $\mathbb{R}\subset \ker \partial_B^\bot$, and $(\partial^\bot_B,\mathcal{C}_B)$ is skew-symmetric, 
\begin{equation}\label{E:skew}
\left\langle \partial^\bot_Bf,g\right\rangle_{L^2(X,\nu_b)}=-\left\langle f,\partial^\bot_B g\right\rangle_{L^2(X,\nu_)},\quad  f,g\in \mathcal{C}_B.
\end{equation}
\end{enumerate}
\end{corollary} 

\begin{proof}
If $f\in \mathcal{D}((\star_b^{-1}\partial_B)^\ast)$, then $(\star_b^{-1}\partial_B)^\ast f\in L^2(X,\nu_b)$ and
\[\big\langle (-\star_b^{-1}\partial_B)^\ast f,\varphi\big\rangle_{L^2(X,\nu_b)}=-\big\langle \star_b f,\partial_B\varphi\big\rangle_{\mathcal{H}}\] 
for all $\varphi\in\mathcal{C}_B$, hence $\star_b f\in \mathcal{D}(\partial_B^\ast)$ and (\ref{E:target}) holds. If $f\in \mathcal{D}(\partial^\bot_B)$, then $\partial_B^\ast(\star_b f) \in L^2(X,\nu_b)$ and
\[\big\langle \partial_B^\ast(\star_b f), \varphi\big\rangle_{L^2(X,\nu_b)}=\big\langle f,\star_b^{-1}\partial_B\varphi\big\rangle_{L^2(X,\nu_b)}\] for all  $\varphi\in\mathcal{C}_B$, which means that $f\in \mathcal{D}((\star_b^{-1}\partial_B)^\ast)$ and (\ref{E:target}) holds. This is (i). 

To see (ii), note that if $b\in \ker \partial_B^\ast$, then for any $f\in\mathcal{C}$ identity (\ref{E:dividentity}) reduces to $\partial^\ast_B \star_b f=-\star_b^{-1}\partial_B f$, a priori in $\mathcal{C}^\ast_B$. But since $\star_b^{-1}\partial_B f\in L^2(X,\nu_b)$ the functional $\varphi\mapsto (\partial^\ast_B \star_b f)(\varphi)$ is bounded with respect to $\|\cdot\|_{L^2(X,\nu_b)}$, hence $\star_b f\in \mathcal{D}(\partial_B^\ast)$ and (\ref{E:opscoincide}) holds in $L^2(X,\nu_b)$. The fact that $\partial_B^\bot\mathbf{1}=0$ is immediate. The skew-symmetry (\ref{E:skew}) follows by combining (\ref{E:target}) and (\ref{E:opscoincide}).
\end{proof}

\begin{remark}
In general $\mathcal{C}\subsetneq \mathcal{D}(\partial_B^\bot)$, as will be seen in Examples \ref{Ex:gluedloops} and \ref{Ex:DpartialBbot} below.
\end{remark}

\begin{remark}
A linear operator $(A,\mathcal{D}(A))$ on a real Hilbert space $H$ is skew-symmetric if and only if $\left\langle Af,f\right\rangle_H=0$ for all $f\in \mathcal{D}(A)$, see for instance \cite[Proposition 2.4]{Arendtetal23}.
\end{remark}

\begin{examples}\label{Ex:unitintbot}\mbox{}
\begin{enumerate}
\item[(i)] For $X=[0,1]$ as in Examples \ref{Ex:unitint} (i) and \ref{Ex:unitintdual} (i) the form $dx$ is an element of the closed subspaces $d(H^1(0,1))$ and $\ker d_B^\ast$ of $\mathcal{H}\cong L^2(0,1)$, where $B=\{0,1\}$. In this case $d_B^\bot f=f'=\star_{dx}^{-1}d_Bf$ for all $f\in \mathring{H}^1(0,1)$ and skew-symmetry is just integration by parts.
\item[(ii)] For $S^1$ as in Examples \ref{Ex:nobd} (ii) and \ref{Ex:unitint} (ii) we have $d^\bot f=f'=\star_{dx}^{-1}df$, $f\in H^1(S^1)$, and integration by parts holds without boundary terms.
\end{enumerate}
\end{examples}

Recall that we are interested in solutions of $\partial_t u=\partial_B^\ast(ub)$ with $b\in \mathcal{H}$. The news in comparison with (\ref{E:cemod}) is that we allow a boundary $B$ satisfying Assumptions \ref{A:B}. Given $b\in \mathcal{H}$ minimal energy-dominant, \emph{choosing} $\nu_b$ as a volume measure gives access to the duality $\star_b$. Using this duality, we can then write equation (\ref{E:cemod}) as 
\begin{equation}\label{E:cemodmod}
\partial_tu +\partial_B^\bot u=0,\quad t\in (0,+\infty),
\end{equation}
with $\partial_B^\bot=-\partial^\ast_B \star_b$. For $\omega=b$ this equation also coincides with the one in the first line of (\ref{E:tte}). By identity (\ref{E:opscoincide}) we may say that for $b\in \ker \partial_B^\ast$ the only difference between (\ref{E:unknown}) and (\ref{E:cemodmod}) stems from different choices of operator extensions and therefore, roughly speaking, from different boundary conditions. In the context of continuity equations $b\in \ker \partial_{B}^\ast$ may be viewed as an incompressibility condition.

\begin{examples}\label{Ex:stillgood}\mbox{}
\begin{enumerate}
\item[(i)] For the real line case with $a\in\mathbb{R}\setminus \{0\}$ and velocity field $b=adx\in\mathcal{H}_{\loc}$, equation (\ref{E:cemodmod}) is an implementation of 
\begin{equation}\label{E:reduced}
\partial_t u+a^{-1}\partial_x u=0,\quad t\in (0,+\infty).
\end{equation} 
Recall from Examples \ref{Ex:S1case} that the inversion of the velocity is due to the volume measure $\nu_b$. The related case of the Dirichlet integral on $X=[0,1]$ with $B=\{0,1\}$ as in Examples \ref{Ex:unitintbot} (i) fits into our discussion above, and (\ref{E:cemodmod}) gives again (\ref{E:reduced}); boundary conditions remain to be discussed. For both $X=\mathbb{R}$ and $X=[0,1]$ the space $\ker d^\ast$ is trivial.
\item[(ii)] In the circle case as in Examples \ref{Ex:nobd} (ii) we have $B=\emptyset$ and any nonzero $b\in \ker d^\ast$ is of the form $b=adx$ with constant $a\in\mathbb{R}\setminus \{0\}$. Again (\ref{E:reduced}) agrees with (\ref{E:cemodmod}).
\item[(iii)] For the case $X=\bigcup_{i=1}^N C_i$ of $N\in \mathbb{N}$ circles glued at a point $p$ and the Dirichlet integral $(\mathcal{E},H^1(X))$ with Kirchhoff vertex conditions at $p$ as in Examples \ref{Ex:indexone} (iii) we have $B=\emptyset$. The space $\ker \partial^\ast$ has dimension $N$. Writing $dx_i$ to denote $dx$ for each circle $C_i$ individually, we find that the set $\{dx_1,..., dx_N\}$ is a basis in  $\ker \partial^\ast$. Any linear combination $\sum_{i=1}^N c_i dx_i$ with all $c_i$ nonzero is minimal energy-dominant, its energy measure is $\sum_{i=1}^N c_i^2\mathcal{H}^1|_{C_i}$.
\item[(iv)] For the Dirichlet integral $(\mathcal{E}, H^1(X))$ with Kirchhoff vertex conditions on a finite metric tree $X$ with root $q_0$ and leaves $q_1,...,q_N$ as in Examples \ref{Ex:indexone} (iv) the space $\ker \partial^\ast$ is trivial. However, if $B=\{q_0,...,q_N\}$ and $h\in H^1(X)$ is harmonic on $X\setminus B$ with Dirichlet boundary values $h(q_0)=0$ and $h(q_i)=1$, $i=1,...,N$, then $\partial h\in \ker \partial_B^\ast$ is minimal energy-dominant. More generally, any harmonic function which is not constant on any edge $e\in E$ is minimal energy-dominant. We can identify $\partial h$ with $(h_e'dx_e)_{e\in E}$ with suitable $h_e'\in \mathbb{R}\setminus\{0\}$. The edge-wise constant form $(dx_e)_{e\in E}$ is not in $\ker \partial_B^\ast$ unless all vertices $p\notin B$ have degree two.
\item[(v)] For the Dirichlet integral with Kirchhoff vertex conditions on a connected finite metric graph $X$ the dimension $\ker \partial^\ast$ is the maximal number of independent cycles in $X$, \cite[Proposition 5.1]{IRT12}. If each edge is part of a cycle, then there are minimal energy-dominant $b\in \ker\partial^\ast$. If a boundary $B$ with at least two vertices is specified, then  -- depending on the structure of $X$ -- derivations of suitable harmonic functions may give minimal energy-dominant $\partial h\in \ker\partial_B^\ast$
\item[(vi)] If $X$ is a p.c.f. self-similar, \cite{Ki01}, or, more generally, a finitely ramified fractal, \cite{T08},  
and not a tree, then $\ker \partial^\ast$ is infinite dimensional, see \cite{CGIS13, IRT12}. Typically also a variety of  minimal energy-dominant harmonic functions exists. This makes the scope of equations of form (\ref{E:cemodmod}) wide and a study interesting. The special case of the Sierpi\'nski gasket is discussed in more detail in Examples \ref{Ex:SG} and \ref{Ex:DpartialBbot} below. The infinite dimensionality of $\ker \partial^\ast$ can also be observed for certain non-finitely ramified fractals.
\end{enumerate}
\end{examples}

\begin{remark}\label{R:medmreduction}
In many cases, including finite metric graphs and p.c.f. fractals, the assumption on $b$ to be minimal energy-dominant is not a severe restriction. It can be relaxed if the space $X$ is reduced accordingly. See Examples \ref{Ex:simpleexHpm} (vi) and Examples \ref{Ex:Kmwellpos} (ii).
\end{remark}

We aim at the well-posedness of Cauchy problems for $m$-dissipative extensions of $(-\star_b^{-1}\partial_B,\mathcal{C}_B)$. The domains of such extensions admit a useful description in terms of boundary quadruples, a concept recently introduced in \cite[Definition 3.1]{Arendtetal23}. To recall this definition, suppose that $A_0$ is a densely defined skew-symmetric linear operator on a real Hilbert space $H$ and $\hat{A}:=(-A_0)^\ast$. Then $\hat{A}$ is an extension of $A_0$. A \emph{boundary quadruple} $(H_-,H_+,G_-,G_+)$ for $A_0$ consists of pre-Hilbert spaces $H_-$ and $H_+$, together with linear maps $G_-:\mathcal{D}(\hat{A}) \to H_-$ and $G_+:\mathcal{D}(\hat{A}) \to H_+$ such that
\begin{equation}\label{E:arendt_equation}  
	\big\langle \hat{A} f_1, f_2\big\rangle_{H} + \big\langle f_1, \hat{A} f_2\big\rangle_{H} = \langle G_+f_1, G_+f_2\rangle_{H_+} - \langle G_-f_1, G_-f_2\rangle_{H_-}
\end{equation}
for all $f_1, f_2\in \mathcal{D}(\hat{A})$,
\begin{equation}\label{E:puzzle}
	\ker G_+ + \ker G_- = \mathcal{D}(\hat{A})
\end{equation}
and
\begin{equation}\label{E:individuallysurjective}
	H_- = G_-(\mathcal{D}(\hat{A})), \quad H_+ = G_+(\mathcal{D}(\hat{A})).
\end{equation}

\begin{remark}\label{R:equivcond}\mbox{}
\begin{enumerate}
\item[(i)] By \cite[Lemma 3.2 and Remark 3.3]{Arendtetal23} conditions (\ref{E:puzzle}) and (\ref{E:individuallysurjective}) together are equivalent to the surjectivity of the linear map
$(G_-,G_+):\mathcal{D}(\hat{A})\to H_-\times H_+$.
\item[(ii)] If $(H_-,H_+,G_-,G_+)$ is a boundary quadruple for $A_0$, then the spaces $H_-$ and $H_+$ are Hilbert spaces and $G_-$ and $G_+$ are bounded linear operators, see \cite[Lemmas 3.4 and 3.5]{Arendtetal23}.
\end{enumerate}
\end{remark}

We quote \cite[Theorem 3.10]{Arendtetal23}, which provides a characterization of all $m$-dissipative extensions of $A_0$ as specific restrictions of $\hat{A}$, parametrized by linear contractions between the boundary spaces $H_-$ and $H_+$.

\begin{theorem}\label{T:Arendt}
Let $H$, $A_0$ and $\hat{A}$ be as stated and assume that $(H_-,H_+,G_-,G_+)$ is a boundary quadruple for $A_0$. A linear operator $(A,\mathcal{D}(A))$ on $H$ is an $m$-dissipative extension of $A_0$ if and only if there is a linear contraction $\Theta:H_- \to H_+$ such that
	\begin{equation}\label{E:AXi}
		\mathcal{D}(A) = \{f\in D(\hat{A}) \mid \Theta G_-f = G_+f \}\quad\text{and}\quad Af = \hat{A}f\quad\text{for all $f\in \mathcal{D}(A)$}. 
	\end{equation}
\end{theorem}

We write $(A^\Theta,\mathcal{D}(A^\Theta))$ for the operator in (\ref{E:AXi}). 
\begin{remark}
The map $\Theta\mapsto A^\Theta$ is known to be a bijection between the set of all linear contractions $\Theta:H_- \to H_+$ and the set of all $m$-dissipative extensions of $A_0$, \cite[Corollary 3.12]{Arendtetal23}.
\end{remark}

\begin{examples}\label{Ex:quadrupelinterval}
For $X=[0,1]$ as in Examples \ref{Ex:unitintbot} (i) the choices $H=L^2(0,1)$ and $\mathcal{D}(A_0)=\mathring{H}^1(0,1)$, $A_0f=-f'$, give $\mathcal{D}(\hat{A})=H^1(0,1)$ and  $\hat{A}f=-f'$. Setting $H_+:=\mathbb{R}$, $H_-:=\mathbb{R}$ and $G_+f:=f(0)$, $G_-f:=f(1)$ we obtain a boundary quadruple $(H_-,H_+,G_-,G_+)$ for $A_0$. By Theorem \ref{T:Arendt} a linear operator $(A,\mathcal{D}(A))$ is an $m$-dissipative extension of $A_0$ if and only if $\mathcal{D}(A)=\{f\in H^1(0,1):\ \theta f(1)=f(0)\}$ with some $-1\leq \theta\leq 1$ and $Af=-f'$ for $f\in \mathcal{D}(A)$.
\end{examples}

We aim at a generalization of Examples \ref{Ex:quadrupelinterval} through an application of Theorem \ref{T:Arendt} to $(-\star_b^{-1}\partial_B,\mathcal{C}_B)$ in place of $A_0$; in this case $(\hat{A},\mathcal{D}(\hat{A}))$ is $(-\partial_B^\bot,\mathcal{D}(\partial_B^\bot))$. Boundary quadruples will be provided in Section \ref{S:quadruples}.

\section{Normal parts and non-coincidence of domains}\label{S:domains}

Let Assumptions \ref{A:basic} and \ref{A:B} be satisfied. We collect basic facts on normal parts and orthogonal decompositions used in later sections, and we show that $\mathcal{D}(\partial_B^\bot)$ may be larger than $\mathcal{C}$. 

Suppose that $\nu$ is a finite Borel measure on $X$ with full support. We say that $u\in \mathcal{C}$ is an element of $\mathcal{D}(\Delta_B)$ if there is some $f\in L^2(X,\nu)$ such that 
\[\mathcal{E}(u,\varphi)=-\int_Xf\varphi\:d\nu,\quad \varphi\in \mathcal{C}_B.\]
In this case we write $\Delta_Bu=f$. A function $u\in \mathcal{C}$ is in $\mathcal{D}(\Delta_B)$ if and only if $\partial u\in \mathcal{D}(\partial_B^\ast)$, and in this case, $\Delta_B u=-\partial_B^\ast\partial u$. In the case $B=\emptyset$ we simply write $\Delta$.

For $B\neq \emptyset$ and $v\in \mathcal{D}(\partial_B^\ast)$ we can define an element $n_Bv$ of $\mathcal{C}^\ast$ by an abstract divergence theorem,
\begin{equation}\label{E:normalpart}
n_Bv(\varphi):=\left\langle v,\partial \varphi\right\rangle_\mathcal{H}-\int_X(\partial_B^\ast v) \varphi\:d\nu,\quad \varphi\in \mathcal{C},
\end{equation}
note that 
\[|n_Bv(\varphi)|\leq \|v\|_\mathcal{H}\mathcal{E}(\varphi)^{1/2}+\|\partial_B^\ast v\|_{L^2(X,\nu)}\|\varphi\|_{L^2(X,\nu)}\leq (1+\nu(X)^{1/2})\|v\|_{\mathcal{D}(\partial_B^\ast)}\|\varphi\|_{\mathcal{C}},\quad \varphi\in \mathcal{C};\]
here $\|\cdot\|_{\mathcal{D}(\partial_B^\ast)}$ as in (\ref{E:divnorm}), but with $\partial_B^\ast$ in place of $\partial^\ast$. The linear functional $n_Bv$ is an abstract variant of the normal part of $v$ with respect to $B$. For any $v\in \mathcal{D}(\partial_B^\ast)$ its normal part vanishes on $\mathcal{C}_B$, that is, $\mathcal{C}_B\subset \ker n_Bv$. For $v\in \ker\partial^\ast$ we have $n_Bv\equiv 0$. In the case that $v=\partial u$ with $u\in \mathcal{D}(\Delta_B)$ we write $(du)_B:=n_B\partial u$ for the abstract normal derivative of $u$ at $B$ and see that (\ref{E:normalpart}) gives the Gauss-Green formula
\begin{equation}\label{E:GaussGreen}
(du)_B(\varphi)=\mathcal{E}(u,\varphi)+\int_X(\Delta_B u)\varphi\:d\nu,\quad \varphi\in \mathcal{C}.
\end{equation}
For $B=\emptyset$ we define $n_Bv$ to be the zero map.

Now fix a point $p\in X\setminus B$. We consider the ideal $\mathcal{C}_p:=\mathcal{C}_{\{p\}}$ in $\mathcal{C}$ consisting of all functions zero at $p$. We make use of the following assumption.

\begin{assumption}\label{A:complete}
We have $\mathbf{1}\in \mathcal{C}$, $\ker \partial=\mathbb{R}$ and $(\mathcal{C}_p,\mathcal{E})$ is a Hilbert space. Moreover, there is a constant $c_p>0$ such that $\|f\|_{\sup}\leq c_p\mathcal{E}(f)^{1/2}$, $f\in \mathcal{C}_p$.
\end{assumption}

\begin{remark}\mbox{}
\begin{enumerate} 
\item[(i)] Assumption \ref{A:complete} is satisfied for any resistance form $(\mathcal{E},\mathcal{C})$
in the sense of \cite[Definition 2.8]{Ki03} for which the associated resistance space is compact, cf. Examples \ref{Ex:qforms} (ii). Conversely, the combination of Assumptions \ref{A:basic} and \ref{A:complete} implies that $(\mathcal{E},\mathcal{C})$ is a resistance form on $X$ and all continuous functions on $X$ are also continuous in the resistance topology. We formulate Assumption \ref{A:complete} as stated because no explicit use of the resistance metric is made. 
\item[(ii)] Assumptions \ref{A:basic} and \ref{A:complete} together imply that every minimal energy-dominant measure for $(\mathcal{E},\mathcal{C})$ has full support.
\end{enumerate}
\end{remark}

Now let Assumption \ref{A:complete} be in force. Then $X$ is compact and $(\mathcal{E},\mathcal{C})$ is a strongly local regular Dirichlet form on $L^2(X,\nu)$. 

\begin{remark}
Assumption \ref{A:complete} is made for convenience. Modifying it and carefully adapting later arguments, one can extend our method to non-compact $X$. 
\end{remark}

The space $\partial(\mathcal{C})$ is a closed subspace of $\mathcal{H}$ and the orthogonal decomposition (\ref{E:Hodge}) becomes
\begin{equation}\label{E:Hodgecomplete}
\mathcal{H}=\partial(\mathcal{C})\oplus\ker\partial^\ast.
\end{equation}

\begin{remark}\mbox{}
\begin{enumerate}
\item[(i)] Using $\varphi=\mathbf{1}$ in (\ref{E:normalpart}) we obtain the divergence theorem,
$n_Bv(\mathbf{1})=-\int_X \partial_B^\ast v\:d\nu$, $v\in \mathcal{D}(\partial^\ast_B)$.
\item[(ii)] If $b=\partial h+v$ with $h\in\mathcal{C}$ harmonic in $X\setminus B$ and $v\in \ker\partial^\ast$ is the unique decomposition (\ref{E:Hodgecomplete}) of $b\in \ker \partial^\ast_B$, then we have 
\begin{equation}\label{E:normalpartnormalder}
n_Bb=(dh)_B.
\end{equation}
\end{enumerate}
\end{remark}

Recall that we write $\mathbb{H}_B$ for the space of functions $h\in\mathcal{C}$ which are harmonic in $X\setminus B$. Clearly $\mathbb{R}\subset \mathbb{H}_B$, and for $B=\emptyset$ only constants are harmonic. In connection with (\ref{E:Hodgecomplete}) we note that $\ker \partial_B^\ast=\partial(\mathbb{H}_B)\oplus \ker \partial^\ast$. If $B\neq \emptyset$, then each $f\in \mathcal{C}$ admits a unique decomposition $f=f_B+h$, where $f_B\in \mathcal{C}_B$ and $h\in \mathbb{H}_B$, and this decomposition is orthogonal with respect to $\mathcal{E}$.

\begin{remark}\label{R:finiteB}
Suppose that $B$ is a finite set. In this case normal parts and normal derivatives can be defined as functions on $B$, as shown in \cite[Section 6]{Ki03}: For each $q\in B$ let $\psi_{q}\in\mathcal{C}$ be the unique function harmonic in $X\setminus B$ and having boundary values $\psi_{q}(q)=1$ and $\psi_{q}(q')=0$, $q'\in B\setminus\{q\}$. Then 
\[\varphi=\sum_{q\in B}\varphi(q)\psi_{q},\quad \varphi\in \mathbb{H}_B.\]
Let $\ell(B)$ be the space of real valued functions on $B$. Given $v\in\mathcal{D}(\partial_B^\ast)$, we can set 
\begin{equation}\label{E:pointwisedefnormal}
n_Bv(q):=n_Bv(\psi_{q}),\quad q\in B,
\end{equation}
to define the normal part $n_Bv$ of $v$ as an element of $\ell(B)$. This gives 
\[n_Bv(\varphi)=\sum_{q\in B}\varphi(q)n_Bv(q),\quad \varphi\in \mathcal{C}.\]
In the special case $v=\partial u$ with $u\in \mathcal{D}(\Delta_B)$ we write $(du)_B(q):=n_B\partial u(q)$, $q\in B$, to denote the normal derivative of $u$ at $q$, seen as an element of $\ell(B)$.
\end{remark}

\begin{examples}\label{Ex:SG}
Let $q_0,q_1,q_2$ be the vertices of an equilateral triangle in $\mathbb{R}^2$ and let $F_i:\mathbb{R}^2\to\mathbb{R}^2$ denote the contractive similarities  $F_i(x)=\frac12(x-q_i)+q_i$, $i=0,1,2$. The \emph{Sierpi\'nski gasket} $K$ is the unique compact set $K\subset \mathbb{R}^2$ such that $K=\bigcup_{i=0}^2 F_i(K)$. Let $(\mathcal{E},\mathcal{C})$ be the standard self-similar resistance form on $K$, \cite{Ki01, Str06}, also called \enquote{standard energy form}. It can be expressed as $\mathcal{E}(u)=\lim_{n\to\infty}\mathcal{E}_{V_n}(u)$, $u\in\mathcal{C}$, with 
\[\mathcal{E}_{V_n}(u)=\frac12\Big(\frac{5}{3}\Big)^n\sum_{x\in V_n}\sum_{y\in V_n}(u(x)-u(y))^2,\] 
where $V_0=\{q_0,q_1,q_2\}$ and $V_{n+1}=\bigcup_{i=0}^2 F_i(V_n)$, $n\in \mathbb{N}$. This corresponds to (\ref{E:approxbydiscreteforms}) with $c(n;x,y)\equiv \big(\frac{5}{3}\big)^n$, $n\in\mathbb{N}$. If $K$ is endowed with the associated resistance metric, then Assumptions \ref{A:basic}, \ref{A:B} and \ref{A:complete} are satisfied for $(\mathcal{E},\mathcal{C})$ on $X=K$ with any finite $B\subset K$. Given $n\in \mathbb{N}\setminus \{0\}$ and a word $\alpha=\alpha_1\cdots \alpha_n\in \{0,1,2\}^n$, we denote its length by $|\alpha|=n$. As usual, we write $F_\alpha:=F_{\alpha_n}\circ \dots \circ F_{\alpha_1}$ and $K_\alpha:=F_\alpha(K)$ and set $\mathcal{E}_{K_\alpha}(u):=\big(\frac53\big)^n\mathcal{E}(u\circ F_\alpha)$ to define an energy form $\mathcal{E}_{K_\alpha}$ on $K_\alpha$. For $n=0$ we define $\emptyset$ to be the only possible word with length $|\emptyset|=0$ and $F_\emptyset$ to be the identity map on $K$.

Any non-constant function $h\in \mathcal{C}$ which is harmonic on $X\setminus V_0$ is also minimal energy-dominant for $(\mathcal{E},\mathcal{C})$. This was proved in \cite[Theorem 5.6]{Hino10}.

The space $\ker \partial^\ast$ is the closure in $\mathcal{H}$ of the union over $n\in\mathbb{N}$ of all finite linear  combinations of form $\sum_{|\alpha|=n} \mathbf{1}_{K_\alpha}\partial h_\alpha$ with $h_\alpha\in \mathcal{C}$ harmonic on $K_\alpha\setminus V_n$ and $\sum_{|\alpha|=n}\mathbf{1}_{K_\alpha}(q)(dh_\alpha)_{F_\alpha(V_0)}(q)=0$ at all $q\in V_n$. Here $(dh_\alpha)_{F_\alpha(V_0)}(q)$ denotes the normal derivative at $q$ of $h_\alpha$, defined in the pointwise sense as in Remark \ref{R:finiteB} but with respect to the energy $\mathcal{E}_{K_\alpha}$ on $K_\alpha$ and the boundary $F_\alpha(V_0)$. This was proved in \cite[Theorem 5.6]{IRT12}. A complementary description of $\ker \partial^\ast$ was given in \cite{CGIS13}: Let $h_i\in \mathcal{C}$ be functions, harmonic on $K_i\setminus F_i(V_0)$, $i=0,1,2$, respectively, and such that $h_i|_{V_0}\equiv 0$, $i=0,1,2$, as well as $h_0(F_0q_1)=h_1(F_1q_2)=h_2(F_2q_0)=-\frac16$ and $h_0(F_0q_2)=h_1(F_1q_0)=h_2(F_2q_1)=\frac16$. Then for each $n\in\mathbb{N}$ and each word $\alpha=\alpha_1\cdots \alpha_n$ of length $n$, 
\begin{equation}\label{E:harmonicforms}
d\zeta_\alpha:=\sum_{i=0}^2 \mathbf{1}_{K_{\alpha i}}\partial h_{\alpha i}
\end{equation} is an element of $\ker \partial^\ast$; here $\alpha i:=\alpha_1\cdots \alpha_n i$ and $h_{\alpha i}=h_i\circ F_\alpha^{-1}$ on $K_{\alpha i}$. By \cite[Theorem 2.27]{CGIS13} the set $\{d\zeta_\emptyset\}\cup \{d\zeta_\alpha:\ \alpha\in \{0,1,2\}^n,\ n\in\mathbb{N}\setminus \{0\}\}$ is an orthogonal basis for $\ker \partial^\ast$. By self-similarity and \cite[Theorem 5.6]{Hino10} the form $d\zeta_\emptyset$ is minimal energy-dominant, and for any $n$ the sum $\sum_{|\alpha|=n}a_\alpha d\zeta_\alpha$ is minimal energy-dominant if all $a_\alpha$ are nonzero.

\begin{figure}[H]
	\centering
\begin{tikzpicture}
\node[anchor=south west,inner sep=0] (image) at (0,0) {\includegraphics[scale=0.2]{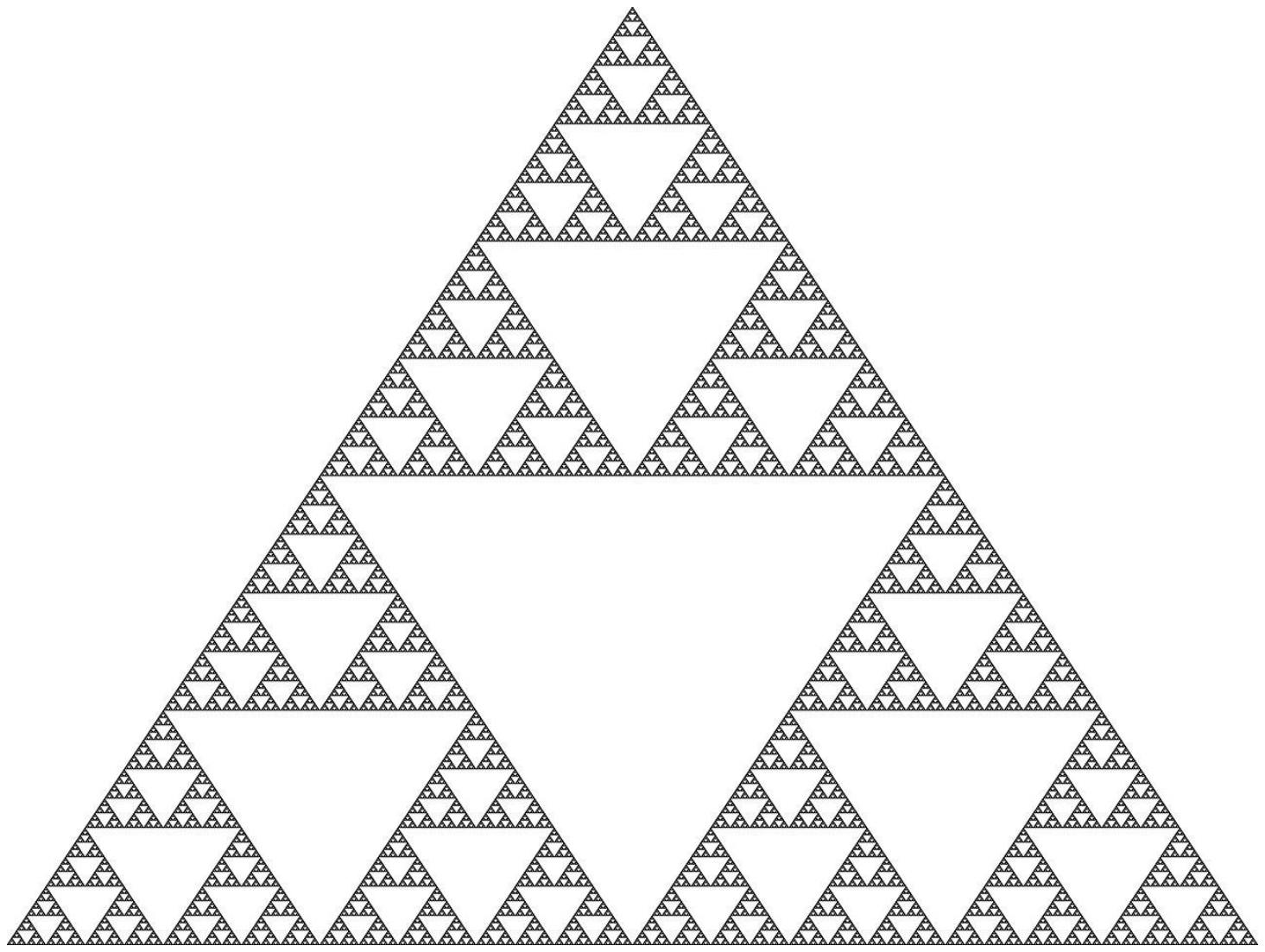}};
\begin{scope}[x={(image.south east)},y={(image.north west)}]
 \node [scale = 1] at (.05,.07) {$q_1$};
 \node [scale = 1] at (0.95,.07) {$q_2$};
  \node [scale = 1] at (0.5,1) {$q_0$};
  \node [scale = 2] at (0.28,.23) {$K_1$};
  \node [scale = 2] at (.72,.23) {$K_2$};
  \node [scale = 2] at (.5,.67) {$K_0$};
\end{scope}
\end{tikzpicture}
\caption{The Sierpi\'nski gasket $K$ and its copies $K_i$.}\label{F:SG}
\end{figure}
\end{examples}

\begin{remark}\label{R:nondegenerate}\mbox{}
\begin{enumerate}
\item[(i)] For standard (that, is level two) Sierpi\'nski gaskets in arbitrary dimension $d\geq 2$ any non-constant harmonic function is minimal energy-dominant, and the same is true for level three Sierpi\'nski gaskets in dimension $d=3$, \cite[p.297]{Hino10}. In \cite{Ts19} it was proved that for Sierpi\'nski gaskets in dimension $d=2$ of arbitrary level $k\geq 2$ any non-constant harmonic function is minimal energy-dominant.
\item[(ii)] Suppose that $(\mathcal{E},\mathcal{C})$ is a self-similar resistance form, obtained from a regular harmonic structure on a connected p.c.f. self-similar fractal $X$ , \cite{Ki01}, with (finite) \enquote{boundary} $V_0$. Following \cite[Definition 4.2]{T08}, we call $(\mathcal{E},\mathcal{C})$ weakly non-degenerate if the space of piecewise harmonic functions is dense in $\mathcal{C}$ with respect to the seminorm $\mathcal{E}^{1/2}$. For instance, the standard energy form on the hexagasket is weakly non-degenerate, but the standard energy form on the Vicsek set is not, see \cite[Section 7]{T00} or \cite[Section 8]{T08}. Consider the Kusuoka type measure $\nu:=\sum_{q\in V_0} \nu_{\psi_q}$, where $\psi_q$ is as in Remark \ref{R:finiteB} with $B=V_0$ and $\nu_{\psi_q}$ is its energy measure. If $(\mathcal{E},\mathcal{C})$ is weakly non-degenerate, then $\nu$ has full support and is minimal energy-dominant, \cite[Proposition 4.3 and Corollary 4.7]{T08}. In this case \cite[Lemma 5.8]{Hino10} guarantees the existence of a minimal energy-dominant harmonic function $h$, hence the duality $\star_{\partial h}$ may be used. 

The six first order copies of the hexagasket, \cite[Example 4.1.2]{Str06}, form a cycle in which any two neighboring first order copies touch at a single point. It might be possible to construct an analog of $d\zeta_\emptyset$ on the hexagasket. See \cite[Example 4.1.2']{Str06}, \cite[Figure 5.1]{Str00} and \cite[Section 7]{OSY02} in this context.
\end{enumerate}
\end{remark}

Suppose that $B\neq 0$. Given $f\in L^2(X,\nu)$ and $g\in \mathcal{C}^\ast$ vanishing on $\mathcal{C}_B$, we call $u\in \mathcal{C}$ a \emph{weak solution} of the Neumann problem
\begin{equation}\label{E:Neumann}
\begin{cases} \Delta_Bu &=f\quad \text{on $X\setminus B$},\\
(du)_B &=g\quad \text{on $B$}, \end{cases}
\end{equation}
if for all $\varphi\in \mathcal{C}$ we have 
\begin{equation}\label{E:weaksolu}
\mathcal{E}(u,\varphi)=g(\varphi)-\int_Xf\varphi\:d\nu.
\end{equation}
It is well-known that there is a weak solution $u$ of (\ref{E:Neumann}) if and only if
\begin{equation}\label{E:compatibility}
\int_X f\:d\nu=g(\mathbf{1}),
\end{equation}
and that in this case $u$ is unique up to an additive constant. By Assumption \ref{A:complete} and (\ref{E:compatibility}) the right-hand side of (\ref{E:weaksolu}) can be estimated in modulus by 
\[|g(\varphi-\varphi(p))|+|\int_Xf(\varphi-\varphi(p))d\nu|\leq \big((1+c_p)\|g\|_{\mathcal{C}^\ast}+c_p\|f\|_{L^2(X,\nu)}\nu(X)^{1/2}\big)\mathcal{E}(\varphi)^{1/2}.\]

Now suppose that $B=\emptyset$. Given $f\in L^2(X,\nu)$, we call $u\in \mathcal{C}$ a \emph{weak solution} of the Poisson equation 
\begin{equation}\label{E:Poisson}
\Delta u=f\quad \text{on $X$} 
\end{equation}
if for all $\varphi\in \mathcal{C}$ we have 
\[\mathcal{E}(u,\varphi)=-\int_X f\varphi\:d\nu.\] 
There is a unique weak solution $u$ of this equation if and only if $\int_X fd\nu=0$, and in this case it is unique up to an additive constant. We regard this situation as the natural \enquote{special case} of (\ref{E:Neumann}) for empty $B$.

In the following we assume in addition that $b\in \ker \partial_B^\ast$ is minimal energy-dominant, the martingale dimension of $(\mathcal{E},\mathcal{C})$ is one and the measure in the above formulas is $\nu=\nu_b$. 

For later use we record the observation that for $b\in \ker \partial^\ast_B$ and $g\in \mathcal{C}$ definition (\ref{E:normalpart}), the Leibniz rule for $\partial$, Poincar\'e duality and (\ref{E:opscoincide}) give
\begin{equation}\label{E:pulloutg}
n_Bb(g\varphi)=\left\langle b,g\partial\varphi\right\rangle_{\mathcal{H}}+\left\langle \varphi,\star_b^{-1}\partial g\right\rangle_{L^2(X,\nu_b)}=\left\langle b,g\partial\varphi\right\rangle_{\mathcal{H}}-\left\langle \varphi,\partial^\ast(gb)\right\rangle_{L^2(X,\nu_b)}=n_B(gb)(\varphi),\quad \varphi\in\mathcal{C}.
\end{equation} 

Another observation is that applying $\star_b$ to (\ref{E:Hodgecomplete}) gives the \enquote{dual} orthogonal decomposition 
\begin{equation}\label{E:HodgeL2}
L^2(X,\nu_b)=\star_b^{-1}\partial(\mathcal{C})\oplus \ker\partial^\bot;
\end{equation}
note that $\ker \partial^\bot$ is the image of $\ker \partial^\ast$ under the isometry $\star_b^{-1}$. In particular, any $f\in L^2(X,\nu_b)$ admits a unique representation of the form 
\begin{equation}\label{E:HodgeL2forf}
f=\star_b^{-1} \partial g+v,
\end{equation}
where $g\in \mathcal{C}$ and $v\in \ker\partial^\bot$. We have $f\in \mathcal{D}(\partial_B^\bot)$ if and only if $g\in \mathcal{D}(\Delta_B)$, and $f\in \ker\partial_B^\bot$ if and only if $g\in\mathbb{H}_B$; recall that $\partial_B^\bot=-\partial_B^\ast \star_b$.

The following examples show that in general the domain $\mathcal{D}(\partial_B^\bot)$ is strictly larger than the domain $\mathcal{C}$ of $\partial$. 

\begin{examples}\label{Ex:gluedloops}
Consider the Dirichlet integral $(\mathcal{E}, H^1(X))$ on the space $X=\bigcup_{i=1}^N C_i$ consisting of $N$ circles $C_i$, glued at a single point $p$ as in Examples \ref{Ex:indexone} (iii) and \ref{Ex:stillgood} (iii). 
\begin{enumerate}
\item[(i)]
If we use the duality $\star_{dx}$ with $dx=\sum_{i=1}^N dx_i\in \ker \partial^\ast$, then $\star_{dx}\mathbf{1}_{C_i}=dx_i$ for all $i$ and 
$\{\mathbf{1}_{C_1},...,\mathbf{1}_{C_N}\}$ is a basis in $\ker\partial^\bot\subset \mathcal{D}(\partial^\bot)$. Clearly $\mathbf{1}_{C_i}\notin H^1(X)$. 
\item[(ii)] Now let $N=2$, let $C_1$ and $C_2$ be two copies of $S^1$, see Figure \ref{F:dual circles}. Let $b=c_1dx_1+c_2dx_2\in\ker\partial^\ast$ with $c_i\in\mathbb{R}\setminus \{0\}$. If orientations are fixed, then the Sobolev space $H^1(C_i\setminus \{p\})$ on each arc $C_i\setminus\{p\}$ can be identified with $H^1(0,1)$, respectively. Given a function $g\in H^1(0,1)$, we write $g(0)$ and $g(1)$ for the values at $0$ and $1$ of its unique continuation to $[0,1]$. Given $f\in L^2(X,\nu_b)$, we use the notation $f_i:=f|_{C_i}$. The domain $\mathcal{D}(\partial^\bot)$ of $\partial^\bot$ based on the duality $\star_b$ is 
\begin{equation}\label{E:racetrack}
\mathcal{D}(\partial^\bot)=\{f=(f_1,f_2)\in H^1(0,1)\oplus H^1(0,1):\ c_1(f_1(1)-f_1(0))+c_2(f_2(1)-f_2(0))=0\};
\end{equation}
the condition on the $f_i$ in (\ref{E:racetrack}) is a \emph{weighted Kirchhoff condition}. For each $f\in \mathcal{D}(\partial^\bot)$ the restriction of $\partial^\bot f\in L^2(X,\nu_b)$ to $C_i\setminus \{p\}$ equals $c_i^{-1}f_i'$, where $f_i'$ denotes the weak derivative of $f_i$ on $C_i\setminus\{p\}$. The condition in (\ref{E:racetrack}) means that $f_1$ and $f_2$ may be discontinuous at $p$, but their (weighted) jumps have to balance to zero.

\begin{figure}[H]
	\centering
	\begin{tikzpicture}
		\draw (0, 0)  circle  (1cm);
		\draw (0, -1) -- node[currarrow]{} (0, -1);
		\draw (0, 1) -- node[currarrow, rotate=180]{} (0, 1);
		\node[right] at (1, 0) {$C_2$};
		\filldraw [black] (-1, 0) circle (1pt);
		\node[left] at (-1, 0) {$p$};
		\draw (-2, 0)  circle  (1cm);
		\draw (-2, -1) -- node[currarrow]{} (-2, -1);
		\draw (-2, 1) -- node[currarrow, rotate=180]{} (-2, 1);
		\node[left] at (-3, 0) {$C_1$};
	\end{tikzpicture}
	\caption{The metric graph of two circles glued together at one point.}\label{F:dual circles}
\end{figure}
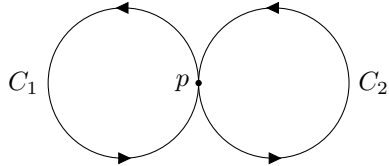
\end{enumerate}  
\end{examples}

\begin{examples}\label{Ex:trees}
Let $(\mathcal{E},H^1(X))$ be the Dirichlet integral on a finite metric tree $X$ with root $q_0$ and a finite number of leaves $q_1,...,q_N$ as specified in Examples \ref{Ex:indexone} (iv) and \ref{Ex:stillgood} (iv). Let $B=\{q_0,...,q_N\}$, let $h\in H^1(X)$ be harmonic on $X\setminus B$ and assume that $h_e'> 0$ for all $e\in E$. If $u\in H^1(X)$ is harmonic in $X\setminus B$, then clearly $\star_{\partial h}^{-1}\partial u\in \mathcal{D}(\partial_B^\bot)$. On the other hand, $\partial u$ can be identified with $(u_e'dx_e)_{e\in E}$ with suitable $u_e\in\mathbb{R}$. Since on each edge $e\in E$ we have $u_e'dx_e=(\star_{\partial h}^{-1}\partial u)|_eh_e'dx_e$, it follows that 
\begin{equation}\label{E:trees}
 \star_{\partial h}^{-1}\partial u=\sum_{e\in E}\frac{u_e'}{h_e'}\mathbf{1}_e.
\end{equation}
Assume now that there is a vertex $p\notin B$ of degree at least three and let $q_i$ be a leaf of the subtree with root $p$. If $e_+$ is the edge going into $p$ and $e_-$ is the edge going out of $p$ and contained in the unique path from $p$ to $q_i$, then $h_{e_+}'>h_{e_-}'$. If $u\in H^1(X)$ is harmonic in $X\setminus B$ with $u(q_0)=0$ and $u(q_i)>0$ and constant except on the unique path from $q_0$ to $q_i$, then $u_{e+}'=u_{e-}'$, consequently $\star_{\partial h}^{-1}\partial u$ is discontinuous at $p$ and not an element of $H^1(X)$. 

In fact, the domain $\mathcal{D}(\partial_B^\bot)$ of $\partial_B^\bot$ is the space of all $f=(f_e)_{e\in E}\in \bigoplus_{e\in E} H^1(e)$ satisfying the \emph{weighted Kirchhoff condition}
\[\sum_{e\in E_+(p)} h_e'f_e(p)-\sum_{e\in E_-(p)}h_e'f_e(p)=0\]
at each vertex $p\notin B$. Here $E_+(p)$ denotes the set of edges with terminal vertex $p$ and $E_-(p)$ the set of edges with initial vertex $p$.
\end{examples}

\begin{examples}\label{Ex:DpartialBbot}
We consider the Sierpi\'nski gasket $K$ as in Figure \ref{F:SG} with the standard energy form $(\mathcal{E},\mathcal{C})$. We use the same notation as in Examples \ref{Ex:SG}. 
\begin{enumerate}
\item[(i)] Let $B=\emptyset$ and let $b:=d\zeta_0+d\zeta_1+d\zeta_2\in \ker \partial^\ast$ with $d\zeta_i$ as in (\ref{E:harmonicforms}). The image of $f=\sum_{i=0}^2 c_i\mathbf{1}_{K_i}$ with $c_i\in\mathbb{R}$ under $\star_b$ is $\star_b f=\sum_{i=0}^2 c_id\zeta_i$, clearly an element of $\ker \partial^\ast$. Consequently $f\in \ker \partial^\bot\subset  \mathcal{D}(\partial^\bot)$, but unless all the $c_i$ are the same,
$f$ is discontinuous and hence not in $\mathcal{C}$.
\item[(ii)] To give an example of a different nature, let $B=V_0$ and let $h\in \mathcal{C}$ be harmonic on $K\setminus V_0$ with boundary values $h(q_0)=1$ and $h(q_1)=h(q_2)=0$ on $V_0$. Clearly $\partial h$ is an element of $\ker \partial_{V_0}^\ast$. By \cite[Theorem 5.6]{Hino10} the function $h$ is minimal energy-dominant for $(\mathcal{E},\mathcal{C})$. In this example we use the duality $\star_{\partial h}$. Let $f_0,f_1,f_2\in\mathcal{C}$ be such that 
\begin{equation}\label{E:fis}
f_0(F_0q_1)=f_1(F_0q_1),\quad f_0(F_0q_2)=f_2(F_0q_2),\quad\text{but}\quad f_1(F_1q_2)\neq f_2(F_1q_2).
\end{equation}
Then $f:=\mathbf{1}_{K_0\setminus \{F_0q_1,F_0q_2\}}f_0+\mathbf{1}_{K_1\setminus \{F_1q_2\}}f_1+\mathbf{1}_{K_2}f_2$ is a bounded Borel function on $K$, but due to its discontinuity at the junction point $F_1q_2$, it cannot be an element of $\mathcal{C}$. Choosing $\nu_h=\nu_{\partial h}$ as the volume measure, we find that $f \in L^2(K,\nu_h)$. Recall that $(dh)_{F_i(V_0)}$ denotes the normal derivative of $h|_{K_i}$ with respect to $\mathcal{E}_{K_i}$. Since $h$ is harmonic, it satisfies 
\begin{equation}\label{E:compat}
(dh)_{F_i(V_0)}(F_0q_i)+(dh)_{F_0(V_0)}(F_0q_i)=0,\quad i=1,2,
\end{equation}
and for this specific $h$ it is easily seen that 
\begin{equation}\label{E:specific}
(dh)_{F_1(V_0)}(F_1q _2)=0\quad\text{and}\quad (dh)_{F_2(V_0)}(F_1q_2)=0.
\end{equation}
For arbitrary $\varphi\in \mathcal{C}_{V_0}$ we now find that
\begin{align}
\big\langle \star_{\partial h} f,\partial \varphi\big\rangle_{\mathcal{H}} &=\sum_{i=0}^2 \big\langle \mathbf{1}_{K_i}f_i\partial h,\partial \varphi\big\rangle_\mathcal{H}\notag\\
&=\sum_{i=0}^2 \big\langle \partial h,\mathbf{1}_{K_i}(\partial(f_i\varphi)-\varphi\partial f_i)\big\rangle_{\mathcal{H}}\notag\\
&=\sum_{i=0}^2 \mathcal{E}_{K_i}(h,\varphi f_i)-\sum_{i=0}^2\big\langle \partial h,\mathbf{1}_{K_i}\varphi\partial f_i\big\rangle_{\mathcal{H}}\notag\\
&=\sum_{i=0}^2\sum_{p\in F_i(V_0)} f_i(p)\varphi(p)(dh)_{F_i(V_0)}(p)-\big\langle \sum_{i=0}^2\mathbf{1}_{K_i}\star_{\partial h}^{-1}\partial f_i, \varphi\big\rangle_{L^2(K,\nu_h)}.\notag
\end{align}
Since $\varphi|_{V_0}\equiv 0$ and by (\ref{E:fis}), (\ref{E:compat}) and (\ref{E:specific}), the first summand on the right-hand side is zero. Consequently $\star_{\partial h}f$ is in $\mathcal{D}(\partial_{V_0}^\ast)$ and
\[\partial_{V_0}^\bot f=\sum_{i=0}^2\mathbf{1}_{K_i}\star_{\partial h}^{-1}\partial f_i=\sum_{i=0}^2\mathbf{1}_{K_i}\frac{d\nu_{h,f_i}}{d\nu_h}.\]
In view of (\ref{E:pulloutg}) one could say that, regardless of the values $f_1(F_1q_2)$ and $f_2(F_1q_2)$, the identities (\ref{E:specific}) force $n_{F_1(V_0)}(\star_{\partial h} f_1)(F_1q_2)=0$ and $n_{F_2(V_0)}(\star_{\partial h} f_2)(F_1q_2)=0$.
\end{enumerate}
\end{examples}




\section{Integration by parts on spaces of martingale dimension one}\label{S:IbP}

Let Assumptions \ref{A:basic}, \ref{A:B} and \ref{A:complete} be satisfied. Assume that $b\in \ker \partial_B^\ast$ is minimal energy-dominant, the martingale dimension of $(\mathcal{E},\mathcal{C})$ is one and $\nu=\nu_b$. 

A functional analytic understanding of how $\mathcal{D}(\partial_B^\bot)$ differs from $\mathcal{C}$ can be gained through the use of a homogeneous Sobolev type space based on $\partial_B^\bot$. Let $\hat{\mathcal{D}}(\partial_B^\bot)$ be the completion of $\mathcal{D}(\partial_B^\bot)/\ker \partial_B^\bot$ with respect the Hilbert seminorm $\|f\|_{\hat{\mathcal{D}}(\partial_B^\bot)}:=\|\partial_B^\bot f\|_{L^2(X,\nu_b)}$. Recall that $\mathcal{C}\subset \mathcal{D}(\partial_B^\bot)$ by Corollary \ref{C:opscoincide}. 

\begin{lemma}\label{L:CandD}
Let Assumptions \ref{A:basic}, \ref{A:B} and \ref{A:complete} be satisfied. Assume that the martingale dimension of $(\mathcal{E},\mathcal{C})$ is one and that $b\in \ker \partial_B^\ast$ is minimal energy-dominant.
\begin{enumerate}
\item[(i)] We have $\mathcal{C}\cap \ker\partial_B^\bot=\mathbb{R}$.
\item[(ii)] The space $\mathcal{C}/\mathbb{R}$ is a closed subspace of $\hat{\mathcal{D}}(\partial_B^\bot)$.
\end{enumerate}
\end{lemma}

\begin{proof}
Since $\mathbb{R}\subset \mathcal{C}$ and $b\in \ker \partial_B^\ast$ we have $\mathbb{R}\subset \mathcal{C}\cap \ker\partial_B^\bot$. For $f\in \mathcal{C}\cap \ker\partial_B^\bot$ identity (\ref{E:opscoincide}) gives $\star_b^{-1}\partial f=0$; this is equivalent to $\partial f=0$, which by Assumption \ref{A:complete} implies that $f$ is constant. This is (i). To see (ii), note that $(f_n)_n\subset \mathcal{C}/\mathbb{R}$  is Cauchy in $\|\cdot\|_{\hat{\mathcal{D}}(\partial_B^\bot)}$ if and only if the functions $\partial_B^\bot f_n=\star_b^{-1}\partial f_n$ form a Cauchy sequence in $L^2(X,\nu_b)$, and this happens if and only if $(\partial f_n)_n$ is Cauchy in $\mathcal{H}$. The closedness of the space $\partial(\mathcal{C})$ implies its completeness, hence there is some $f\in \mathcal{C}/\mathbb{R}$ such that 
\[0=\lim_{n\to \infty}\|\partial f_n-\partial f\|_{\mathcal{H}}=\lim_{n\to\infty}\|\star_b^{-1}\partial f_n-\star_b^{-1}\partial f\|_{L^2(X,\nu_b)}=\lim_{n\to\infty}\|\partial_B^\bot f_n-\partial_B^\bot f\|_{L^2(X,\nu_b)}=\lim_{n\to \infty}\|f_n-f\|_{\hat{\mathcal{D}}(\partial_B^\bot)}.\]
\end{proof}

\begin{remark}
By Assumption \ref{A:complete} the operator $(\partial,\mathcal{C})$ is closed on $L^2(X,\nu_b)$. This can be used to show that $\mathcal{C}$ is a closed subspace of the Hilbert space $(\mathcal{D}(\partial_B^\bot), \|\cdot\|_{\mathcal{D}(\partial_B^\bot)})$.
\end{remark}

Suppose for a moment that $B=\emptyset$; in this case $b\in \ker \partial^\ast$. Let $f\in \mathcal{D}(\partial^\bot)$. Viewing $f$ as its class in $\hat{\mathcal{D}}(\partial^\bot)$, we can choose a representative $g\in \mathcal{C}$ modulo constants of the orthogonal projection of $f$ onto $\mathcal{C}/\mathbb{R}$. Let 
\begin{equation}\label{E:fminusg}
f-g=\star_b^{-1}\partial u+w
\end{equation}
with $u\in \mathcal{D}(\Delta)$ and $w\in \ker \partial^\bot$ be the unique decomposition (\ref{E:HodgeL2}) of $f-g\in \mathcal{D}(\partial^\bot)$. An application of $\partial^\bot$ to both sides of (\ref{E:fminusg}) gives
\begin{equation}\label{E:afterappl}
\partial^\bot f-\star_b^{-1}\partial g=\partial^\bot(f-g)=\partial^\bot(\star_b^{-1}\partial u)=\Delta u.
\end{equation}
This means that $f\in \mathcal{D}(\partial^\bot)$ differs from $g\in \mathcal{C}$ by the sum of an element of $\ker \partial^\bot$ and the derivative $\star^{-1}_b\partial u$ of a weak solution of the Poisson equation (\ref{E:afterappl}) in the sense of (\ref{E:Poisson}); note that by (\ref{E:GaussGreen}) we have
\begin{equation}\label{E:intzero} 
\int_X(\partial^\bot f-\star_b^{-1}\partial g) d\nu_b=0.
\end{equation} 
Moreover, orthogonality in $\hat{\mathcal{D}}(\partial^\bot)$ gives
\[\left\langle \Delta u,\star_b^{-1}\partial \varphi\right\rangle_{L^2(X,\nu_b)}=\left\langle f-g,\varphi\right\rangle_{\hat{\mathcal{D}}(\partial^\bot)}=0\]
for all $\varphi\in\mathcal{C}/\mathbb{R}$, consequently $\Delta u\in \ker\partial^\bot$ and (\ref{E:afterappl}) is just the orthogonal decomposition (\ref{E:HodgeL2}) for $\partial^\bot f$.

In the sequel we allow the case $B\neq \emptyset$. Then boundary terms appear and instead of being the weak solution to a Poisson equation, the function $u$ is a weak solution to a Neumann problem. We formulate a corresponding result in Theorem \ref{T:keydecomp} below. Using this result, we then obtain an abstract integration by parts formula for $(\partial_B^\bot,\mathcal{D}(\partial_B^\bot))$, Theorem \ref{T:IbP}.  

We prepare these theorems by some preliminary observations. On $\mathcal{D}(\partial_B^\bot)$ the operator $\partial_B^\bot$ is no longer skew-symmetric, and for $B\neq \emptyset$ skew-symmetry is lost even on $\mathcal{C}$. The failure to be skew-symmetric is balanced in boundary terms involving normal parts.

\begin{proposition}
Let Assumptions \ref{A:basic}, \ref{A:B} and \ref{A:complete} be in force. Assume that the martingale dimension of $(\mathcal{E},\mathcal{C})$ is one and that $b\in \ker \partial_B^\ast$ is minimal energy-dominant. 
\begin{enumerate}
\item[(i)] For any $g\in \mathcal{C}$ we have 
\begin{equation}\label{E:IbPC}
\big\langle \star_b^{-1}\partial g,\varphi\big\rangle_{L^2(X,\nu_b)}+\big\langle g,\star_b^{-1}\partial\varphi\big\rangle_{L^2(X,\nu_b)}=n_Bb(g\varphi),\quad \varphi\in \mathcal{C}.
\end{equation}
\item[(ii)] For any $f\in \mathcal{D}(\partial_B^\bot)$ we have 
\begin{equation}\label{E:IbPD}
\big\langle \partial_B^\bot f,\varphi\big\rangle_{L^2(X,\nu_b)}+\big\langle f,\star_b^{-1}\partial\varphi\big\rangle_{L^2(X,\nu_b)}=n_B(fb)(\varphi),\quad \varphi\in \mathcal{C}.
\end{equation}
\end{enumerate}
\end{proposition}

\begin{proof}
Item (i) follows using (\ref{E:normalpart}) and the Leibniz rule,
\begin{multline}
\left\langle \star_b^{-1}\partial g,\varphi\right\rangle_{L^2(X,\nu_b)}-n_Bb(g\varphi)=\left\langle \varphi\partial g,b\right\rangle_\mathcal{H}-\left\langle \partial(g\varphi),b\right\rangle_\mathcal{H}\notag\\
=-\left\langle g\partial \varphi,b\right\rangle_\mathcal{H}=-\left\langle \partial \varphi,gb\right\rangle_\mathcal{H}=-\left\langle g,\star_b^{-1}\partial \varphi\right\rangle_{L^2(X,\nu_b)}.
\end{multline}
Item (ii) follows by an application of duality to the first summand on the right hand side of (\ref{E:normalpart}),
\[n_B(fb)(\varphi)=\left\langle fb,\partial \varphi\right\rangle_\mathcal{H}-\left\langle \partial_B^\ast (fb),\varphi\right\rangle_{L^2(X,\nu_b)}=\left\langle f,\star_b\partial\varphi\right\rangle_{L^2(X,\nu_b)}+\left\langle \partial^\bot f,\varphi\right\rangle_{L^2(X,\nu_b)}.\] 
\end{proof}

\begin{examples}\label{Ex:domainsunitinterval} For the unit interval $X=[0,1]$ with $B=\{0,1\}$ and duality $\star_{dx}$ as in Examples \ref{Ex:unitintdomains} and \ref{Ex:unitintbot} we find that $n_B(fdx)(\varphi)=f(1)\varphi(1)-f(0)\varphi(0)$, for any $f,\varphi\in H^1(0,1)$. Moreover, (\ref{E:IbPC}) and (\ref{E:IbPD}) coincide.
\end{examples}

Identity (\ref{E:arendt_equation}) and Examples \ref{Ex:quadrupelinterval} and \ref{Ex:domainsunitinterval} hint that integration by parts formulae should be helpful in order to construct boundary quadruples. 

\begin{remark} Another thought that motivates the consideration of boundary values is the following. For any $f\in L^2(X,\nu_b)$ with $f\not\equiv\mathbf{1}$ the field $\star_b f=fb$ differs from $b$, and from the point of view of (\ref{E:Hodgecomplete}) the modulation of $b$ by the coefficient $f$ in $fb$ is difficult to understand. If $f\in \mathcal{D}(\partial_B^\bot)$, then the normal part $n_B(fb)$ of $fb$ at the boundary $B$ is an accessible piece of information about $fb$, and one can try to compare it to $n_Bb$.
\end{remark}

The next result provides a variant of (\ref{E:fminusg}) for which $B$ may be nonempty. We agree to understand the case of empty $B$ as included in the sense of the Poisson problem (\ref{E:Poisson}) and the above discussion.

\begin{theorem}\label{T:keydecomp}
Let Assumptions \ref{A:basic}, \ref{A:B} and \ref{A:complete} be in force. Assume that the martingale dimension of $(\mathcal{E},\mathcal{C})$ is one and that $b\in \ker\partial_B^\ast$ is minimal energy-dominant.
\begin{enumerate}
\item[(i)] A function $f\in L^2(X,\nu_b)$ is an element of $\mathcal{D}(\partial_B^\bot)$ if and only if it admits a representation
\begin{equation}\label{E:keydecomp}
f=g+\star_b^{-1}\partial u+w,
\end{equation}
where $g\in\mathcal{C}$, $u\in \mathcal{D}(\Delta_B)$ is a weak solution of the Neumann problem 
\begin{equation}\label{E:keyNeumann}
\begin{cases} \Delta_B u &=\partial_B^\bot f-\star_b^{-1}\partial g\quad \text{on $X\setminus B$},\\
(du)_B &=n_B(fb)-n_Bb(g\:\cdot)\quad \text{on $B$}, \end{cases}
\end{equation}
and $w\in \ker\partial^\bot$. 
\item[(ii)] If $f=\widetilde{g}+\star_b^{-1}\partial \widetilde{u}+\widetilde{w}$ is another such representation of $f\in \mathcal{D}(\partial_B^\bot)$, then there is a single constant $c\in\mathbb{R}$ such that 
\[\widetilde{g}=g+c,\quad \star_b^{-1}\partial\widetilde{u}=\star_b^{-1}\partial u-c\star_b^{-1}\partial h\quad\text{and}\quad \widetilde{w}=w-c\star_b^{-1} v,\] 
where $b=\partial h+v$ is the unique decomposition (\ref{E:Hodgecomplete}) of $b$.
\item[(iii)] For any $f \in \mathcal{D}(\partial_B^\bot)$ the class modulo constants of $g$ is the orthogonal projection in $\hat{\mathcal{D}}(\partial_B^\bot)$ onto $\mathcal{C}/\mathbb{R}$ of the class of $f$ modulo $\ker\partial_B^\bot$.
\end{enumerate}
\end{theorem}

Theorem \ref{T:keydecomp} says that what may prevent an element of $\mathcal{D}(\partial_B^\bot)$ from being in $\mathcal{C}$ is the sum of an element of $\ker \partial^\bot$ and the derivative $\partial_B^\bot u$ of a weak solution $u$ of the Neumann problem (\ref{E:keyNeumann}).

\begin{remark}\mbox{}
\begin{enumerate}
\item[(i)] Suppose that $f\in \mathcal{D}(\partial_B^\bot)$ and $f=g+\star_b^{-1}\partial u+w$ is a representation for $f$ as in (\ref{E:keydecomp}). Then 
$\star_b^{-1}\partial u+w$ is the unique decomposition of $f-g\in \mathcal{D}(\partial_B^\bot)$ as in (\ref{E:HodgeL2}). 
\item[(ii)] If $b\in \ker \partial^\ast$, then $\star_b^{-1}\partial\widetilde{u}=\star_b^{-1}\partial u$ and $\widetilde{w}=w-c$ in the uniqueness statement.
If $b=\partial h$ with $h\in\mathcal{C}$ harmonic in $X\setminus B$, then $\star_b^{-1}\partial\widetilde{u}=\star_b^{-1}\partial u-c$ and $\widetilde{w}=w$ in the uniqueness statement. 
\end{enumerate}
\end{remark}


We use a variant of (\ref{E:afterappl}) as the starting point for a proof of Theorem \ref{T:keydecomp}.

\begin{proof}
It is clear that each $f$ with representation (\ref{E:keydecomp}) is an element of $\mathcal{D}(\partial_B^\bot)$. To verify the converse, suppose that $f\in \mathcal{D}(\partial_B^\bot)$. An application of (\ref{E:HodgeL2}) to $\partial_B^\bot f$ gives
\begin{equation}\label{E:partialbotfdecomp}
\partial_B^\bot f=\star_b^{-1} \partial g+z
\end{equation}
with uniquely determined $g\in\mathcal{C}$ and $z\in \ker \partial^\bot$. By (\ref{E:IbPC}) and (\ref{E:IbPD}) with $\varphi= \mathbf{1}$, we find that 
\[\int_X z\:d\nu_b=\left\langle \partial_B^\bot f-\star_b^{-1}\partial g,\mathbf{1}\right\rangle_{L^2(X,\nu_b)}=n_b(fb)(\mathbf{1})-n_Bb(g).\]
Consequently there is a weak solution $u\in \mathcal{D}(\Delta_B)$ of (\ref{E:keyNeumann}), unique up to an additive constant. Now set $w:=f-g-\star_b^{-1}\partial u$ and observe that (\ref{E:IbPC}), (\ref{E:IbPD}), (\ref{E:keyNeumann}) and (\ref{E:GaussGreen})  give 
\begin{align}
\left\langle w,\star_b^{-1}\partial\varphi\right\rangle_{L^2(X,\nu_b)}&=
n_B(fb)(\varphi)-n_Bb(g\varphi) -\left\langle \partial^\bot_Bf-\star_b^{-1}\partial g,\varphi\right\rangle_{L^2(X,\nu_b)}-\left\langle \star_b^{-1}\partial u,\star_b^{-1}\partial\varphi\right\rangle_{L^2(X,\nu_b)}\notag\\
&=(du)_B(\varphi)-\left\langle \Delta_Bu,\varphi\right\rangle_{L^2(X,\nu_b)}-\mathcal{E}(u,\varphi)\notag\\
&=0\notag
\end{align}
for all $\varphi\in \mathcal{C}$. This shows (i).

If $f=\widetilde{g}+\star_b^{-1}\partial \widetilde{u}+\widetilde{w}$ is another representation of $f$ as in (\ref{E:keydecomp}), then the uniqueness in (\ref{E:partialbotfdecomp}) forces $\star_b^{-1} \partial\widetilde{g}=\star_b^{-1}\partial g$, which implies that $\widetilde{g}=g+c$ with some $c\in\mathbb{R}$. By (\ref{E:keyNeumann}) we then find that $u-\widetilde{u}$ is harmonic in $X\setminus B$ and satisfies $(d(u-\widetilde{u}))_B=c(dh)_B$, which means that $u-\widetilde{u}$ equals $ch$ up to an additive constant, hence $\star_b^{-1}\partial u-\star_b^{-1}\partial\widetilde{u}=c\star_b^{-1}\partial h$. Since $\star_b^{-1} b=1$, it follows that $\widetilde{w}-w=-c\star_b^{-1} v$. This gives (ii).

To see (iii), note that for any $\varphi\in\mathcal{C}/\mathbb{R}$, we have 
\[\left\langle f-g,\varphi\right\rangle_{\hat{\mathcal{D}}(\partial_B^\bot)}=\left\langle \star_b^{-1}\partial u+w,\varphi\right\rangle_{\hat{\mathcal{D}}(\partial_B^\bot)}=\left\langle \Delta_B u+\partial_B^\bot w,\partial_B^\bot \varphi\right\rangle_{L^2(X,\nu_b)}=\left\langle z,\star_b^{-1}\partial \varphi\right\rangle_{L^2(X,\nu_b)}=0\] 
by (\ref{E:keydecomp}), (\ref{E:keyNeumann}), (\ref{E:opscoincide}) and the fact that $w,z\in \ker\partial^\bot$.
\end{proof}

The following examples show that for the interval and the circle the classical formulas are recovered and that for glued circles and finite metric trees the summands in (\ref{E:keydecomp}) and (\ref{E:partialbotfdecomp}) admit meaningful interpretations.

\begin{examples}\label{Ex:keydecomp}\mbox{}
\begin{enumerate}
\item[(i)] On $[0,1]$ with $b=dx$ and $B=\{0,1\}$ the representation (\ref{E:keydecomp}) reduces to $f=g+u'$ with $g\in H^1(0,1)$ and $u'\in\mathbb{R}$. Recall that $\ker \partial^\bot=\{0\}$ and each harmonic function $u$ is an affine function with constant slope $\star_{dx}du=u'$. This is in line with Examples \ref{Ex:unitint} (i) and \ref{Ex:unitintdomains}.
\item[(ii)] For the Dirichlet integral on $S^1$ with $b=dx$ the representation (\ref{E:keydecomp}) reduces to $f=g+c$ with $g\in H^1(S^1)$ and a constant $c\in\mathbb{R}$. This is in line with the observation  that $\mathcal{C}=H^1(S^1)=\mathcal{D}(\partial^\bot)$, and it is due to the fact that $\ker \partial^\bot= \mathbb{R}$, consequently $\Delta u=0$ by (\ref{E:intzero}) and therefore $u\in \mathbb{R}$.
\item[(iii)] Suppose that $(\mathcal{E},H^1(X))$ is the Dirichlet integral with Kirchhoff vertex conditions on the space $X=\bigcup_{j=1}^2C_j$ consisting of two copies $C_j$ of $S^1$, glued at a point $p$, as in Examples \ref{Ex:gluedloops} (ii). As there we use $\star_b$ with $b=c_1dx_1+c_2dx_2$, where $c_i\in\mathbb{R}\setminus\{0\}$. For $f\in \mathcal{D}(\partial^\bot)$ we have $\partial^\bot f|_{C_j\setminus \{p\}}=c_j^{-1}f_j'$, and for $u\in \mathcal{D}(\Delta)$ it follows that $\Delta u|_{C_j\setminus \{p\}}=c_j^{-2}u_j''$; here $f_j'$ is the weak derivative of $f_j:=f|_{C_j}$ and $u_j''$ the weak second derivative of $u|_{C_j}$ on $C_j\setminus \{p\}$, respectively.

Given $f\in\mathcal{D}(\partial^\bot)$ there are a unique function $g\in H^1(X)/\mathbb{R}$ and unique constants $z_j\in\mathbb{R}$ such that (\ref{E:partialbotfdecomp}) holds with $z=z_1\mathbf{1}_{C_1}+z_2\mathbf{1}_{C_2} \in \ker \partial^\bot$. This implies that $z_j=c_j^{-1}(f_j'-g_j')$ on $C_j\setminus\{p\}$, and since $\int_{C_j} g_j'dx=0$, it follows that 
\begin{equation}\label{E:circleszi}
z_j=c_j^{-1}(f_j(1)-f_j(0)),\quad j=1,2.
\end{equation}

The representation (\ref{E:keydecomp}) becomes 
\[f=g+\sum_{j=1}^2 c_j^{-1}u_j'\mathbf{1}_{C_j}+\sum_{j=1}^2 w_i\mathbf{1}_{C_j}\]
with $u$ being the unique weak solution of $\Delta u=z$ and with suitable constants $w_j\in\mathbb{R}$. Since $c_j^{-2}u_j''=z_j$, the function $u_j$ must be quadratic on $[0,1]$ and $u_j'$ must be linear with slope $u_j'(1)-u_j'(0)=u_j''$. Since 
$u_j(0)=u_j(1)$ \emph{by the continuity of $u$}, this gives $u_j'(0)=-\frac12 u_j''$ and therefore 
\begin{equation}\label{E:circleswi}
w_j=\frac12(f_j(0)+f_j(1))-g(p), \quad j=1,2.
\end{equation}

The constants $z_j$ describe \emph{the jumps (at $p$) of the components $f_j$ on the individual cycles $C_j$}. From (\ref{E:circleszi}) and the \emph{weighted Kirchhoff conditions} in (\ref{E:racetrack}) we recover (\ref{E:intzero}), now reduced to $c_1^2z_1+c_2^2z_2=0$. Up to the value of $g$ at $p$, which by a shift as in Theorem \ref{T:keydecomp} (ii) can be made zero, the constants $w_j$ are the \emph{average values at $p$ of the components $f_j$ on the individual cycles $C_j$}.

\item[(iv)] Let $(\mathcal{E},H^1(X))$ be the Dirichlet integral with Kirchhoff vertex conditions on a finite metric tree $X$ with root $q_0$ and leaves $q_1,...,q_N$. Let $B=\{q_0,...,q_N\}$ and choose $b=\partial h$ with $h\in \mathbb{H}_B$ satisfying $h_e'>0$, $e\in E$, as in 
Examples \ref{Ex:trees}. With $\star_{\partial h}$ the representation (\ref{E:keydecomp}) becomes 
\begin{equation}\label{E:keydecomptree}
f=g+\sum_{e\in E}\frac{u_e'}{h_e'}\mathbf{1}_e,
\end{equation}
where $g\in H^1(X)$ and $u\in H^1(X)$ is harmonic on $X\setminus B$ with Neumann boundary conditions $(du)_B=n_B(f\partial h)-g(dh)_B$ on $B$. Here we use the same notation as in (\ref{E:trees}). Theorem \ref{T:keydecomp} shows that functions of form (\ref{E:trees}) are not only elements of $\mathcal{D}(\partial_B^\bot)$, but the only possible  differences between elements of $\mathcal{D}(\partial_B^\bot)$ and elements of $\mathcal{C}_B$.
\end{enumerate}
\end{examples}

For more complicated examples, such as fractal spaces, the available information on the summands $\star_b\partial u$ and $w$ and $z$ in (\ref{E:keydecomp}) and (\ref{E:partialbotfdecomp}) is less explicit. We first discuss the case of a Sierpi\'nski gasket graph, which still permits explicit calculations. These calculations are described here and carried out in detail in Appendix \ref{S:repSgraph}.

\begin{examples}\label{Ex:repSgraph}
Let $q_0,q_1,q_2$ be the vertices of an equilateral triangle in $\mathbb{R}^2$. Let $p_i$ denote the midpoint dividing the edge opposite to $q_i$, $i=0,1,2$, into two, and connect these new vertices $q_i$ by three additional edges. The resulting metric graph with edge set $E$ is the first level graph approximation to the Sierpi\'nski gasket, see Figure \ref{F:gasket graph}. To be consistent with later notation, we denote this metric graph by $K^{(1)}$. We consider the Dirichlet integral with Kirchhoff vertex conditions $(\mathcal{E},H^1(K^{(1)}))$ as specified in (\ref{E:DirichletdomainKirchhoff}) and (\ref{E:Dirichletintergralmg}). We write $K_i^{(1)}$ for the subtriangle graph containing $q_i$ and $E_i$ for the set of its edges. To fix an orientation, let the triangle graphs $q_0q_1q_2$ and $p_0p_1p_2$ be oriented counterclockwise. We write $pq$ for the edge $e\in E$ from vertex $p$ to vertex $q$ in this orientation, recall that this means we identify $e$ with $[0,1]$, $p$ with $0$ and $q$ with $1$. 

Let $B=\emptyset$ and let $b=(b_edx_e)_{e\in E}\in \ker \partial^\ast$ be defined by 
\[b_{q_0p_2}=b_{p_2q_1}=b_{q_1p_0}=b_{p_0q_2}=b_{q_2p_1}=b_{p_1q_0}=1\quad \text{and}\quad b_{p_0p_1}=b_{p_1p_2}=b_{p_2p_0}=2,\]
as depicted in Figure \ref{F:gasket graph}. We consider the duality $\star_b$ and operators $\partial^\bot$ with respect to the field $b$. Up to the different orientation and scaling, it corresponds to $d\zeta_\emptyset$ in Examples \ref{Ex:SG}. 
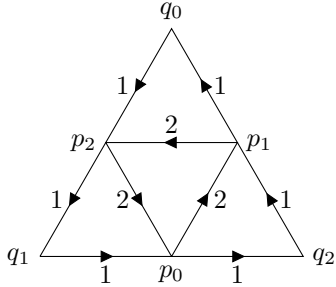
\begin{figure}[H]
	\centering	
	\begin{tikzpicture}[scale=3.5]
		\draw (0, 0) -- node[pos=0.25, currarrow, sloped, allow upside down]{} node[pos=0.25, below]{1} node[pos= 0.75, currarrow, sloped, allow upside down]{} (1, 0) node[pos=0.75, below]{1} -- node[pos=0.25, currarrow, sloped, allow upside down]{} node[pos=0.25, right]{1} node[pos= 0.75, currarrow, sloped, allow upside down]{} node[pos=0.75, right]{1} (1/2,0.866) -- node[pos=0.25, currarrow, sloped, allow upside down]{} node[pos=0.25, left]{1} node[pos= 0.75, currarrow, sloped, allow upside down]{} node[pos=0.75, left]{1} (0, 0);
		\draw (1/2, 0) -- node[currarrow, sloped, allow upside down]{} node[right]{2}  (3/4, 0.433) -- node[currarrow, sloped, allow upside down]{} node[above]{2}  (1/4, 0.433) -- node[currarrow, sloped, allow upside down]{} node[left]{2}   (1/2, 0);
		\node[above] at (1/2, 0.866) {$q_0$};
		\node[left] at (0, 0) {$q_1$};
		\node[right] at (1, 0) {$q_2$};
		\node[below] at (1/2, 0) {$p_0$};
		\node[right] at (3/4, 0.433) {$p_1$};
		\node[left] at (1/4, 0.433) {$p_2$};
	\end{tikzpicture}
	\caption{The first level $K^{(1)}$ of the metric graph approximation to the Sierpi\'nski gasket, and the values of $b$.}
	\label{F:gasket graph}
\end{figure} 
We write 
\[b_0:=dx_{q_0p_2}-dx_{p_1p_2}+dx_{p_1q_0}\]
and define $b_1$ and $b_2$ analogously. Each $b_i$ is supported on the edges of $K^{(1)}_i$. The set $\{b,b_0,b_1,b_2\}$ is an orthogonal basis for $\ker \partial^\ast$. It follows that an orthogonal basis for $\ker \partial^\bot$ is given by  $\{\mathbf{1},\varphi_0,\varphi_1,\varphi_2\}$, where $\varphi_i=\star_b^{-1}b_i$. A comparison shows that 
\begin{equation}\label{E:phi0}
\varphi_0=\mathbf{1}_{q_0p_2}-\frac12\mathbf{1}_{p_1p_2}+\mathbf{1}_{p_1q_0}
\end{equation}
and that $\varphi_1$ and $\varphi_2$ satisfy analogous identities. 

A direct calculation shows that $\mathcal{D}(\partial^\bot)$ is the space of all $f\in L^2(K^{(1)},\nu_b)$ with $f|_e\in H^1(e)$, $e\in E$, satisfying the \emph{weighted Kirchhoff conditions}
\begin{equation}\label{E:Kirchhoff1}
f|_{p_1q_0}(1)-f|_{q_0p_2}(0)=0  
\end{equation}
at $q_0$ and analogous identities at $q_1$ and $q_2$ and also
\begin{equation}\label{E:Kirchhoff2}
f|_{q_2p_1}(1)-f|_{p_1q_0}(0)+2f|_{p_0p_1}(1)-2f|_{p_1p_2}(0)=0
\end{equation}
at $p_1$ and analogous identities at $p_0$ and $p_2$. Given $f\in \mathcal{D}(\partial^\bot)$, let $g$, $u$ and $z$ be as in (\ref{E:keydecomp}) and $w$ as in (\ref{E:partialbotfdecomp}). By the preceding we know there are constants $z_i$ and $w_i$, $i\in \{\emptyset, 0,1,2\}$, such that 
\begin{equation}\label{E:zandw}
z=z_\emptyset+\sum_{i=0}^2 z_i\varphi_i\quad \text{and}\quad w=w_\emptyset+\sum_{i=0}^2 w_i\varphi_i.
\end{equation}
Since $(\Delta u)|_e=b_e^{-2}u_e''$ holds on each edge $e\in E$, the comparison of (\ref{E:phi0}) and (\ref{E:zandw}) gives
\begin{equation}\label{E:guidequadratic}
u_{q_0p_2}''=z_\emptyset+z_0,\quad \frac14 u_{p_1p_2}''=z_\emptyset-\frac12 z_0\quad \text{and}\quad u_{p_1q_0}''=z_\emptyset+z_0
\end{equation}
for the edges of $K^{(1)}_0$. 
On the other hand, (\ref{E:afterappl}) gives
\[u_{q_0p_2}''=f_{q_0p_2}'-g_{q_0p_2}',\quad \frac14 u_{p_1p_2}''=\frac12(f_{p_1p_2}'-g_{p_1p_2}')\quad \text{and}\quad u_{p_1q_0}''= f_{p_1q_0}'-g_{p_1q_0}'.\]
Combining with (\ref{E:guidequadratic}), integrating along the triangle graph $q_0p_2p_1$ with counterclockwise orientation gives
\begin{align}\label{E:z0}
z_0&=\frac13(f_{q_0p_2}(1)-f_{q_0p_2}(0)-(f_{p_1p_2}(1)-f_{p_1p_2}(0))+f_{p_1q_0}(1)-f_{p_1q_0}(0))\notag\\
&=\frac13(f_{q_0p_2}(1)-f_{p_1p_2}(1)+f_{p_1p_2}(0)-f_{p_1q_0}(0));
\end{align}
in the first equality we have used the fact that by the continuity of $g$ the integral of $g'$ along the triangle graph $q_0p_2p_1$ vanishes, in the second we have used (\ref{E:Kirchhoff1}). Formula (\ref{E:z0}) for $z_0$ describes \emph{the jumps of $f$ along the triangle graph $q_0p_2p_1$}. Similar formulas hold for $z_1$ and $z_2$. A short calculation shows that 
\begin{equation}\label{E:zemptysetexposed}
z_\emptyset=0.
\end{equation}
The next observation is that for the components of $\star_b^{-1}\partial u$ on the edges of $K_0^{(1)}$ we have 
\begin{equation}\label{E:uprimeexposed}
u_{p_1q_0}'(x)=z_0x-z_0,\quad u_{q_0p_2}'(x)=z_0x\quad \text{and}\quad \frac12u_{p_1p_2}'(x)=-z_0x+z_0,\quad x\in [0,1];
\end{equation}
analogous identities hold on the edges of $K_1^{(1)}$ and $K_2^{(1)}$.

We finally note that 
\begin{equation}\label{E:w0exposed}
w_0=\frac19(2f_{p_1q_0}(0)-2f_{p_1p_2}(0)+4f_{q_0p_2}(1)-4f_{p_1p_2}(1));
\end{equation}
analogous identities hold for $w_1$ and $w_2$. For $w_\emptyset$ we obtain
\begin{equation}\label{E:wemptysetexposed}
w_\emptyset=\frac13(f_{q_0p_2}(0)-g(q_0)+f_{q_1p_0}(0)-g(q_1)+f_{q_2p_1}(0)-g(q_2))-\frac13\sum_{i=0}^2w_i.
\end{equation}
The detailed calculations for (\ref{E:zemptysetexposed}), (\ref{E:uprimeexposed}), (\ref{E:w0exposed}) and (\ref{E:wemptysetexposed}) are provided in Appendix \ref{S:repSgraph}.
\end{examples}

\begin{examples}\label{E:repSG}
Let $(\mathcal{E},\mathcal{C})$ be the standard energy form on the Sierpi\'nski gasket $K$ as in Examples \ref{Ex:SG}. Recall notation (\ref{E:harmonicforms}). Let $B=\emptyset$ consider $b=d\zeta_\emptyset$. For any $\alpha$ with $|\alpha|\geq 1$, set $\varphi_\alpha:=\star_b^{-1}d\zeta_\alpha$. Then 
\[\{\mathbf{1}\}\cup \{\varphi_\alpha:\ \alpha\in \{0,1,2\}^n, n\in \mathbb{N}\setminus \{0\}\}\] 
is an orthogonal basis for $\ker \partial^\bot$. Locally the functions $\varphi_\alpha$ relate gradients of harmonic functions to those of their shifts under the similarities $F_i$. On $K_{00}$, for instance, we have $\partial(h_0\circ F_0^{-1})=\varphi_0\partial h_0$. This connects to the well-known harmonic extension matrices, see \cite[Examples 3.2.6]{Ki01} or \cite[Section 1.3]{Str06}. As in (\ref{E:zandw}) we have $z=z_\emptyset+\sum_{|\alpha|\geq 1}z_\alpha \varphi_\alpha$, a similar expansion holds for $w$. For any $f\in \mathcal{D}(\partial^\bot)$ the decomposition 
\[\star_b\partial^\bot f-\partial g=\star_bz=z_\emptyset d\zeta_\emptyset+\sum_{|\alpha|\geq 1}z_\alpha d\zeta_\alpha\]
holds. As in \cite[Definition 2.1]{CGIS13} let $\ell_\emptyset$ denote the \enquote{lacuna} formed by the edges of the large \enquote{removed} triangle. Integration over $\ell_\emptyset$ in the sense of \cite[Definition 2.2 and Theorem 2.3.1]{CGIS13} gives
\begin{equation}\label{E:nice}
\int_{\ell_\emptyset} \star_b\partial^\bot f= z_\emptyset-\frac13\sideset{}{'}\sum_{|\alpha|\geq 1} z_\alpha,
\end{equation} 
where the summation $\sum'$ is taken over all $\alpha$ with $|\alpha|\geq 1$ such that an edge of $K_\alpha$ is contained in $\ell_\emptyset$. Formula (\ref{E:nice}) is an analog of the integrations performed to obtain (\ref{E:circleszi}) and (\ref{E:z0}).
\end{examples}

Using Theorem \ref{T:keydecomp} we can prove the following abstract integration by parts formula for $(\partial_B^\bot,\mathcal{D}(\partial_B^\bot))$. It takes into account the boundary $B$, the choice of $b$ and the topological structure of $X$. In the next section it will be used to obtain familiar and new integration by parts identities in the context of boundary quadruples.

\begin{theorem}\label{T:IbP}
Let Assumptions \ref{A:basic}, \ref{A:B} and \ref{A:complete} be in force. Assume that the martingale dimension of $(\mathcal{E},\mathcal{C})$ is one and that $b\in \ker\partial_B^\ast$ is minimal energy-dominant. Suppose that $f_1,f_2\in \mathcal{D}(\partial_B^\bot)$ have representations 
\begin{equation}\label{E:decomps}
f_i=g_i+\star_b^{-1}\partial u_i+w_i,\quad i=1,2,
\end{equation}
as in (\ref{E:keydecomp}) with $g_i\in\mathcal{C}$, $u_i\in \mathcal{D}(\Delta_B)$ and $w_i\in \ker\partial^\bot$ and let 
$z_i:=\partial_B^\bot f_i-\star_b^{-1}\partial g_i$, $i=1,2$, be as in (\ref{E:partialbotfdecomp}). Then 
\begin{multline}\label{E:IbP}
\left\langle \partial_B^\bot f_1,f_2\right\rangle_{L^2(X,\nu_b)}+\left\langle f_1,\partial_B^\bot f_2\right\rangle_{L^2(X,\nu_b)}\\
=n_B(f_1b)(g_2)+n_B(f_2b)(g_1)-n_Bb(g_1g_2)+\left\langle w_1,z_2\right\rangle_{L^2(X,\nu_b)}+\left\langle z_1,w_2\right\rangle_{L^2(X,\nu_b)}.
\end{multline}
The right-hand side in (\ref{E:IbP}) does not depend on the choice of representatives in (\ref{E:decomps}).
\end{theorem}

Note that (\ref{E:IbP}) generalizes (\ref{E:IbPC}) and (\ref{E:IbPD}).

\begin{proof}
An application of $\partial_B^\bot$ to (\ref{E:decomps}) recovers (\ref{E:partialbotfdecomp}) for $f_1$ and $f_2$. By (\ref{E:partialbotfdecomp}), (\ref{E:IbPD}), (\ref{E:decomps}) and the fact that $z_1\in \ker \partial^\bot\subset  \ker \partial_B^\bot$, we have 
\begin{align*}
\left\langle \partial_B^\bot f_1,f_2\right\rangle_{L^2(X,\nu_b)}&=\left\langle \star_b^{-1}\partial g_1,f_2\right\rangle_{L^2(X,\nu_b)}+\left\langle z_1,f_2\right\rangle_{L^2(X,\nu_b)}\\
&=-\left\langle g_1,\partial_B^\bot f_2\right\rangle_{L^2(X,\nu_b)}+n_B(f_2b)(g_1)+\left\langle z_1,g_2\right\rangle_{L^2(X,\nu_b)}+\left\langle z_1,w_2\right\rangle_{L^2(X,\nu_b)}.
\end{align*}
Using (\ref{E:keyNeumann}), (\ref{E:GaussGreen}), the duality $\star_b$ and the fact that $z_2\in \ker \partial^\bot$, we find that 
\begin{align}
\left\langle z_1,g_2\right\rangle_{L^2(X,\nu_b)}=\left\langle \Delta_B u_1,g_2\right\rangle_{L^2(X,\nu_b)}&=-\left\langle \star_b^{-1}\partial u_1,\star_b^{-1}\partial g_2\right\rangle_{L^2(X,\nu_b)}+(du_1)_B(g_2)\notag\\
&=-\left\langle \star_b^{-1}\partial u_1,\partial_B^\bot f_2\right\rangle_{L^2(X,\nu_b)}+n_B(f_1b)(g_2)-n_Bb(g_1g_2).\notag
\end{align}
Plugging in, 
\[\left\langle \partial_B^\bot f_1,f_2\right\rangle_{L^2(X,\nu_b)}=-\left\langle g_1+\star_b^{-1}\partial u_1,\partial_B^\bot f_2\right\rangle_{L^2(X,\nu_b)}+n_B(f_1b)(g_2)+n_B(f_2b)(g_1)-n_Bb(g_1g_2)+\left\langle z_1,w_2\right\rangle_{L^2(X,\nu_b)}\]
Adding $\left\langle f_1,\partial_B^\bot f_2\right\rangle_{L^2(X,\nu_b)}$ on both sides and using (\ref{E:decomps}), we now arrive at (\ref{E:IbP}), note that $\left\langle w_1, \partial_B^\bot f_2\right\rangle_{L^2(X,\nu_b)}=\left\langle w_1, z_2\right\rangle_{L^2(X,\nu_b)}$ by (\ref{E:partialbotfdecomp}). 

To check that the right-hand side does not depend on the representation, suppose that 
\[f_i=(g_i+c_i)+(\star_b^{-1}\partial u_i-c_i\star_b^{-1}\partial h)+(w_i-c_i\star_b^{-1} v),\quad i=1,2,\]
where $c_1,c_2\in\mathbb{R}$ and 
\begin{equation}\label{E:Hodgeofb}
b=\partial h+v
\end{equation} 
is the unique decomposition of $b\in \ker \partial_B^\ast$ according to (\ref{E:Hodgecomplete}). Then 
\begin{multline}
n_B(f_1b)(g_2+c_2)+n_B(f_2b)(g_1+c_1)-n_Bb((g_1+c_1)(g_2+c_2))\notag\\
+\left\langle w_1-c_1\star_b^{-1}v,z_2\right\rangle_{L^2(X,\nu_b)}+\left\langle z_1,w_2-c_2\star_b^{-1} v\right\rangle_{L^2(X,\nu_b)}
\end{multline}
equals the right-hand side in (\ref{E:IbP}) plus
\begin{equation}\label{E:residue}
n_B(f_1b)(c_2)+n_B(f_2b)(c_1)-c_1n_Bb(g_2)-c_2n_Bb(g_1)
-c_1\left\langle \star_b^{-1}v,z_2\right\rangle_{L^2(X,\nu_b)}-c_2\left\langle z_1,\star_b^{-1} v\right\rangle_{L^2(X,\nu_b)},
\end{equation}
note that (\ref{E:normalpart}) gives $n_Bb(c_1c_2)=0$. By (\ref{E:keyNeumann}), (\ref{E:compatibility}) and the orthogonality in (\ref{E:HodgeL2}) we have 
\[n_B(f_1b)(c_2)-c_2n_Bb(g_1)-c_2\left\langle z_1,\star_b^{-1} v\right\rangle_{L^2(X,\nu_b)}=c_2\left\langle z_1,\mathbf{1}-\star_b^{-1} v\right\rangle_{L^2(X,\nu_b)}=c_2\left\langle z_1,\star_b^{-1}\partial h\right\rangle_{L^2(X,\nu_b)}=0,\]
and similarly for the summands involving $c_1$. This shows that (\ref{E:residue}) equals zero.
\end{proof}


\section{Boundary quadruples for finite boundaries}\label{S:quadruples}

In \cite[Examples 3.6 and 3.7]{Arendtetal23} two abstract constructions of boundary quadruples were provided, both general enough to work for any densely defined skew symmetric linear operator on a real or complex Hilbert space $H$. In the case of the operator $\star_b^{-1}\partial_B$ on $L^2(X,\nu_b)$ a connection between the representation theoretic \cite[Example 3.6]{Arendtetal23} and the structure of $X$ cannot be expected, while \cite[Example 3.7]{Arendtetal23} quickly leads to ordinary differential equations, which for general $X$ are unavailable. Here we use Theorem \ref{T:IbP} to construct very specific boundary quadruples for $\star_b^{-1}\partial_B$ whose connection to the loop structure of $X$, the boundary $B$ and the properties of the vector field $b$ can actually be evaluated.

Let Assumptions \ref{A:basic} and \ref{A:complete} be in force. Suppose in addition that $B$ is a finite set of points. In this case the validity of Assumption \ref{A:B} follows from Remark \ref{R:atomfree}.

We use the notation of Remark \ref{R:finiteB}: Given $v\in\mathcal{D}(\partial_B^\ast)$, $q\mapsto n_Bv(q)$ denotes the
normal part $n_Bv$ of $v$, seen as an element of the space $\ell(B)$ of real valued functions on $B$. In the special case that $v=\partial u$ with $u\in \mathcal{D}(\Delta_B)$ we write $q\mapsto (dh)_B(q)$.

Assume that $b\in \ker \partial_B^\ast$ is minimal energy-dominant and the martingale dimension of $(\mathcal{E},\mathcal{C})$ is one. Then for any $f_1,f_2\in \mathcal{D}(\partial_B^\bot)$ the integration by parts formula (\ref{E:IbP}) becomes
\begin{multline}\label{E:IbPfinite}
\left\langle \partial_B^\bot f_1,f_2\right\rangle_{L^2(X,\nu_b)}+\left\langle f_1,\partial_B^\bot f_2\right\rangle_{L^2(X,\nu_b)}
=\sum_{q\in B} \big\lbrace g_2(q)n_B(f_1b)(q)+g_1(q)n_B(f_2b)(q)-g_1(q)g_2(q)n_Bb(q)\big\rbrace \\
+\left\langle w_1,z_2\right\rangle_{L^2(X,\nu_b)}+\left\langle z_1,w_2\right\rangle_{L^2(X,\nu_b)}.
\end{multline}

We construct boundary quadruples under these assumptions. To make the necessary bookkeeping more accessible, we first present the construction for the case of solenoidal fields $b\in \ker \partial^\ast$ in Theorem \ref{T:simplequadruple} before discussing the case of general divergence free vector fields $b\in \ker \partial^\ast_B$ in Theorem \ref{T:generalquadruple}.

\subsection{Solenoidal velocity fields} If $b\in \ker \partial^\ast$, then $n_Bb\equiv 0$ on all of $B$ and the first summand on the right-hand side of (\ref{E:IbPfinite}) reduces to 
\begin{equation}\label{E:simple}
\sum_{q\in B} \big\lbrace g_2(q)n_B(f_1b)(q)+g_1(q)n_B(f_2b)(q)\big\rbrace.
\end{equation}
In this case boundary quadruples for  $(-\star_b^{-1}\partial_B,\mathcal{C}_B)$ can be constructed in a fairly transparent manner. 

To do so, we endow 
\begin{equation}\label{E:Hpmtilde}
\widetilde{H}_{\pm}:=\ell(B)\oplus \ker \partial^\bot
\end{equation}
with the scalar product 
\[\big\langle (\mathring{F}_1,F_1^\bot), (\mathring{F}_2,F_2^\bot)\big\rangle_{\widetilde{H}_\pm}:=\sum_{q\in B}\mathring{F}_1(q)\mathring{F}_2(q)+\big\langle F_1^\bot,F_2^\bot\big\rangle_{L^2(X,\nu_b)},\quad (\mathring{F}_i,F_i^\bot)\in \widetilde{H}_{\pm},\ i=1,2.\]
Since $\ell(B)$ is finite dimensional and $\ker \partial_B^\bot$ is a closed subspace of $L^2(X,\nu_b)$, the space $(\widetilde{H}_{\pm}, \left\langle \cdot,\cdot\right\rangle_{\widetilde{H}_\pm})$ is Hilbert. Let 
\begin{equation}\label{E:simplequotient}
H_{\pm}:=\widetilde{H}_{\pm}\ominus\lin(\{(\mathbf{1},-\mathbf{1})\})
\end{equation}
be the orthogonal complement in $\widetilde{H}_\pm$ of the one-dimensional subspace spanned by $(\mathbf{1},-\mathbf{1})\in \ell(B)\oplus \ker \partial_B^\bot$. With the choices made here, the boundary spaces $H_-$ and $H_+$ in Theorem \ref{T:Arendt} coincide, and we use the notation $H_\pm$ for this single space. We make the agreement that $\ell(\emptyset):=\{0\}$ and that if a summand in (\ref{E:Hpmtilde}) is trivial, we omit it there, in (\ref{E:simplequotient}) and in what follows. 

By $P_{\pm}$ we denote the orthogonal projection in $\widetilde{H}_\pm$ onto $H_{\pm}$. Given $f\in \mathcal{D}(\partial_B^\bot)$, let $z\in \ker \partial^\bot$ be as in (\ref{E:partialbotfdecomp}) and choose an arbitrary representation (\ref{E:keydecomp}). Now set
\begin{equation}\label{E:simpleGs}
G_-f:=\frac{1}{\sqrt{2}} P_{\pm}\Big((g+n_B(fb)),(w+z)\Big)\quad\text{and} \quad 
G_+f:=\frac{1}{\sqrt{2}} P_{\pm}\Big((g-n_B(fb)),(w-z)\Big).
\end{equation}

\begin{lemma}
Let Assumptions \ref{A:basic} and \ref{A:complete} be in force. Assume that the martingale dimension of $(\mathcal{E},\mathcal{C})$ is one, $B$ is finite and $b\in \ker\partial^\ast$ is minimal energy-dominant. Then (\ref{E:simpleGs}) defines linear maps $G_+:\mathcal{D}(\partial_B^\bot)\to H_{\pm}$ and $G_-:\mathcal{D}(\partial_B^\bot)\to H_{\pm}$.
\end{lemma}
\begin{proof}
If $c\in\mathbb{R}$ and $f=(g+c)+(w-c)$ is an alternative representation (\ref{E:keydecomp}) of $f$, then the vectors in the right-hand sides in (\ref{E:simpleGs}) become
$\big((g+n_B(fb)+c),(w+z-c)\big)$ and $\big((g-n_B(fb)+c),(w-z-c)\big)$.
Consequently their images under $P_\pm$ do not depend on the choice of $c$.
\end{proof}

The above definitions provide boundary quadruples.

\begin{theorem}\label{T:simplequadruple}
Let Assumptions \ref{A:basic} and \ref{A:complete} be in force. Assume that the martingale dimension of $(\mathcal{E},\mathcal{C})$ is one, $B$ is finite and $b\in \ker\partial^\ast$ is minimal energy-dominant. Then $(H_\pm,H_\pm,G_-,G_+)$ as defined in (\ref{E:simplequotient}) and (\ref{E:simpleGs}) is a boundary quadruple for $(-\star_b^{-1}\partial_B,\mathcal{C}_B)$.
\end{theorem}

One can always choose a distinguished representation of $f$ which realizes the orthogonal projection $P_\pm$ in (\ref{E:simpleGs}) simultaneously for both $G_-f$ and $G_+f$. A proof can be found in Appendix \ref{S:distinguished}.

\begin{lemma}\label{L:simpledistinguished}
Let Assumptions \ref{A:basic} and \ref{A:complete} be in force. Assume that the martingale dimension of $(\mathcal{E},\mathcal{C})$ is one, $B$ is finite and $b\in \ker\partial^\ast$ is minimal energy-dominant. Given $f\in \mathcal{D}(\partial_B^\bot)$ with representation (\ref{E:keydecomp}), we can find a single constant $c\in\mathbb{R}$ such that
\begin{equation}\label{E:simpledistinguished}
G_-f=\frac{1}{\sqrt{2}}\Big((g+c+n_B(fb)),(w-c+z)\Big)\quad\text{and} \quad 
G_+f=\frac{1}{\sqrt{2}}\Big((g+c-n_B(fb)),(w-c-z)\Big),
\end{equation}
seen as equalities in $\widetilde{H}_{\pm}$.
\end{lemma}

We prove Theorem \ref{T:simplequadruple}.
\begin{proof}
Suppose that $f_1,f_2\in  \mathcal{D}(\partial_B^\bot)$ and suppose that (\ref{E:decomps}) are distinguished representations for $f_1$ and $f_2$ as in (\ref{E:simpledistinguished}). Then 
\begin{align}
\left\langle G_-f_1,G_-f_2\right\rangle_{\widetilde{H}_{\pm}}&-\left\langle G_+f_1,G_+f_2\right\rangle_{\widetilde{H}_{\pm}}\notag\\
&=\frac12\sum_{B}(g_1+n_B(f_1b))(g_2+n_B(f_2b))-\frac12\sum_{B}(g_1-n_B(f_1b))(g_2-n_B(f_2b))\notag\\
&\qquad+\frac12\left\langle w_1+z_1,w_2+z_2\right\rangle_{L^2(X,\nu_b)}-\frac12\left\langle w_1-z_1,w_2-z_2\right\rangle_{L^2(X,\nu_b)}\notag\\
&=\sum_{B} \big\lbrace g_2n_B(f_1b)+g_1n_B(f_2b)\big\rbrace 
+\left\langle w_1,z_2\right\rangle_{L^2(X,\nu_b)}+\left\langle z_1,w_2\right\rangle_{L^2(X,\nu_b)}\notag\\
&=\left\langle \partial_B^\bot f_1,f_2\right\rangle_{L^2(X,\nu_b)}+\left\langle f_1,\partial_B^\bot f_2\right\rangle_{L^2(X,\nu_b)}\label{E:simplequadruple}
\end{align}
by Theorem \ref{T:IbP}; recall that by this theorem the last equality is independent of the choice of representations (\ref{E:decomps}) for $f_1$ and $f_2$. This shows (\ref{E:arendt_equation}) for $\hat{A}=-\partial_B^\bot=(\star_b^{-1}\partial_B)^\ast$. 

By Remark \ref{R:equivcond} (i) it remains to show that the linear map $(G_-,G_+): \mathcal{D}(\partial_B^\bot)\to H_{\pm}^2$ is surjective. Given $(\mathring{F}_-,F_-^\bot), (\mathring{F}_+,F_+^\bot) \in \widetilde{H}_{\pm}$ we will provide a preimage $f \in \mathcal{D}(\partial_B^\bot)$ for $(P_\pm (\mathring{F}_-,F_-^\bot), P_\pm (\mathring{F}_+,F_+^\bot))\in H_{\pm}^2$ under $(G_-, G_+)$. Choose $c\in \mathbb{R}$ such that 
\begin{equation}\label{E:ensurecompat}
c\big(\#(B)+\nu_b(X)\big)=\frac{1}{\sqrt{2}}\Big\lbrace \big\langle F_-^\bot-F_+^\bot,\mathbf{1}\big\rangle_{L^2(X,\nu_b)}-\sum_{q\in B}\big(\mathring{F}_-(q)-\mathring{F}_+(q)\big)\Big\rbrace;
\end{equation}
since $\nu_b(X)>0$ the existence of such $c$ is guaranteed. Set 
\[w:=\frac{1}{\sqrt{2}}(F_+^\bot+F_-^\bot)\quad \text{and}\quad z:=\frac{1}{\sqrt{2}}(F_-^\bot-F_+^\bot)-c,\]
both are elements of $\ker \partial^\bot$. Let $u\in \mathcal{D}(\partial_B^\bot)$ be a weak solution of the Neumann problem 
\begin{equation}\label{E:Neumannsolenoidal}
\begin{cases}
\Delta_B u&=z\quad \quad \text{on $X\setminus B$},\\
(du)_B &=\frac{1}{\sqrt{2}}(\mathring{F}_--\mathring{F}_+)+c\quad \text{on $B$}, \end{cases}
\end{equation}
note that by (\ref{E:ensurecompat}) condition (\ref{E:compatibility}) is satisfied. Now let $g\in\mathcal{C}$ be a function with boundary values
\[g|_{B}=\frac{1}{\sqrt{2}}(\mathring{F}_++\mathring{F}_-).\]
Then $f:=g+\star_b^{-1}\partial u+w$ is an element of $\mathcal{D}(\partial_B^\bot)$. Since 
\[\frac{1}{\sqrt{2}}(w+z)=F_-^\bot-\frac{c}{\sqrt{2}},\quad\frac{1}{\sqrt{2}}(w-z)=F_+^\bot+\frac{c}{\sqrt{2}}\]
and by (\ref{E:normalpartnormalder}) and (\ref{E:Neumannsolenoidal}) also
\[\frac{1}{\sqrt{2}}(g+n_B(fb))=\mathring{F}_-+\frac{c}{\sqrt{2}}\quad \text{and} \quad \frac{1}{\sqrt{2}}(g-n_B(fb))=\mathring{F}_+-\frac{c}{\sqrt{2}}\quad \text{on $B$},\]
we find that 
\[G_-f=P_\pm(\mathring{F}_-,F_-^\bot)\quad \text{and}\quad G_+f=P_\pm(\mathring{F}_+,F_+^\bot).\]
\end{proof}

\begin{examples}\label{Ex:solenoidal}\mbox{}
\begin{enumerate}
\item[(i)] For the Dirichlet integral on $S^1$ with duality $\star_{dx}$ we have $B=\emptyset$, $\ker \partial^\bot\cong \mathbb{R}$, and $H_\pm=\{0\}$, both $G_-$ and $G_+$ are zero. In this case (\ref{E:arendt_equation}) is just the classical formula 
\[\int_{S^1}f_1'f_2dx=-\int_{S^1}f_1f_2'dx,\ f_1,f_2\in H^1(S^1),\] 
in line with Examples \ref{Ex:unitintbot} (ii) and with \cite[Proposition 3.15]{Arendtetal23}.  
\item[(ii)] We continue Examples \ref{Ex:keydecomp} (iii) on the two glued circles. Since $B=\emptyset$, the spaces $\widetilde{H}_{\pm}$ and $\ker \partial^\bot$ can be identified and the map $\iota(c_1\mathbf{1}_{C_1}+c_2\mathbf{1}_{C_2}):=(c_1,c_2)$ can be regarded as an isometric isomorphism from $\widetilde{H}_{\pm}$ onto $\mathbb{R}^2$. If $f\in\mathcal{D}(\partial^\bot)$ and $z$ and $w$ are as in Examples \ref{Ex:keydecomp} (iii), then their images under $\iota$ are their components $(z_1,z_2)\in\mathbb{R}^2$ and $(w_1,w_2)\in\mathbb{R}^2$ determined by (\ref{E:circleszi}) and (\ref{E:circleswi}).
The image $H_{\pm}$ of $\widetilde{H}_{\pm}$ under $P_\pm$ is one-dimensional and can be identified with $\lin(\{(1,-1)\})\subset \mathbb{R}^2$ in the sense that $P_\pm=\iota^{-1}P\iota$, where $P(x_1,x_2)=\frac{1}{2}(x_1-x_2,x_2-x_1)$ is the projection in $\mathbb{R}^2$ onto this subspace. Using the isometric isomorphism $\kappa(y,-y):=\sqrt{2}y$ from $(\lin(\{(1,-1)\}),\|\cdot\|_{\mathbb{R}^2})$ onto $(\mathbb{R},|\cdot|)$ it then follows that 
\begin{equation}\label{E:circlesGpm}
\kappa\circ\iota\circ G_-f=\frac12(w_1+z_1-(w_2+z_2))\quad \text{and}\quad  \kappa\circ\iota\circ G_+f=\frac12(w_1-z_1-(w_2-z_2)).
\end{equation}
The identities (\ref{E:circleszi}) and (\ref{E:circleswi}) yield the expressions
\begin{equation}\label{E:circlesbops}
w_1\pm z_1-(w_2\pm z_2)
=\Big(\frac12\pm c_1^{-1}\Big)f_1(1)-\Big(\frac12\pm c_2^{-1}\Big)f_2(1)+\Big(\frac12\mp c_1^{-1}\Big)f_1(0)-\Big(\frac12\mp c_2^{-1}\Big)f_2(0).
\end{equation} 
\item[(iii)] Let $K^{(1)}$ be the Sierpi\'nski gasket graph and $(\mathcal{E},H^1(K^{(1)}))$ the Dirichlet integral on $K^{(1)}$ with Kirchhoff vertex conditions as in Examples \ref{Ex:repSgraph}. Let $b\in \ker \partial^\ast$ be as specified there and $B=\emptyset$. In this case 
\begin{equation}\label{E:SGgraphHpm}
H_\pm=\ker\partial^\bot\ominus \mathbb{R}\cong \lin(\{\varphi_0,\varphi_1,\varphi_2\}),
\end{equation}
where $\varphi_i=\star_b^{-1}b_i$ are as in Examples \ref{Ex:repSgraph}, cf. (\ref{E:phi0}). Given $v=v_\emptyset+\sum_{i=0}^2 v_i\varphi_i\in \ker \partial^\bot$, we have 
\[P_\pm v=\sum_{i=0}^2 v_i\varphi_i.\]
Given $f\in \mathcal{D}(\partial^\bot)$ with $w$ and $z$ as in (\ref{E:keydecomp}) and (\ref{E:partialbotfdecomp}), we find that
\begin{equation}\label{E:SGgraphGpm}
G_-f=\frac{1}{\sqrt{2}}\sum_{i=0}^2 (w_i+z_i)\varphi_i\quad \text{and}\quad G_+f=\frac{1}{\sqrt{2}}\sum_{i=0}^2 (w_i-z_i)\varphi_i.
\end{equation}
\item[(iv)] Let $(\mathcal{E},\mathcal{C})$ be the standard energy form on the Sierpi\'nski gasket $K$ as in Examples \ref{Ex:SG} and \ref{Ex:DpartialBbot}. Let $B=\emptyset$ and let $b\in \ker \partial^\ast$ be minimal energy-dominant. Then $\star_b^{-1}(\lin (\{b\}))=\mathbb{R}$, and its orthogonal complement in $\ker \partial^\bot$ is the infinite dimensional space 
\[H_{\pm}=\ker\partial^\bot\ominus \mathbb{R}\cong\star^{-1}_b(\ker\partial^\ast\ominus \lin (\{b\})).\] 
Writing $P_\pm$ for the orthogonal projection in $\ker \partial^\bot$ onto $H_{\pm}$, we find that 
\[G_-f=\frac{1}{\sqrt{2}}P_{\pm}(w+z)\quad \text{and}\quad G_+f=\frac{1}{\sqrt{2}}P_{\pm}(w-z).\]
For $b=d\zeta_\emptyset$ the space $\ker\partial^\ast\ominus \lin (\{b\})$ is the closure of the span of all $d\zeta_\alpha$ with $|\alpha|\geq 1$.
\end{enumerate}
\end{examples}

\subsection{Divergence free velocity fields} For general $b\in \ker \partial_B^\ast$ the set $\{n_Bb\neq 0\}$ may be nonempty, but one can still reach a structure similar to (\ref{E:simple}), now for relative quantities. 

\begin{lemma}
Let Assumptions \ref{A:basic} and \ref{A:complete} be in force and let $B$ be finite. Assume that the martingale dimension of $(\mathcal{E},\mathcal{C})$ is one and that $b\in \ker \partial_B^\ast$ is minimal energy-dominant. Let $f_1,f_2\in\mathcal{D}(\partial_B^\bot)$ with representations (\ref{E:decomps}). Then we have 
\begin{equation}\label{E:relativequants}
g_2n_B(f_1b)+g_1n_B(f_2b)-g_1g_2n_Bb=\left\lbrace \frac{n_B(f_1b)}{n_Bb}\frac{n_B(f_2b)}{n_Bb}-\frac{(du_1)_B}{n_Bb}\frac{(du_2)_B}{n_Bb}\right\rbrace n_Bb
\end{equation}
on $\{n_Bb\neq 0\}$.
\end{lemma}

\begin{remark}
For $q\in B$ with $n_Bb(q)\neq 0$ the ratio $\frac{n_B(f_ib)(q)}{n_Bb(q)}$ 
may be seen as the relative change through modulation by $f_i$ of the normal part $n_Bb$ of $b$ at $q$.
\end{remark}

\begin{proof}
By (\ref{E:keyNeumann}) we have 
\begin{equation}\label{E:decompnormalpart}
n_B(f_ib)=(du_i)_B+g_in_Bb,\quad i=1,2,
\end{equation}
and multiplication with $g_j$ gives
\[g_jn_B(f_ib)=g_j(du_i)_B+g_ig_jn_Bb,\quad (i,j)=(1,2),(2,1).\]
Consequently
\[g_2n_B(f_1b)+g_1n_B(f_2b)-g_1g_2n_Bb=g_2(du_1)_B+g_1(du_2)_B+g_1g_2n_Bb.\]
On the other hand (\ref{E:decompnormalpart}) also gives
\[n_B(f_1b)n_B(f_2b)=(du_1)_B(du_2)_B+g_1n_Bb(du_2)_B+g_2n_Bb(du_1)_B+g_1g_2n_Bbn_Bb,\]
and combining, we arrive at (\ref{E:relativequants}).
\end{proof}

Let
\[\mathring{B}:=\{n_Bb=0\},\quad \check{B}:=\{n_Bb<0\}\quad \text{and}\quad \hat{B}:=\{n_Bb>0\};\]
these sets give a partition of $B$. We view $|n_Bb|$ as a measure on $\check{B}$ and $\hat{B}$, respectively. On $\mathring{B}$ we use the counting measure as before. Endowing
\begin{equation}\label{E:Hpmtilde4}
\widetilde{H}_{\pm}:=\ell(\mathring{B})\oplus \ell(\check{B},|n_Bb|)\oplus \ell(\hat{B},|n_Bb|)\oplus \ker \partial^\bot
\end{equation}
with the scalar product 
\begin{multline}
\big\langle (\mathring{F}_1,\check{F}_1,\hat{F}_1,F_1^\bot), (\mathring{F}_2,\check{F}_2,\hat{F}_2, F_2^\bot)\big\rangle_{H_\pm}:=\sum_{b\in \mathring{B}}\mathring{F}_1(q)\mathring{F}_2(q)+\sum_{b\in \check{B}}\check{F}_1(q)\check{F}_2(q)|n_Bb|(q)\notag\\
+\sum_{b\in \hat{B}}\hat{F}_1(q)\hat{F}_2(q)|n_Bb|(q)+\big\langle F_1^\bot,F_2^\bot\big\rangle_{L^2(X,\nu_b)},\quad (\mathring{F}_i,\check{F}_i,\hat{F}_i,F_i^\bot)\in \widetilde{H}_{\pm},\ i=1,2,
\end{multline}
we obtain a Hilbert space $(\widetilde{H}_{\pm},\left\langle\cdot,\cdot\right\rangle_{\widetilde{H}_\pm})$.

\begin{remark}\label{R:divfreebalance}
Since $b\in \ker \partial_B^\ast$, it follows that $\check{B}$ is nonempty if and only if $\hat{B}$ is and that
\[\sum_{\check{B}}|n_Bb|=-\sum_{\check{B}}n_Bb=\sum_{\hat{B}}n_Bb=\sum_{\hat{B}}|n_Bb|=\frac12 \sum_{B}|n_Bb|.\]
\end{remark}

Recall the orthogonal decomposition (\ref{E:Hodgeofb}), that is, $b=\partial h+v$ with $h\in \mathbb{H}_B$ and $v\in \ker \partial^\ast$. Let
\begin{equation}\label{E:Hminus}
H_-:=\widetilde{H}_{\pm}\ominus \lin\big(\big\lbrace \big(\frac{\mathbf{1}}{\sqrt{2}},\mathbf{1},0,-\frac{1}{\sqrt{2}}\star_b^{-1}v\big)\big\rbrace\big)
\end{equation}
be the orthogonal complement in $\widetilde{H}_{\pm}$ of the subspace spanned by $\big(\frac{\mathbf{1}}{\sqrt{2}},\mathbf{1},0,-\frac{1}{\sqrt{2}}\star_b^{-1}v\big)$ 
and
\begin{equation}\label{E:Hplus}
H_+:=\widetilde{H}_{\pm}\ominus\lin\big(\big\lbrace \big(\frac{\mathbf{1}}{\sqrt{2}},0,-\mathbf{1},-\frac{1}{\sqrt{2}}\star_b^{-1}v\big)\big\rbrace\big)
\end{equation}
be that of the subspace spanned by $\big(\frac{\mathbf{1}}{\sqrt{2}},0,-\mathbf{1},-\frac{1}{\sqrt{2}}\star_b^{-1}v\big)$. We write $P_-$ and $P_+$ for the orthogonal projection onto the subspaces $H_-$ and $H_+$, respectively. Given $f\in \mathcal{D}(\partial_B^\bot)$, let $z\in \ker\partial^\bot$ be as in (\ref{E:partialbotfdecomp}), choose a representation (\ref{E:keydecomp}) and let $g$, $u$ and $w$ be as there. We define
\begin{equation}\label{E:Gminus}
G_-f:=P_-\Big(\frac{1}{\sqrt{2}}(g+n_B(fb)), \frac{(du)_B}{|n_Bb|}, \frac{n_B(fb)}{|n_Bb|},\frac{1}{\sqrt{2}}(w+z)\Big)
\end{equation}
and
\begin{equation}\label{E:Gplus}
G_+f:=P_+\Big(\frac{1}{\sqrt{2}}(g-n_B(fb)),\frac{n_B(fb)}{|n_Bb|},\frac{(du)_B}{|n_Bb|},\frac{1}{\sqrt{2}}(w-z)\Big).
\end{equation}

If one of the summands in (\ref{E:Hpmtilde4}) is trivial, then we omit it in (\ref{E:Hminus}), (\ref{E:Hplus}), (\ref{E:Gminus}) and (\ref{E:Gplus}), as well as in the following.

\begin{lemma}
Let Assumptions \ref{A:basic} and \ref{A:complete} be in force. Assume that the martingale dimension of $(\mathcal{E},\mathcal{C})$ is one, $B$ is finite and $b\in \ker\partial_B^\ast$ is minimal energy-dominant. Then (\ref{E:Gminus}) and (\ref{E:Gplus}) define linear maps 
$G_-:\mathcal{D}(\partial_B^\bot)\to H_-$
and 
$G_+:\mathcal{D}(\partial_B^\bot)\to H_+$.
\end{lemma}

\begin{proof}
Let $c\in\mathbb{R}$ and suppose that $f=(g+c)+\star_b^{-1}\partial(u-ch) +(w-c\star_b^{-1} v)$ is an alternative representation (\ref{E:keydecomp}) of $f$. Recall (\ref{E:normalpartnormalder}). Since $(d(u-ch))_B=(du)_B-c(dh)_B$ and $|n_Bb|=-(dh)_B$ on $\check{B}$, the right-hand side in (\ref{E:Gminus}) is replaced by
\[\Big(\frac{1}{\sqrt{2}}(g+c+n_B(fb)),\frac{(du)_B}{|n_Bb|}+c,\frac{n_B(fb)}{|n_Bb|},\frac{1}{\sqrt{2}}(w-c\star_b^{-1}v+z)\Big).\]
Similarly, the right-hand side in (\ref{E:Gplus}) is replaced by 
\[\Big(\frac{1}{\sqrt{2}}(g+c-n_B(fb)),\frac{n_B(fb)}{|n_Bb|},\frac{(du)_B}{|n_Bb|}-c,\frac{1}{\sqrt{2}}(w-c\star_b^{-1}v-z)\Big).\]
It follows that $G_-f$ and $G_+f$ do not depend on the choice of $c$.
\end{proof}

We obtain boundary quadruples.

\begin{theorem}\label{T:generalquadruple}
Let Assumptions \ref{A:basic} and \ref{A:complete} be in force. Assume that the martingale dimension of $(\mathcal{E},\mathcal{C})$ is one, $B$ is finite and $b\in \ker\partial_B^\ast$ is minimal energy-dominant. Then $(H_-,H_+,G_-,G_+)$ as defined in (\ref{E:Hminus}), (\ref{E:Hplus}), (\ref{E:Gminus}) and (\ref{E:Gplus}) is a boundary quadruple for $(-\star_b^{-1}\partial_B,\mathcal{C}_B)$.
\end{theorem}

By the next Lemma we can again find distinguished representations. A proof is given in Appendix \ref{S:distinguished}.

\begin{lemma}\label{L:distinguished}
Let Assumptions \ref{A:basic} and \ref{A:complete} be in force. Assume that the martingale dimension of $(\mathcal{E},\mathcal{C})$ is one, $B$ is finite and $b\in \ker\partial_B^\ast$ is minimal energy-dominant. Given $f\in \mathcal{D}(\partial_B^\bot)$ with representation (\ref{E:keydecomp}), we can always find a single constant $c\in \mathbb{R}$ such that 
\begin{equation}\label{E:distinguished-}
G_-f=\Big(\frac{1}{\sqrt{2}}(g+c+n_B(fb)), \frac{(du)_B}{|n_Bb|}+c, \frac{n_B(fb)}{|n_Bb|},\frac{1}{\sqrt{2}}(w-c\star_b^{-1}v+z)\Big)
\end{equation}
and 
\begin{equation}\label{E:distinguished+}
G_+f=\Big(\frac{1}{\sqrt{2}}(g+c-n_B(fb)),\frac{n_B(fb)}{|n_Bb|},\frac{(du)_B}{|n_Bb|}-c,\frac{1}{\sqrt{2}}(w-c\star_b^{-1}v-z)\Big),
\end{equation}
seen as equalities in $\widetilde{H}_{\pm}$.
\end{lemma}

We prove Theorem \ref{T:generalquadruple}.

\begin{proof}
Let $f_1,f_2\in \mathcal{D}(\partial_B^\bot)$ and suppose that (\ref{E:decomps}) are distinguished representations as in (\ref{E:distinguished-}) and (\ref{E:distinguished+}). Then 
\begin{align}
\left\langle G_-f_1,G_-f_2\right\rangle_{\widetilde{H}_{\pm}}&-\left\langle G_+f_1,G_+f_2\right\rangle_{\widetilde{H}_{\pm}}\notag\\
&=-\frac12\sum_{\mathring{B}}(g_1-n_B(f_1b))(g_2-n_B(f_2b))+\frac12\sum_{\mathring{B}}(g_1+n_B(f_1b))(g_2+n_B(f_2b))\notag\\
&\qquad -\sum_{\check{B}}\frac{n_B(f_1b)}{n_Bb}\frac{n_B(f_2b)}{n_Bb}|n_Bb|+\sum_{\check{B}}\frac{(du_1)_B}{n_Bb}\frac{(du_2)_B}{n_Bb}|n_Bb|\notag\\
&\qquad\qquad -\sum_{\hat{B}}\frac{(du_1)_B}{n_Bb}\frac{(du_2)_B}{n_Bb}|n_Bb|+\sum_{\hat{B}}\frac{n_B(f_1b)}{n_Bb}\frac{n_B(f_2b)}{n_Bb}|n_Bb|\notag\\
&\qquad\qquad-\frac12\left\langle w_1-z_1,w_2-z_2\right\rangle_{L^2(X,\nu_b)}+\frac12\left\langle w_1+z_1,w_2+z_2\right\rangle_{L^2(X,\nu_b)}\notag\\
&=\sum_{\mathring{B}} \big\lbrace g_2n_B(f_1b)+g_1n_B(f_2b)\big\rbrace +\sum_{\check{B}\cup\hat{B}}\left\lbrace \frac{n_B(f_1b)}{n_Bb}\frac{n_B(f_2b)}{n_Bb}-\frac{(du_1)_B}{n_Bb}\frac{(du_2)_B}{n_Bb}\right\rbrace n_Bb\notag\\
&\qquad\qquad +\left\langle w_1,z_2\right\rangle_{L^2(X,\nu_b)}+\left\langle z_1,w_2\right\rangle_{L^2(X,\nu_b)}\notag\\
&=\left\langle \partial_B^\bot f_1,f_2\right\rangle_{L^2(X,\nu_b)}+\left\langle f_1,\partial_B^\bot f_2\right\rangle_{L^2(X,\nu_b)},\label{E:generalquadruple}
\end{align}
where we have used Theorem \ref{T:IbP}. This shows (\ref{E:arendt_equation}) for $\hat{A}=-\partial_B^\bot$. 

We verify that 
\begin{equation}\label{E:showsurjectiv}
(G_-,G_+):\mathcal{D}(\partial_B^\bot)\to H_-\times H_+
\end{equation}
is surjective. Let 
\[(\mathring{F}_-,\check{F}_-,\hat{F}_-,F^\bot_-)\in \widetilde{H}_{\pm}\quad\text{and}\quad (\mathring{F}_+,\check{F}_+,\hat{F}_+,F^\bot_+)\in \widetilde{H}_{\pm}\] 
be arbitrary representatives of their classes in $H_-$ and $H_+$. Recall (\ref{E:Hodgeofb}). Choose $c\in \mathbb{R}$ such that 
\begin{equation}\label{E:ensurecompatgeneral}
c\Big(\#(\mathring{B})+\sum_{B}|n_Bb|+\left\|v\right\|_{\mathcal{H}}^2\Big)
=-\frac{1}{\sqrt{2}}\sum_{\mathring{B}}(\mathring{F}_--\mathring{F}_+)+\sum_{\check{B}}\check{F}_-n_Bb-\sum_{\hat{B}}\hat{F}_+n_Bb+\frac{1}{\sqrt{2}}\left\langle F_-^\bot-F_+^\bot,\mathbf{1}\right\rangle_{L^2(X,\nu_b)};
\end{equation}
clearly such $c$ exists. Let 
\begin{equation}\label{E:defwandz}
w:=\frac{1}{\sqrt{2}}(F_+^\bot+F_-^\bot)\quad \text{and}\quad z:=\frac{1}{\sqrt{2}}(F_-^\bot-F_+^\bot)-c\star_b^{-1}v;
\end{equation}
clearly both $w$ and $z$ are elements of $\ker\partial_B^\bot$. Let $u\in \mathcal{D}(\Delta_B)$ be a weak solution for the Neumann problem
\begin{equation}\label{E:defuviaNeumann}
\begin{cases} \Delta_Bu &=z\quad \text{on $X\setminus B$},\\
(du)_B &=-\check{F}_-n_Bb-cn_Bb\quad \text{on $\check{B}$},\\
(du)_B &=\hat{F}_+n_Bb+cn_Bb\quad \text{on $\hat{B}$},\\
(du)_B &=\frac{1}{\sqrt{2}}(\mathring{F}_--\mathring{F}_+)+c\quad \text{on $\mathring{B}$};\end{cases}
\end{equation}
note that by Remark \ref{R:divfreebalance} and (\ref{E:ensurecompatgeneral}) condition (\ref{E:compatibility}) is satisfied. Let $g\in \mathcal{C}$ be a function with boundary values 
\[g=\frac{1}{\sqrt{2}}(\mathring{F}_++\mathring{F}_-)\ \text{on $\mathring{B}$},\quad g=\check{F}_--\check{F}_++c\ \text{on $\check{B}$}\quad\text{and}\quad g=\hat{F}_--\hat{F}_+-c\ \text{on $\hat{B}$}.\]
Define an element $f\in \mathcal{D}(\partial_B^\bot)$ by $f:=g+\star_b\partial u+w$. Now (\ref{E:defwandz}) gives
\[\frac{1}{\sqrt{2}}(w+z)=F_-^\bot-\frac{c}{\sqrt{2}}\star_b^{-1} v\quad \text{and}\quad \frac{1}{\sqrt{2}}(w-z)=F_+^\bot+\frac{c}{\sqrt{2}}\star_b^{-1} v.\]
Since 
\begin{equation}\label{E:ansatzfconseq}
n_B(fb)=n_B(gb)+(du)_B=gn_Bb+(du)_B
\end{equation}
by (\ref{E:pulloutg}), it follows from (\ref{E:defuviaNeumann}) that 
\[\frac{1}{\sqrt{2}}(g+n_B(fb))=\mathring{F}_-+\frac{c}{\sqrt{2}}\quad \text{and}\quad \frac{1}{\sqrt{2}}(g-n_B(fb))=\mathring{F}_+-\frac{c}{\sqrt{2}}\quad \text{on $\mathring{B}$}.\]
Using again (\ref{E:defuviaNeumann}) and (\ref{E:ansatzfconseq}), we find that 
\[\frac{(du)_B}{|n_Bb|}=\check{F}_-+c\quad\text{and}\quad \frac{n_B(fb)}{|n_Bb|}=\check{F}_+\quad \text{on $\check{B}$}\]
and similarly 
\[\frac{(du)_B}{|n_Bb|}=\hat{F}_++c\quad\text{and}\quad \frac{n_B(fb)}{|n_Bb|}=\hat{F}_-\quad\text{on $\hat{B}$}.\]
Putting together and comparing with (\ref{E:Gminus}) and (\ref{E:Gplus}), this gives 
\[G_-f=\Big(\mathring{F}_-+\frac{c}{\sqrt{2}},\check{F}_-+c,\hat{F}_-,F_-^\bot-\frac{c}{\sqrt{2}}\star_b^{-1}v\Big)=\Big(\mathring{F}_-,\check{F}_-,\hat{F}_-,F_-^\bot\Big)+c\Big(\frac{\mathbf{1}}{\sqrt{2}},\mathbf{1},0,-\frac{c}{\sqrt{2}}\star_b^{-1}v\Big)\]
and 
\[G_+f=\Big(\mathring{F}_+-\frac{c}{\sqrt{2}},\check{F}_+,\hat{F}_++c,F_+^\bot+\frac{c}{\sqrt{2}}\star_b^{-1}v\Big)=\Big(\mathring{F}_+,\check{F}_+,\hat{F}_+,F_+^\bot\Big)-c\Big(\frac{\mathbf{1}}{\sqrt{2}},0,-\mathbf{1},-\frac{c}{\sqrt{2}}\star_b^{-1}v\Big),\]
which in view of (\ref{E:Hminus}) and (\ref{E:Hplus}) show the surjectivity of (\ref{E:showsurjectiv}).
\end{proof}

\begin{examples}\label{Ex:simpleexHpm}\mbox{}
\begin{enumerate}
\item[(i)] For $X=[0,1]$ with $b=dx$ and $B=\{0,1\}$ the preceding recovers the situation of
Examples \ref{Ex:quadrupelinterval}: Since $|n_B(dx)|=1$, the spaces
\[H_-=\ell(\{0\})\oplus \ell(\{1\})\ominus \lin(\{(1,0)\})=\ell(\{1\})\quad \text{and}\quad H_+=\ell(\{0\})\oplus \ell(\{1\})\ominus \lin(\{(0,-1)\})=\ell(\{0\})\]
can both be identified with $\mathbb{R}$. Moreover,
\[G_-f=n_B(fdx)(1)=f(1)\quad \text{and}\quad G_+f=n_B(fdx)(0)=-f(0).\]
\item[(ii)] We continue Examples \ref{Ex:keydecomp} (iv), where $X$ is a finite metric tree with root $q_0$ and leaves $q_1,...,q_N$, $B=\{q_0,...,q_N\}$, and choose $b=\partial h$ with $h=(h_e)_{e\in E}$ harmonic on $X\setminus B$ and such that $h(q_0)=0$, $h(q_i)>0$, $i=1,...,N$, and $h_e'>0$ for all $e\in E$. Then (\ref{E:Hpmtilde4}) reduces to $\widetilde{H}_\pm=\ell(\{q_0\},|(dh)_B|)\oplus \ell(\{q_1,...,q_N\},|(dh)_B|)$, which is $\mathbb{R}^{N+1}$ with scalar product 
\begin{equation}\label{E:treeweightedSP}
\left\langle x,y\right\rangle_{\widetilde{H}_\pm}=\sum_{i=0}^N x_iy_ih_{e_i}',\quad x,y\in \mathbb{R}^{N+1}.
\end{equation} 
The orthogonal complement $H_-=\widetilde{H}_\pm\ominus \lin(\{(1,\mathbf{0})\})$ of $(1,0,...0)$ in $(\mathbb{R}^{N+1}\left\langle \cdot,\cdot\right\rangle_{\widetilde{H}_\pm})$ can be identified with $\{0\}\times \mathbb{R}^N$. Given $f\in \mathcal{D}(\partial_B^\bot)$, we find that 
\[G_-f=P_-\Big(\frac{(du)_B}{|(dh)_B|},\frac{n_B(f\partial h)}{|(dh)_B|}\Big)=P_-\Big(-\frac{u_{e_0}'}{h_{e_0}'},f(q_1),...,f(q_N)\Big)=(0,f(q_1),...,f(q_N));\]
here we use the notation of Examples \ref{Ex:keydecomp} (iv). The space $H_+=\widetilde{H}_\pm\ominus \lin(\{(0,\mathbf{1})\})$ is the orthogonal complement of $(0,1,...,1)$, the projection $P_+$ onto it is as in (\ref{E:distinguished+}), and this gives
\[G_+f=P_+\Big(\frac{n_B(f\partial h)}{|(dh)_B|},\frac{(du)_B}{|(dh)_B|}\Big)=P_+\Big(-f(q_0), \frac{u_{e_1}'}{h_{e_1}'},...,\frac{u_{e_N}'}{h_{e_N}'}\Big)=\Big(-f(q_0), \frac{u_{e_1}'}{h_{e_1}'}-\frac{u_{e_0}'}{h_{e_0}'}, ...,\frac{u_{e_N}'}{h_{e_N}'}-\frac{u_{e_0}'}{h_{e_0}'}\Big).\]
The orthogonality relation $\sum_{i=1}^N\Big(\frac{u_{e_i}'}{h_{e_i}'}-\frac{u_{e_0}'}{h_{e_0}'}\Big)h_{e_i}'=0$
reflects the harmonicity of $h$ and $u$.
\item[(iii)] Let $(\mathcal{E},\mathcal{C})$ be the standard energy form on the Sierpi\'nski gasket $K$ as in Examples \ref{Ex:SG} and \ref{Ex:DpartialBbot}. Let $B=V_0$ and $b=\partial h$, where $h\in\mathcal{C}$ is harmonic on $K\setminus V_0$ with boundary values $h(q_0)=0$ and $h(q_1)=h(q_2)=1$. Then $(dh)_{V_0}(q_0)=-2$ and $(dh)_{V_0}(q_1)=(dh)_{V_0}(q_2)=1$. Up to isometry of Hilbert spaces,
$\widetilde{H}_{\pm}=\mathbb{R}^3\oplus \ker \partial^\bot$, here we consider $\mathbb{R}^3$ endowed with the scalar product
\begin{equation}\label{E:SPiii}
(x,y)\mapsto 2x_0y_0+x_1y_1+x_2y_2.
\end{equation}
Moreover, $H_-=(\mathbb{R}^3\ominus \lin(\{(1,0,0)\}))\oplus \ker \partial^\bot$  and, given $f\in \mathcal{D}(\partial^\bot_{V_0})$, 
\begin{align}
G_-f&=P_-\big(\frac12(du)_{V_0}(q_0),n_{V_0}(f\partial h)(q_1),n_{V_0}(f\partial h)(q_2),\frac{1}{\sqrt{2}}(w+z)\big)\notag\\
&=\big(0,n_{V_0}(f\partial h)(q_1),n_{V_0}(f\partial h)(q_2),\frac{1}{\sqrt{2}}(w+z)\big).\notag
\end{align}
Similarly, $H_+=(\mathbb{R}^3\ominus \lin(\{(0,-1,-1)\}))\oplus \ker \partial^\bot$  and, given $f\in \mathcal{D}(\partial^\bot_{V_0})$, 
\begin{align}
G_+f&=P_+\big(\frac12n_{V_0}(f\partial h)(q_0), (du)_{V_0}(q_1),(du)_{V_0}(q_2),\frac{1}{\sqrt{2}}(w-z)\big)\notag\\
&=\big(\frac12n_{V_0}(f\partial h)(q_0), (du)_{V_0}(q_1)+\frac12(du)_{V_0}(q_0),(du)_{V_0}(q_2)+\frac12(du)_{V_0}(q_0),\frac{1}{\sqrt{2}}(w-z)\big).\notag
\end{align}
\item[(iv)] Let $K$, $(\mathcal{E},\mathcal{C})$ and $B=V_0$ be as in (iii), but now consider $b=\partial h$, where $h\in\mathcal{C}$ is harmonic on $K\setminus V_0$ with boundary values $h(q_0)=0$ and $h(q_1)=-\frac{1}{6}$ and $h(q_2)=\frac{1}{6}$.
Then $(dh)_{V_0}(q_0)=0$, $(dh)_{V_0}(q_1)=-\frac12$ and $(dh)_{V_0}(q_2)=\frac12$. Up to isometry, $\widetilde{H}_{\pm}=\mathbb{R}^3\oplus \ker \partial^\bot$, where $\mathbb{R}^3$ is endowed with the scalar product 
\begin{equation}\label{E:degenSP}
(x,y)\mapsto x_0y_0+\frac12 x_1y_1+\frac12 x_2y_2.
\end{equation}
Moreover, 
\begin{equation}\label{E:degenHpm}
H_-=\widetilde{H}_{\pm}\ominus \lin(\{(\frac{1}{\sqrt{2}},1,0,0)\}),\quad H_+=\widetilde{H}_{\pm}\ominus \lin(\{(\frac{1}{\sqrt{2}},0,-1,0)\})
\end{equation}
and, given $f\in \mathcal{D}(\partial_{V_0}^\bot)$,
\begin{align}\label{E:degenGminus}
G_-f&=P_-\big(\frac{1}{\sqrt{2}}(u(q_0)+n_{V_0}(f\partial h)(q_0)), 2(du)_{V_0}(q_1), 2n_{V_0}(f\partial h)(q_2),\frac{1}{\sqrt{2}}(w+z)\big)\\
&=\big(\frac{1}{\sqrt{2}}\big(\frac12(u(q_0)+n_{V_0}(f\partial h)(q_0))-(du)_{V_0}(q_1)\big), (du)_{V_0}(q_1)-\frac12(u(q_0)+n_{V_0}(f\partial h)(q_0)),\notag\\
&\qquad\qquad \qquad 2n_{V_0}(f\partial h)(q_2),\frac{1}{\sqrt{2}}(w+z)\big)\notag
\end{align}
and 
\begin{align}\label{E:degenGplus}
G_+f&=P_+\big(\frac{1}{\sqrt{2}}(u(q_0)-n_{V_0}(f\partial h)(q_0)), 2n_{V_0}(f\partial h)(q_1),2(du)_{V_0}(q_2), \frac{1}{\sqrt{2}}(w-z)\big)\\
&=\big(\frac{1}{\sqrt{2}}\big(\frac{1}{2}(u(q_0)-n_{V_0}(f\partial h)(q_0))+(du)_{V_0}(q_2)\big), 2n_{V_0}(f\partial h)(q_1),\notag\\
&\qquad\qquad\qquad (du)_{V_0}(q_2)+\frac12(u(q_0)-n_{V_0}(f\partial h)(q_0)),\frac{1}{\sqrt{2}}(w-z)\big).\notag
\end{align}

\item[(v)] Recall the notation of Examples \ref{Ex:SG}. Let $K^{(0)}$ be the complete (triangle) metric graph with vertex set $V_0$. For each $m\geq 0$, let $K^{(m)}$ denote the $m$-level graph approximation to the Sierpi\'nski gasket $K$, that is, the metric graph with vertex set $V_m$ determined by $K^{(m)}:=\bigcup_{|\alpha|=m}F_\alpha(K^{(0)})$. For $m=1$ we recover $K^{(1)}$ from Examples \ref{Ex:repSgraph}. 

Let $h$ be the harmonic function on $K\setminus V_0$ with boundary values $h(q_0)=0$, $h(q_1)=-\frac{1}{6}$ and $h(q_2)=\frac{1}{6}$. Since the harmonic extension matrices have full rank, cf. \cite[Examples 3.2.6]{Ki01} or \cite[Section 1.3]{Str06},
the edge-wise linear extensions to $K^{(m)}$ of $h|_{V_m}$, which we denote by $h^{(m)}$, are energy-dominant for the Dirichlet integral $(\mathcal{E},H^1(K^{(m)}))$.  Since on each edge the energy measures of the $h^{(m)}$ are a constant times the Lebesgue measure, they are minimal energy-dominant. Moreover, $h^{(m)}$ is harmonic on $K^{(m)}\setminus V_0$. On $K^{(m)}$ we use $\nu_{\partial h^{(m)}}$ and $\star_{\partial h^{(m)}}$.

For fixed $m\geq 0$ the elements of the $\frac12(3^{m+1}-1)$-dimensional space $\ker \partial^\ast$ correspond to the loops in $K^{(m)}$. Through local restriction to $V_m$ and linear extension it may be seen as a subspace of the infinite-dimensional kernel of the coderivation on $K$ as considered in (iv). Also the boundary quadruples appear as \enquote{substructures}: We have $\widetilde{H}_{\pm}=\mathbb{R}^3\oplus \ker \partial^\bot$, where $\mathbb{R}^3$ is endowed with the scalar product (\ref{E:degenSP}), and $H_-$, $H_+$, $G_-$ and $G_+$ are as in (\ref{E:degenHpm}), (\ref{E:degenGminus}) and (\ref{E:degenGplus}). Note that the normal derivatives of $h^{(m)}$ on $V_0$ have the same values as those of $h$.

\item[(vi)] Let $K^{(m)}$ and $K$ be as in (v), but now consider the harmonic function with $h(q_0)=0$, $h(q_1)=h(q_2)=1$. The edge-wise linear extension $h^{(m)}$ of $h|_{V_m}$ is constant on the edge opposite to $q_0$ in the smallest triangle containing $q_0$ and therefore not energy-dominant on $K^{(m)}$. However, it is minimal energy-dominant on the reduced metric graph $\widetilde{K}^{(m)}$ obtained from $K^{(m)}$ by removing this edge; see Figure \ref{F:reduced}. The normal derivatives of $h^{(m)}$ restricted to this reduced space are as those of $h$ in (iii), and again we have $\widetilde{H}_{\pm}=\mathbb{R}^3\oplus \ker \partial^\bot$, where $\mathbb{R}^3$ is endowed with (\ref{E:SPiii}). What is different now is $\ker \partial^\bot$, since for $\widetilde{K}^{(m)}$ the space $\ker \partial^\ast$ has only $\frac12(3^{m+1}-1)-1$ dimensions. However, we can still use the formulas for $H_-$, $H_+$, $G_-$ and $G_+$ provided in (iii).
\begin{figure}[H]
	\centering
	\begin{tikzpicture}[scale=3.5]
		\draw (0, 0) --  (1, 0)  --  (1/2,0.866) --  (0, 0);
		\draw (1/2, 0) -- (3/4, 0.433);
		\draw (1/4, 0.433) -- (1/2, 0);
		\node[above] at (1/2, 0.866) {$q_0$};
		\node[left] at (0, 0) {$q_1$};
		\node[right] at (1, 0) {$q_2$};
		\node[below] at (1/2, 0) {$p_0$};
		\node[right] at (3/4, 0.433) {$p_1$};
		\node[left] at (1/4, 0.433) {$p_2$};
		\draw (2, 0) --  (3, 0)  --  (5/2,0.866) --  (2, 0);
		\draw (5/2, 0) -- (11/4, 0.433);
		\draw (9/4, 0.433) -- (5/2, 0);
		\draw (9/4, 0.433) -- (11/4, 0.433);
		\draw (9/4, 0) -- (17/8, 0.217);
		\draw (9/4, 0) -- (19/8, 0.217);
		\draw (17/8, 0.217) -- (19/8, 0.217);
		\draw (11/4, 0) -- (23/8, 0.217);
		\draw (11/4, 0) -- (21/8, 0.217);
		\draw (21/8, 0.217) -- (23/8, 0.217);
		\draw (5/2, 0.433) -- (21/8, 0.65);
		\draw (5/2, 0.433) -- (19/8, 0.65);
		\node[above] at (5/2, 0.866) {$q_0$};
		\node[left] at (2, 0) {$q_1$};
		\node[right] at (3, 0) {$q_2$};
		\node[below] at (5/2, 0) {$p_0$};
		\node[right] at (11/4, 0.433) {$p_1$};
		\node[left] at (9/4, 0.433) {$p_2$};
	\end{tikzpicture}
	\caption{The reduced metric graphs $\widetilde{K}^{(1)}$ and $\widetilde{K}^{(2)}$.}
	\label{F:reduced}
\end{figure}
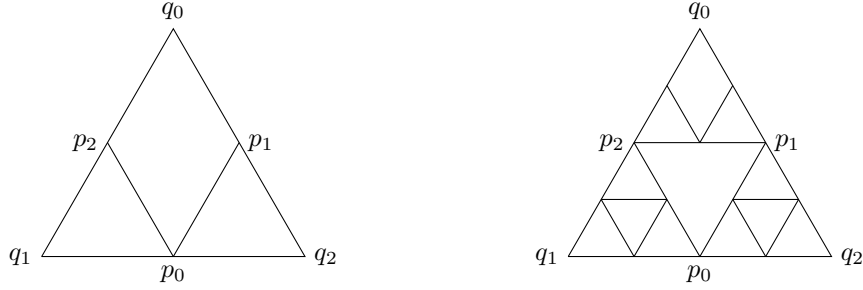 
\end{enumerate}
\end{examples}

\section{Generator domains and well-posedness}\label{S:wellpos}

Similarly as before we use the agreement that if a summand in (\ref{E:Hpmtilde4}) is trivial, then all related spaces and maps are reduced accordingly. If in the following theorem we have $b\in \ker\partial^\ast$, then we agree to set $H_-:=H_\pm$ and $H_+:=H_\pm$, where $H_\pm $ is as in (\ref{E:simplequotient}), and to use $G_-$ and $G_+$ as in (\ref{E:simpleGs}). 
The results in \cite{Arendtetal23} give the following statement on well-posedness for continuity equations.

\begin{theorem}\label{T:wellposed}
Let Assumptions \ref{A:basic} and \ref{A:complete} be in force. Assume that the martingale dimension of $(\mathcal{E},\mathcal{C})$ is one, $B$ is finite and $b\in \ker\partial^\ast_B$ is minimal energy-dominant. Let $H_-$, $H_+$ $G_-$ and $G_+$ be as in (\ref{E:Hminus}), (\ref{E:Hplus}), (\ref{E:Gminus}) and (\ref{E:Gplus}), and let $\Theta:H_-\to H_+$ be a linear contraction.  
\begin{enumerate}
\item[(i)] The operator
\[\mathcal{D}(A^{\Theta}):=\{f\in\mathcal{D}(\partial_B^\bot):\ \Theta G_- f=G_+f\},\qquad A^{\Theta}f:=-\partial_B^\bot f,\quad f\in \mathcal{D}(A^{\Theta}),\]
is an extension of $(-\star_b^{-1}\partial_B,\mathcal{C}_B)$ and generates a strongly continuous semigroup $(T_t^\Theta)_{t\geq 0}$ of contractions on $L^2(X,\nu_b)$.  
\item[(ii)] For any $\mathring{v}\in \mathcal{D}(A^\Theta)$ the function $v(t)=T^\Theta_t\mathring{v}$, $t>0$, $v(0)=\mathring{v}$, is the unique element $v$ of $C^1([0,+\infty),L^2(X,\nu_b))\cap C([0,+\infty),\mathcal{D}(A^{\Theta}))$ such that 
\begin{equation}\label{E:simpleCauchy}
\begin{cases} \frac{d}{dt} v&=A^\Theta v,\quad t\in (0,+\infty),\\ 
v(0)&=\mathring{v}.\end{cases}
\end{equation}
Moreover, $v$ satisfies the mass balance equation 
\begin{equation}\label{E:massbalance}
\int_Xv(t)\:d\nu_b=\int_X\mathring{v}\:d\nu_b-\sum_{q\in B}\int_0^tn_B(v(s)b)(q)ds,\quad t\in [0,+\infty).
\end{equation}
\end{enumerate}
\end{theorem}

As usual, we call $v$ as in Theorem \ref{T:wellposed} the (unique) \emph{solution of the Cauchy problem} (\ref{E:simpleCauchy}).

\begin{proof}
Item (i) follows from \cite[Theorem 3.10]{Arendtetal23} as quoted in Theorem \ref{T:Arendt} and from Theorem \ref{T:simplequadruple}. Item (ii) then follows from the Lumer-Phillips theorem, \cite[Theorem 3.1]{LP61}. Integrating the familiar identity
\[v(t)-\mathring{v}=T^\Theta_t\mathring{v}-\mathring{v}=\int_0^t A^\Theta T^\Theta_s\mathring{v}\:ds=-\int_0^t\partial_B^\bot v(s)\:ds\]
over $X$ and using (\ref{E:IbPD}) with $\varphi=\mathbf{1}$, we arrive at (\ref{E:massbalance}).
\end{proof}

We view the Cauchy problem (\ref{E:simpleCauchy}) as a rigoros formulation of the 
boundary initial value problem 
\begin{equation}\label{E:simplebivp}
\begin{cases} \partial_tv(t,x)+\partial_B^\bot v(t,x) & =0\quad \text{for $(t,x)\in (0,+\infty)\times X$},\\
\Theta G_-v(t,x) & =G_+v(t,x)\quad \text{for $(t,x)\in[0,+\infty)\times X$},\\
v(0,x) &=\mathring{v}(x)\quad \text{for $x\in X$}.\end{cases}
\end{equation}
for equation (\ref{E:cemodmod}). Strictly speaking, the second line may have to be augmented by further conditions needed to ensure that $v(t,\cdot)$ stays in the correct function space $\mathcal{D}(A^\Theta)$. Moreover, not in all cases the space $X$ in the second line of (\ref{E:simplebivp}) as written can be replaced by a boundary $B\subset X$. It may also encode a certain behaviour with respect to the loop structure, see Examples \ref{Ex:circleswellpos} below.


\begin{remark}\label{R:kernelconst}
Let $\mathring{v}\in\mathcal{D}(A^\Theta)$.
It is quickly seen that $\ker A^\Theta=\mathcal{D}(A^\Theta)\cap \ker \partial^\bot_B$ and that for the solution $v$ of (\ref{E:simpleCauchy}) the following are equivalent:
\begin{enumerate}
\item[(i)] $\frac{d}{dt}v(t)=0$ for all $t>0$,
\item[(ii)] $T^\Theta_t\mathring{v}=\mathring{v}$ for all $t\geq 0$,
\item[(iii)] $\mathring{v}\in \ker A^\Theta$,
\item[(iv)] $v(t)\in  \ker A^\Theta$ for all $t\geq 0$. 
\end{enumerate}
That (i) implies (ii) follows by strong continuity, that (ii) implies (iii) by the semigroup formula for the generator. Since $A^\Theta T^\Theta_t \mathring{v}=T^\Theta_t A^\Theta\mathring{v}=0$, $t>0$, we see that (iii) implies (iv), and by first line in (\ref{E:simpleCauchy}) item (iv) gives (i). If one of these conditions is satisfied, we call the solution $v$ \emph{stationary}.
\end{remark}

We discuss some examples. 

\begin{examples}\label{Ex:classtheta} The classical examples fit in as expected:
\begin{enumerate}
\item[(i)] For the circle $X=S^1$ with duality $\star_{dx}$ as in Examples \ref{Ex:solenoidal} (i) the only possible choice is $\Theta\equiv 0$. In this case $\mathcal{D}(A^0)=H^1(S^1)$, the semigroup is as stated in Examples \ref{Ex:S1case}.

The Cauchy problem (\ref{E:simpleCauchy}) may be viewed as a formulation of
\[\begin{cases} \partial_t v(t,x)+\partial_x v(t,x) & =0\quad \text{for $(t,x)\in (0,+\infty)\times S^1$,}\\
v(0,x) &=\mathring{v}(x)\quad \text{for $x\in S^1$.}\end{cases}\]
The solution $v$ is periodic. If $\mathring{v}$ is constant on $S^1$, then $v(t)=\mathring{v}$ for all $t$, in line with Remark \ref{R:kernelconst}.

\item[(ii)] For $X=[0,1]$ with $b=dx$ and $B=\{0,1\}$ as in Examples \ref{Ex:simpleexHpm} (i) the possible linear contractions $\Theta$ are the multiplications by numbers $-1\leq \theta\leq 1$ as described in Examples \ref{Ex:quadrupelinterval}.

The Cauchy problem (\ref{E:simpleCauchy}) may be interpreted as an implementation of
\[\begin{cases} \partial_tv(t,x)+\partial_x v(t,x) & =0\quad \text{for $(t,x)\in (0,+\infty)\times (0,1)$,}\\
\theta v(t,1) & =-v(t,0)\quad \text{for $t\in [0,+\infty)$},\\
v(0,x) &=\mathring{v}(x)\quad \text{for $x\in [0,1]$.}\end{cases}\]
The choice $\theta=-1$ reproduces the circle case (i) with periodic solutions $v$. The special case $\theta=0$ gives $\mathcal{D}(A^0)=\{f\in H^1(0,1): f(0)=0\}$, the generated semigroup is the nilpotent right translation semigroup acting by $T^0_t f(x)=f(x-t)$ if $t\leq x$ and $T^0_t f(x)=0$ if $t>x$, see \cite[I.4.17 and II.3.19]{EngelNagel00}.
In this case the second line in (\ref{E:simplebivp}) is $v(t,0)=0$, $t\geq 0$. The semigroup $(T^\theta_t)_{t\geq 0}$  is positivity preserving if and only if $-1\leq \theta\leq 0$, see \cite[p. 265]{Arendt86}.
\end{enumerate}
\end{examples}

\begin{examples}\label{Ex:circleswellpos} For the two glued circles $X=\bigcup_{j=1}^2 C_j$ in Examples \ref{Ex:solenoidal} (ii) with $b=c_1dx_1+c_2dx_2$ and $B=\emptyset$ a contraction $\Theta$ on the one-dimensional space $H_{\pm}$ can be identified with the multiplication by a number $-1\leq \theta\leq 1$, and by (\ref{E:circlesGpm}) and (\ref{E:circlesbops}) we have $\theta G_-f=G_+f$ if and only if 
\begin{align}
(\theta+1)(c_1^{-1}f_1(1)-c_2^{-1}f_2(1))&+\frac12(\theta-1)(f_1(1)-f_2(1))\notag\\
&=(\theta+1)(c_1^{-1}f_1(0)-c_2^{-1}f_2(0))-\frac12(\theta-1)(f_1(0)-f_2(0)).\notag
\end{align}
For the special case of a single positive constant $c_1=c_2=c$ this admits the convenient expression
\begin{equation}\label{E:hinge}
\bar{\theta}(f_1(1)-f_2(1))=f_1(0)-f_2(0),
\end{equation}
where
\[\bar{\theta}:=\frac{c^{-1}(\theta+1)+\frac12(\theta-1)}{c^{-1}(\theta+1)-\frac12(\theta-1)}.\]
The function $\theta\mapsto \bar{\theta}$ is strictly increasing on $[-1,1]$ with values ranging from $-1$ to $1$. Condition
(\ref{E:hinge}) relates the difference at $p$ between the \enquote{endpoint values} $f_i(1)$ to that between the \enquote{initial point values} $f_i(0)$. Combining it with the weighted Kirchhoff condition in (\ref{E:racetrack}) shows that $\mathcal{D}(A^\theta)$ is the subspace of $H^1(0,1)\oplus H^1(0,1)$ consisting of all $f=(f_1,f_2)$  such that 
\[\frac12\begin{pmatrix}
1+\bar{\theta} & 1-\bar{\theta}\\ 1-\bar{\theta} & 1+\bar{\theta}\end{pmatrix}\begin{pmatrix}
f_1(1)\\ f_2(1)
\end{pmatrix}
=\begin{pmatrix} f_1(0)\\ f_2(0)\end{pmatrix}.\]
Given $f_1(1)$ and $f_2(1)$, the values $f_1(0)$ and $f_2(0)$ are determined. 
For $\bar{\theta}\neq 0$ also the converse holds. The choice $\bar{\theta}= 0$ forces the symmetry condition $f_1(0)=f_2(0)=\frac12(f_1(1)+f_2(1))$. We have $\bar{\theta}=1$ if and only if $\theta=1$; in this case $z=0$ in (\ref{E:partialbotfdecomp}), and since globally harmonic functions are constant, it follows that $\star^{-1}_b\partial u=0$ in (\ref{E:fminusg}) and (\ref{E:keydecomp}). Similarly, $\bar{\theta}=-1$ if and only if $\theta=-1$; in this case $w=0$ in (\ref{E:fminusg}) and (\ref{E:keydecomp}).

The Cauchy problem (\ref{E:simpleCauchy}) for $A^\theta$ may be viewed as a rigoros formulation of
\[\begin{cases} \partial_tv(t,x)+c^{-1}\partial_x v(t,x) & =0\quad \text{for $(t,x)\in (0,+\infty)\times X$,}\\
\frac12(1+\bar{\theta})v_1(t,1)+\frac12(1-\bar{\theta})v_2(t,1)&=v_1(t,0)\quad \text{\enquote{for $t\in [0,+\infty)$ and at $p$},}\\
\frac12(1-\bar{\theta})v_1(t,1)+\frac12(1+\bar{\theta})v_2(t,1)&=v_2(t,0)\quad \text{\enquote{for $t\in [0,+\infty)$ and at $p$},}\\
v(0,x) &=\mathring{v}(x)\quad \text{for $x\in X$;}\end{cases}\]
here condition (\ref{E:racetrack}) is included explicitely.

Suppose that $C_1$ and $C_2$ have opposite orientations and that $\mathring{v}$ is a bump compactly supported on $C_1$ and away from $p$. The case $\bar{\theta}=1$ describes a periodic motion $v$ where the bump travels along $C_1$. The case $\bar{\theta}=-1$ describes a periodic motion $v$ where the bump \enquote{switches tracks} from one circle to the other each time it passes through $p$. 


\end{examples}

\begin{remark}\label{R:noweakwellpos}\mbox{}
\begin{enumerate}
\item[(i)] Under the hypotheses stated in Theorem \ref{T:wellposed} the form $(\mathcal{E},\mathcal{C})$ and the measure $\nu_b$ satisfy Assumptions \ref{A:Dirichletform} (i) and (ii). Moreover, the unique solution $v$ of (\ref{E:simpleCauchy}) is also a solution of (\ref{E:cemod}) in the weak sense. In fact, we have $b\in L^\infty(X,(\mathcal{H}_x)_{x\in X},\nu_b)$ by  (\ref{E:densityone}) and, given arbitary $\varphi\in \mathcal{C}$ and $t\in (0,T)$, 
\[\left\langle v(t)b, \partial\varphi\right\rangle_{\mathcal{H}}=
\left\langle \partial^\ast (\star_bv(t)),\varphi\right \rangle_{L^2(X,\nu_b)}=-\left\langle \partial^\bot v(t),\varphi\right\rangle_{L^2(X,\nu_b)}=\big\langle\frac{d}{dt}v(t),\varphi\big\rangle_{L^2(X,\nu_b)}=\frac{d}{dt}\left\langle v(t),\varphi\right\rangle_{L^2(X,\nu_b)}.\]
For any $\psi\in C_c^1([0,T))$ integration by parts now yields
\[\int_0^T \psi(t)\left\langle v(t)b, \partial\varphi\right\rangle_{\mathcal{H}}\:dt=-\psi(0)\left\langle \mathring{v},\varphi\right\rangle_{L^2(X,\nu_b)}-\int_0^T\psi'(t)\left\langle v(t),\varphi\right\rangle_{L^2(X,\nu_b)}\:dt,\]
that is, (\ref{E:weaksol}). 
\item[(ii)] In general different choices of the contraction $\Theta$ lead to different solutions $v$ for one and the same initial condition $\mathring{v}$. As a consequence, solutions in the weak sense as discussed in Section \ref{S:Dirichletweak} are not unique, cf. Remark \ref{R:announcenonweakwellpos}. Examples \ref{Ex:circleswellpos} provides a concrete counterexample:
If we choose $\mathring{v}$ to be a bump compactly supported on $C_1$ and away from $p$, then the choices $\bar{\theta}=1$ and $\bar{\theta}=-1$ lead to different  solutions in the weak sense with common initial condition $\mathring{v}$.
\end{enumerate}
\end{remark}

\begin{examples}\label{Ex:wellposedtrees}
We continue Examples \ref{Ex:simpleexHpm} (ii) for finite metric trees. Recall that $\widetilde{H}_{\pm}$ equals $\mathbb{R}^{N+1}$, endowed with the scalar product (\ref{E:treeweightedSP}). Given a linear contraction $\Theta:H_-\to H_+$, the map $\overline{\Theta}:=\Theta P_-$ is a linear contraction from $\widetilde{H}_\pm$ into $H_+$, its matrix is of the form
\[\overline{\Theta}=\begin{pmatrix}
\ast & \theta_{01}  & \dots & \theta_{0N}\\
\ast & \theta_{11} & \dots & \theta_{1N}\\
\vdots &  \vdots & & \vdots\\
\ast & \theta_{N1} & \dots & \theta_{NN}
\end{pmatrix}\quad \text{and satisfies}\quad \begin{pmatrix}
\theta_{11}  & \dots & \theta_{1N}\\
\vdots &  & \vdots\\
\theta_{N1}  & \dots & \theta_{NN}
\end{pmatrix}^T \begin{pmatrix}
h_{e_1}'\\ \vdots \\ h_{e_N}'
\end{pmatrix}=0. \]
Here $\ast$ stands for an arbitrary entry. The stated condition ensures that $\overline{\Theta}$ maps into $H_+$. An element $f\in \mathcal{D}(\partial_B^\bot)$ is in $\mathcal{D}(A^\Theta)$ if any only if 
\begin{equation}\label{E:treethetabar}
\overline{\Theta}\begin{pmatrix}
0\\ f(q_1)\\ \vdots\\ f(q_N)
\end{pmatrix}=\begin{pmatrix}
-f(q_0)\\ \frac{u_{e_1}'}{h_{e_1}'}-\frac{u_{e_0}'}{h_{e_0}'}\\ \vdots\\ \frac{u_{e_N}'}{h_{e_N}'}-\frac{u_{e_0}'}{h_{e_0}'}
\end{pmatrix}.
\end{equation}

For $N=1$ the metric tree $X$ consists only of a single edge $e=e_0=e_1$ with initial point $q_0$ and endpoint $q_1$. The only option for $\overline{\Theta}$ is to have $-1\leq \theta_{01}\leq 1$ and $\theta_{11}=0$. This forces $-\theta_{01}f(q_1)=f(q_0)$ and recovers Examples \ref{Ex:classtheta} (ii). For general $N\geq 1$ we point out the following special cases:
\begin{enumerate}
\item[(a)] For $\Theta\equiv 0$ we have $f(q_0)=0$ and
\begin{equation}\label{E:ueisequal}
\frac{u_{e_i}'}{h_{e_i}'}=\frac{u_{e_0}'}{h_{e_0}'},\quad i=1,...,N,
\end{equation}
that is, $u$ is a constant multiple of $h$, up to an additive constant. Together with (\ref{E:keydecomptree}) in Examples \ref{Ex:keydecomp} (iv) this implies that $\mathcal{D}(A^0)=\{f\in H^1(X):\ f(q_0)=0\}$, what generalizes the nilpotent case in Examples \ref{Ex:classtheta} (ii).
\item[(b)] If $\theta_{0i}=-h'_{e_i}/h_{e_0}'$, $i=1,...,N$ and all other entries of $\overline{\Theta}$ are zero, then (\ref{E:ueisequal}) still holds and 
\[\mathcal{D}(A^\Theta)=\{g\in H^1(X):\ \sum_{i=1}^N g(q_i)h_{e_i}'=g(q_0)h'_{e_0}\}.\] 
This generalizes the periodic case in Examples \ref{Ex:classtheta} (ii).
\end{enumerate}

The first line in the boundary initial value problem (\ref{E:simplebivp}) becomes
\[\partial_tv(t,x)-(h'_e)^{-1}\partial_xv(t,x)=0,\quad x\in e,\quad e\in E,\quad t\in (0,+\infty),\]
the second is replaced by (\ref{E:treethetabar}) with $v(t,\cdot)$ in place of $f$. Suppose that $N=2$ and that $X$ consists of an edge $e_0$ connecting $q_0$ to a vertex $p$ and edges $e_1$ and $e_2$ connecting $p$ to $q_1$ and $q_2$, respectively. Suppose first that $\mathring{v}$ is a bump compactly supported in the interior of $e_0$. In special case (a) the bump travels with physical velocity $h_{e_0}'$ (in the sense of Remark \ref{R:inverse}) on $e_0$ until it reaches $p$. There it splits, one copy travels along $e_1$ with physical velocity $h_{e_1}'$, the other along $e_2$ with physical velocity $h_{e_2}'$ until both eventually vanish at $q_1$ respectively $q_2$. If $h_{e_1}'=h_{e_2}'=\frac12h_{e_0}'$, then the special case (b) gives periodic solutions: The bump travels and splits as in case (a), but when the two copies reach $q_1$ respectively $q_2$, the bump enters again at $q_0$. 
\end{examples}

\begin{examples}\label{Ex:SGgraphwellposed}
We continue Examples \ref{Ex:repSgraph} and \ref{Ex:solenoidal} (iii) on the Sierpi\'nski gasket graph. Let $K^{(1)}$, $\mathcal{E}$ and $b$ be as specified there and $B=\emptyset$. Recall the formulas (\ref{E:SGgraphHpm}) and (\ref{E:SGgraphGpm}) for $H_\pm$, $G_-$ and $G_+$. 
\begin{enumerate}
\item[(a)] Suppose that $\Theta=\mathrm{id}$ is the identity. A function $f\in \mathcal{D}(\partial^\bot)$ is a member of $\mathcal{D}(A^\mathrm{id})$ if and only if it satisfies 
\begin{equation}\label{E:SGgraphzizero}
z_i=0,\quad i=0,1,2.
\end{equation}
Consequently
\[\mathcal{D}(A^\mathrm{id})=\{f=g+\sum_{i=0}^2 w_i\varphi_i:\ g\in H^1(K^{(1)})\},\]
here the $\varphi_i$ are as in Examples \ref{Ex:repSgraph}. Recall that $z_\emptyset=0$ by (\ref{E:zemptysetexposed}) and note that the constant $w_\emptyset$ can be absorbed into $g$. The weighted Kirchhoff conditions (\ref{E:Kirchhoff1}) and (\ref{E:Kirchhoff2}), necessarily satisfied for any element $f$ of $\mathcal{D}(\partial^\bot)$, say that $f$ is continuous at $q_0$, $q_1$, $q_2$ and that at $p_0$, $p_1$ and $p_2$ the sum of the \enquote{outer} incoming component plus twice the \enquote{inner} incoming component of $f$ equals the sum of the \enquote{outer} outgoing component plus twice the \enquote{inner} outgoing component of $f$. For $f\in \mathcal{D}(A^\mathrm{id})$ jumps inside $K_0^{(1)}$ may occur at $p_1$ and $p_2$, but they can only stem from $w$ and therefore must have the same size. In fact, by (\ref{E:z0}) and (\ref{E:SGgraphzizero}) there is a constant $c_0\in \mathbb{R}$ such that on $K_0^{(1)}$ we have 
\[f_{p_1q_0}(0)=f_{p_1p_2}(0)+c_0\quad \text{and}\quad f_{q_0p_2}(1)=f_{p_1p_2}(1)+c_0,\]
and since $f_{p_1q_0}(0)=g(p_1)+w_\emptyset+w_0$ and $f_{p_1p_2}(0)=g(p_1)+w_\emptyset-\frac12w_0$, it follows that $c_0=\frac32 w_0$. On $K_1^{(1)}$ and $K_2^{(1)}$ the situation is similar.

The solution of the boundary initial value problem (\ref{E:simplebivp}) has to satisfy (\ref{E:Kirchhoff1}), (\ref{E:Kirchhoff2}) and (\ref{E:SGgraphzizero}) at all times. Suppose that $\mathring{v}=\mathring{g}\in H^1(K^{(1)})$ is a bump with compact support \enquote{around $p_1$}, but not containing any full edge. It follows $b$ in a periodic motion: At $p_1$ it splits into an \enquote{outer} part traveling along $p_1q_0$ and $q_0p_2$ with physical velocity $1$  (in the sense of Remark \ref{R:inverse}) and an \enquote{inner} part traveling along $p_1p_2p_0$ with physical velocity $2$. After one full cycle and another crossing of $p_1p_2$ the inner part reunites with the outer at $p_2$, just to split again and to pass through $K_1^{(1)}$ in the same manner, and so on.
\item[(b)] Another interesting choice is $\Theta=\frac17\mathrm{id}$. A function $f\in \mathcal{D}(\partial^\bot)$ is in $\mathcal{D}(A^{\frac17\mathrm{id}})$ if and only if it satisfies $\frac{1}{7\sqrt{2}}(w_i+z_i)=\frac{1}{\sqrt{2}}(w_i-z_i)$, $i=0,1,2$. Multiplying both sides with $\frac{7}{\sqrt{2}}$ and balancing, we see that this is equivalent to having $3w_i=4z_i$, $i=0,1,2$. A comparison with formula (\ref{E:3w4z}) below reveals that this is the same as to say that $f_{p_1q_0}(0)=f_{p_1p_2}(0)$, $f_{p_2q_1}(0)=f_{p_2p_0}(0)$ and $f_{p_0q_2}(0)=f_{p_0p_1}(0)$, which means that at each $p_i$ the two parts of the solution on the outgoing edges must be equal.
\end{enumerate}
\end{examples}

\begin{examples}\label{Ex:SGwellpos} Let $(\mathcal{E},\mathcal{C})$ be the standard energy form on the Sierpi\'nski gasket $K$.
\begin{enumerate}
\item[(i)] Let $B=V_0$ and let $b=\partial h$ with $h$ having normal derivatives $(dh)_{V_0}(q_0)=-2$ and $(dh)_{V_0}(q_1)=(dh)_{V_0}(q_2)=1$, as in Examples \ref{Ex:simpleexHpm} (iii). Let $H_-$, $H_+$, $G_-$ and $G_+$ be as there. Consider the linear contraction $\Theta:H_-\to H_+$ for which $\Theta|_{\ker \partial^\bot}=\mathrm{id}$ and $\overline{\Theta}=\Theta|_{\mathbb{R}^3\ominus \lin(\{(1,0,0)\})}\circ P_-|_{\mathbb{R}^3}$ from $\mathbb{R}^3$ into  $\mathbb{R}^3\ominus \lin(\{(0,-1,-1)\})$ is represented by the matrix 
\[\overline{\Theta}=\begin{pmatrix*}[r]
0 & -\frac12 & -\frac12\\
0 & 0 & 0\\
0 & 0 & 0
\end{pmatrix*}.\]
Then $f\in \mathcal{D}(\partial^\bot_{V_0})$ is in $\mathcal{D}(A^\Theta)$ if and only if $z=0$, $n_{V_0}(f\partial h)(q_1)+n_{V_0}(f\partial h)(q_2)=-n_{V_0}(f\partial h)(q_0)$
and $(du)_{V_0}(q_1)=(du)_{V_0}(q_2)=-\frac12 (du)_{V_0}(q_0)$. The vanishing of $z$ implies that $u$ is harmonic in $K\setminus V_0$. By the preceding relation for its normal derivatives $u$ must be a multiple of $h$, up to an additive constant. Using the fact that $n_{V_0}(w\partial h)=0$ for $w\in\ker \partial^\bot$ together with Remark \ref{R:finiteB} and identity (\ref{E:pulloutg}), we find that  
\begin{equation}\label{E:SGTnormal}
\mathcal{D}(A^\Theta)=\{f=g+w:\ g\in \mathcal{C},\ w\in \ker\partial^\bot,\ g(q_1)+g(q_2)=2g(q_0)\}.
\end{equation}
In contrast to periodic solutions on a metric tree, where mass leaves at $q_1$ and $q_2$ and re-enters at $q_0$, cf. special case (b) in Examples \ref{Ex:wellposedtrees}, we may add a superposition $w \in\ker \partial^\bot$ of stationary solutions along arbitrary lacunas.

\item[(ii)] Let $B=V_0$ be as before, but now let $b=\partial h$ with $h$ having normal derivatives $(dh)_{V_0}(q_0)=0$, $(dh)_{V_0}(q_1)=-\frac12$ and $(dh)_{V_0}(q_2)=\frac12$, as in Examples \ref{Ex:simpleexHpm} (iv). Let also $H_-$, $H_+$, $G_-$ and $G_+$ be as there. Consider the linear contraction $\Theta:H_-\to H_+$ for which $\Theta|_{\ker \partial^\bot}=\mathrm{id}$ and $\overline{\Theta}=\Theta|_{\mathbb{R}^3\ominus \lin(\{(\frac{1}{\sqrt{2}},1,0)\})}\circ P_-|_{\mathbb{R}^3}$ from $\mathbb{R}^3$ into  $\mathbb{R}^3\ominus \lin(\{(\frac{1}{\sqrt{2}},0,-1)\})$ is represented by 
\[\overline{\Theta}=\begin{pmatrix*}[r]
1& 0 & 0\\
0 & 0 & -1\\
0 & -1 & 0
\end{pmatrix*}.\] 
Now $f\in \mathcal{D}(\partial^\bot_{V_0})$ is in $\mathcal{D}(A^\Theta)$ if and only if $z=0$, $n_{V_0}(f\partial h)(q_1)=-n_{V_0}(f\partial h)(q_2)$ and 
\begin{equation}\label{E:fdhvsdu}
n_{V_0}(f\partial h)(q_0)=(du)_{V_0}(q_1)+(du)_{V_0}(q_2).
\end{equation}
The condition $z=0$ implies that $u$ must be harmonic in $K\setminus V_0$, and since $(dh)_{V_0}(q_0)=0$, we have $n_{V_0}(f\partial h)(q_0)=(du)_{V_0}(q_0)=-(du)_{V_0}(q_1)-(du)_{V_0}(q_2)$. Together with (\ref{E:fdhvsdu}) this gives $n_{V_0}(f\partial h)(q_0)=(du)_{V_0}(q_0)=0$ and $(du)_{V_0}(q_1)=-(du)_{V_0}(q_2)$. Similarly as before, $u$ must be a constant multiple of $h$, up to an additive constant. Now
\begin{equation}\label{E:SGTdegen}
\mathcal{D}(A^\Theta)=\{f=g+w:\ g\in \mathcal{C},\ w\in \ker\partial^\bot, \ g(q_1)=g(q_2)\}.
\end{equation}
In comparison to the case of periodic solutions on the interval where mass leaves at $1$ and re-enters at $0$, cf. Examples \ref{Ex:classtheta} (ii), we may again add stationary solutions along the lacunas.
\item[(iii)] Let $B= \emptyset$ and choose $b=d\zeta_\emptyset\in \ker\partial^\ast$. Let $H_\pm$, $G_-$ and $G_+$ be as in Examples \ref{Ex:solenoidal} (iv). Consider the identity map $\Theta=\mathrm{id}$. Given $f\in\mathcal{D}(\partial^\bot)$, the orthogonal projection of $z$ in $\ker \partial^\bot$ onto $H_\pm$ is zero, which means that $z$ is constant. Since $\int_K zd\nu_b=0$, it follows that $z=0$ and therefore $\star_b^{-1}\partial u=0$. Taking into account the structure of $w$, the preceding shows that  
\[\mathcal{D}(A^{\mathrm{id}})=\{f=g+\sum_{|\alpha|\geq 1}a_\alpha \star_b^{-1}d\zeta_\alpha:\ g\in \mathcal{C},\ a_\alpha\in\mathbb{R},\ \sum_{|\alpha|\geq 1}a_\alpha^2\|d\zeta_\alpha\|_{\mathcal{H}}^2<+\infty\}.\]
The difference to the circle case in Examples \ref{Ex:classtheta} (i) is the possibly presence of a superposition of stationary solutions $\star_b^{-1}d\zeta_\alpha\in H_{\pm}$ along any lacuna except $\ell_\emptyset$. In view of the local exactness of $d\zeta_\emptyset$, \cite{CGIS13, IRT12}, this is fully in line with (ii).
\end{enumerate}
\end{examples}

\begin{examples}\label{Ex:Kmwellpos}\mbox{}
\begin{enumerate}
\item[(i)] Recall Examples \ref{Ex:simpleexHpm} (v). There we provided boundary quadruples for the Dirichlet integral on the $m$-level graph approximation $K^{(m)}$ to the Sierpi\'nski gasket $K$ with $B=V_0$ and $b=\partial h^{(m)}$, where $h^{(m)}$ is the edge-wise linear extension of $h|_{V_m}$, $h$ being the harmonic function in Examples \ref{Ex:SGwellpos} (ii). \enquote{Restricting} the identity $\Theta|_{\ker \partial^\bot}=\mathrm{id}$ in Examples \ref{Ex:SGwellpos} (ii) to the finite dimensional space $\ker \partial^\bot$ in the current example and keeping $\overline{\Theta}$ as there, we arrive at analogous findings.
\item[(ii)] Now let $h$ be the harmonic function in Examples \ref{Ex:SGwellpos} (i) and let $b=\partial h^{(m)}$ the minimal energy-dominant velocity field on the reduced metric graph $\widetilde{K}^{(m)}$ as in Examples \ref{Ex:simpleexHpm} (vi). Proceeding as in (i) but with $\Theta|_{\ker \partial^\bot}=\mathrm{id}$ and $\overline{\Theta}$ from Examples \ref{Ex:SGwellpos} (i), we obtain similar results as there. Loosely speaking, the vanishing of $\partial h^{(m)}$ and $\nu_{\partial h^{(m)}}$ on the removed edge of $K^{(m)}$ would result in an infinite inverse velocity, which could be interpreted in the sense that a solution $v$ would see the endpoints of this edge as identified, see Figure \ref{F:identify} for $m=1$.
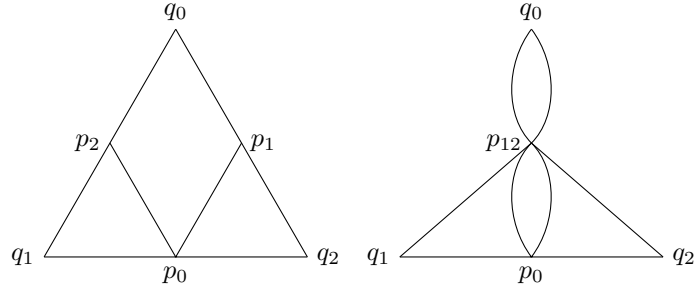
\begin{figure}[H]
	\centering
	\begin{tikzpicture}[scale=3.5]
		\draw (1/2, 0.866) --  (3/4, 0.433) --  (1/2, 0) --  (1, 0);
		\draw (1/2, 0.866) --   (1/4, 0.433) --  (1/2, 0)  --  (0, 0);
		\draw (3/4, 0.433) --  (1, 0);
		\draw (1/4, 0.433) -- (0, 0);
		\node[above] at (1/2, 0.866) {$q_0$};
		\node[left] at (0, 0) {$q_1$};
		\node[right] at (1, 0) {$q_2$};
		\node[below] at (1/2, 0) {$p_0$};
		\node[right] at (3/4, 0.433) {$p_1$};
		\node[left] at (1/4, 0.433) {$p_2$};
	\end{tikzpicture}
	\begin{tikzpicture}[scale=3.5]
		\draw (1/2, 0.866) .. controls (1/2+0.1, 0.733) and (1/2+0.1, 0.533)  ..   (1/2, 0.433) .. controls (1/2+0.1, 0.333) and (1/2+0.1, 0.133)  ..  (1/2, 0) --  (1, 0);
		\draw (1/2, 0.866) .. controls (1/2-0.1, 0.733) and (1/2-0.1, 0.533)  ..   (1/2, 0.433) .. controls (1/2-0.1, 0.333) and (1/2-0.1, 0.133)  ..   (1/2, 0)  --  (0, 0);
		\draw (1/2, 0.433) --  (1, 0);		
		\draw (1/2, 0.433) --  (0, 0);
		\node[above] at (1/2, 0.866) {$q_0$};
		\node[left] at (0, 0) {$q_1$};
		\node[right] at (1, 0) {$q_2$};
		\node[below] at (1/2, 0) {$p_0$};
		\node[left] at (1/2, 0.433) {$p_{12}$};
	\end{tikzpicture}
	\caption{Removing the edge $p_1p_2$ with $h'_{p_1p_2}=0$ from $K^{(1)}$ and identifying $p_1$ and $p_2$.}
	\label{F:identify}
\end{figure}

\end{enumerate}
\end{examples}

\section{Cylindrical initial conditions and approximation}\label{S:cylindrical}

We discuss the pull-back to $X$ of solutions on intervals.

\begin{proposition}\label{P:intervaltospace}
Let Assumptions \ref{A:basic} and \ref{A:complete} be in force. Assume that the martingale dimension of $(\mathcal{E},\mathcal{C})$ is one, $B$ is finite, $h\in \mathcal{C}$ is minimal energy-dominant and harmonic in $X\setminus B$ with $\min_B h=0$ and $\max_B h=a>0$. 
\begin{enumerate}
\item[(i)] Suppose that $V\in C^1((0,+\infty)\times \mathbb{R})$ satisfies $\partial_tV(t,y)+\partial_yV(t,y)=0$, $t>0$, $y\in (0,a)$.
Then 
\[v(t,x):=V(t,h(x)),\quad t\geq 0,\quad x\in X\setminus B,\] 
defines an element $v(t):=v(t,\cdot)$ of $C^1((0,+\infty),L^2(X,\nu_{\partial h}))\cap C((0,+\infty),\mathcal{C})$ satisfying
\begin{equation}\label{E:transferedeq}
\frac{d}{dt}v(t)+\star_{\partial h}^{-1}\partial v(t)=0, \quad t\in (0,+\infty).
\end{equation}
\item[(ii)] Given $\mathring{V}\in C^1(\mathbb{R})$, the function $V(t,y):=\mathring{V}(y-t)$, $t\geq 0$, $y\in \mathbb{R}$, is as
is as stated in (i), and $v(t,x)=\mathring{V}(h(x)-t)$, $t\geq 0$, $x\in X\setminus B$, satisfies (\ref{E:transferedeq}).
\end{enumerate}
\end{proposition}

\begin{examples}
If $X=[0,1]$, $B=\{0,1\}$, $a>0$, $h(x)=ax$ and $\mathring{V}\in C^1(\mathbb{R})$, then 
\[\mathring{V}(h(x)-t)=(\mathring{V}\circ h)(x-a^{-1}t),\quad t>0,\quad x\in (0,1).\]
\end{examples}

Under the stated assumptions the chain rule, \cite[Theorem 3.2.2]{FOT94}, implies that $\partial F(f)=F'(f)\partial f$ for all $F\in C^1(\mathbb{R})$ and $f\in \mathcal{C}$, cf. \cite[Remark 5.1]{HR16}. Here we use that $\partial\mathbf{1}=0$.

\begin{proof}
For any $t\geq 0$ we have 
\[\int_X |\partial_tV(t,h(x))|^2\nu_{\partial h}(dx)\leq \mathcal{E}(h)\sup_{y\in [0,a]}|\partial_tV(t,y)|^2\leq \mathcal{E}(h)\sup_{s\in [t-\varepsilon,t+\varepsilon]}\sup_{y\in [0,a]}|\partial_tV(s,y)|^2,\]
and similarly for $V(t,h)$ itself. By the mean value theorem and bounded convergence $\frac{d}{dt}v(t)=\partial_tV(t,h)$ and 
\[\lim_{s\to t}\int_X |\partial_tV(t,h(x))-\partial_tV(s,h(x))|^2\nu_{\partial h}(dx)=0.\]
Moreover, $v(t)\in\mathcal{C}$ and $\mathcal{E}(V(t,h))\leq \sup_{y\in [0,a]}|\partial_xV(t,y)|^2\mathcal{E}(h)$ by the Markov property, and              
\[\partial v(t)=\partial_xV(t,h)\partial h=\star_{\partial h}\partial_xV(t,h)=-\star_{\partial h}\partial_tV(t,h)=-\star_{\partial h}\frac{d}{dt}v(t).\] 
by the chain rule and the hypothesis. Together with the preceding this also gives $\lim_{s\to t}\mathcal{E}(v(t)-v(s))=0$
and therefore (i). Item (ii) now follows easily. 
\end{proof}

For simplicity we now concentrate on the special case of the Sierpi\'nski gasket $K$. We view a function $f\in C^1(S^1)$ as a periodic function on $\mathbb{R}$ with period $1$. The next observation shows that for cylindrical initial conditions and periodic boundary conditions Proposition \ref{P:intervaltospace} recovers the unique solutions of (\ref{E:simpleCauchy}) in the situations of Examples \ref{Ex:SGwellpos} (i) and (a slight modification of) Examples \ref{Ex:SGwellpos} (ii).

\begin{theorem}\label{T:cylindrical}
Let $(\mathcal{E},\mathcal{C})$ be the standard energy form on $K$ and $B=V_0$. Suppose that 
\begin{enumerate}
\item[(i)] $h\in\mathcal{C}$ is the harmonic function on $K\setminus V_0$ with boundary values $h(q_0)=0$, $h(q_1)=h(q_2)=1$ and $\Theta$ is as in Examples \ref{Ex:SGwellpos} (i) or
\item[(ii)] $h\in\mathcal{C}$ is the harmonic function on $K\setminus V_0$ with boundary values
$h(q_0)=\frac12$, $h(q_1)=0$ and $h(q_2)=1$ and $\Theta$ is as in Examples \ref{Ex:SGwellpos} (ii).
\end{enumerate}
Then for any $\mathring{V}\in C^1(S^1)$ and $w\in \ker \partial^\bot$ we have $\mathring{V}\circ h+w\in \mathcal{D}(A^\Theta)$ and 
\[T^\Theta(t)(\mathring{V}\circ h)=\mathring{V}(h-t)+w,\quad t\geq 0.\]
\end{theorem}

\begin{remark}
A similar statement can be proved for the situation of Examples \ref{Ex:SGwellpos} (iii) by piecing together.
\end{remark}

\begin{proof}
By Theorem \ref{T:wellposed} and Examples \ref{Ex:SGwellpos} (i) and (ii) it suffices to see that $\mathring{V}\circ h\in \mathcal{C}$ satisfies the conditions on $g$ in (\ref{E:SGTnormal}) and (\ref{E:SGTdegen}). But this is trivial since $\mathring{V}(1)=\mathring{V}(0)$ and therefore $\mathring{V}(h(q_i))=\mathring{V}(0)$, $i=0,1,2$.
\end{proof}

\begin{remark}\label{R:naivepos}
If in  Theorem \ref{T:cylindrical} we have $\mathring{V}\geq 0$ and $w=0$, then $T^\Theta(t)(\mathring{V}\circ h)\geq 0$ for all $t\geq 0$.
\end{remark}

Recall Examples \ref{Ex:simpleexHpm} (v) and (vi) and Examples \ref{Ex:Kmwellpos} (i) and (ii) on the $m$-level graph approximations to $K$. For initial conditions of cylindrical type the uniform convergence of solutions is easily seen.

\begin{corollary}\label{C:cylindrical} Let $\mathring{V}\in C^1(S^1)$.
\begin{enumerate}
\item[(i)] Let $h$ be harmonic $K\setminus V_0$ with boundary values $h(q_0)=\frac12$, $h(q_1)=0$ and $h(q_2)=1$ and let $\Theta$ be as in Examples \ref{Ex:SGwellpos} (ii). Let $h^{(m)}$ be the edge-wise linear extension of $h|_{V_m}$ and let $\Theta^{(m)}$ be as in Examples \ref{Ex:Kmwellpos} (i). Let $v$ be the unique solution of (\ref{E:simpleCauchy}) on $K$ with $b=\partial h$ and initial condition $\mathring{V}\circ h$
and $v^{(m)}$ the unique solution of (\ref{E:simpleCauchy}) on $K^{(m)}$ with $b=\partial h^{(m)}$ and initial condition $\mathring{V}\circ h^{(m)}$. Then 
\[\lim_{m\to\infty}\sup_{(t,x)\in [0,+\infty)\times K^{(m)}}\big|v(t,x)-v^{(m)}(t,x)\big|=0.\]
\item[(ii)] Let $h$ be harmonic $K\setminus V_0$ with boundary values $h(q_0)=0$, $h(q_1)=h(q_2)=1$ and let $\Theta$ be as in Examples \ref{Ex:SGwellpos} (i). Let $h^{(m)}$ be the edge-wise linear extension of $h|_{V_m}$ and let $\Theta^{(m)}$ be as in Examples \ref{Ex:Kmwellpos} (i). Let $v$ be the unique solution of (\ref{E:simpleCauchy}) on $K$ with $b=\partial h$ and initial condition $\mathring{V}\circ h$
and $v^{(m)}$ the unique solution of (\ref{E:simpleCauchy}) on $\widetilde{K}^{(m)}$ with $b=\partial h^{(m)}$ and initial condition $\mathring{V}\circ h^{(m)}$. Then 
\[\lim_{m\to\infty}\sup_{(t,x)\in [0,+\infty)\times \widetilde{K}^{(m)}}\big|v(t,x)-v^{(m)}(t,x)\big|=0.\]
\end{enumerate}
\end{corollary}

\begin{proof}
By Proposition \ref{P:intervaltospace} and since $K^{(m)}\subset K$, the pointwise evaluations make sense. We have 
\[|v(t,x)-v^{(m)}(t,x)|=\big|\mathring{V}(h(x)-t)-\mathring{V}(h^{(m)}(x)-t)\big|\leq \|\mathring{V}'\|_{\sup}|h(x)-h^{(m)}(x)|\] 
for any $(t,x)\in [0,+\infty)\times K^{(m)}$. On the other hand, 
\[\sup_{x\in K^{(m)}}|h(x)-h^{(m)}(x)|\leq \sup_{|\alpha|=m}|\max_{x\in V_m\cap K_\alpha}h(x)-\min_{x\in V_m\cap K_\alpha}h(x)|,\]
which converges to zero as $m$ goes to infinity. This gives (i), for (ii) replace $K^{(m)}$ by $\widetilde{K}^{(m)}$.
\end{proof}

\section{Adjoints, stationarity and comments on positivity}\label{S:statio}

For some periodic cases, related to Examples \ref{Ex:wellposedtrees}, \ref{Ex:SGgraphwellposed}, \ref{Ex:SGwellpos} and \ref{Ex:Kmwellpos},
we briefly look at adjoints.

\begin{lemma}\label{L:peradjoint}
Let Assumptions \ref{A:basic} and \ref{A:complete} be in force. Assume that the martingale dimension of $(\mathcal{E},\mathcal{C})$ is one and let $b\in \ker\partial_B^\ast$ be minimal energy-dominant. Suppose that we are in one of the following cases:
\begin{enumerate}
\item[(i)] We have $B= \emptyset$, $b\in\ker \partial^\bot$ and $\mathcal{D}(A^\Theta)=\{f=g+w:\ g\in\mathcal{C},\ w\in \ker \partial^\bot\ominus \mathbb{R}\}$.
\item[(ii)] We have $B=\{q_0,q_1,...,q_N\}$, $b=\partial h$ with $h\in\mathbb{H}_B$ satisfying $0<(dh)_B(q_i)=-\frac{1}{N}(dh)_B(q_0)$, $i=1,...,N$, and $\mathcal{D}(A^\Theta)=\{f=g+w:\ g\in\mathcal{C},\ w\in \ker\partial^\bot,\ \sum_{i=1}^N g(q_i)=g(q_0)\}$.
\item[(iii)] We have $B=\{q_0,q_1,q_2\}$, $b=\partial h$ with $h\in\mathbb{H}_B$ satisfying $-(dh)_B(q_1)=(dh)_B(q_2)>0$ and $(dh)_B(q_0)=0$ and $\mathcal{D}(A^\Theta)=\{f=g+w:\ g\in\mathcal{C},\ w\in \ker\partial^\bot,\ g(q_1)=g(q_2)\}$.
\end{enumerate}
Then $((A^\Theta)^\ast,\mathcal{D}((A^\Theta)^\ast))$ is an extension of $(-A^\Theta,\mathcal{D}(A^\Theta))$.
\end{lemma}

\begin{remark}\mbox{}
\begin{enumerate}
\item[(i)] In general this extension property may fail, a counterexample is the nilpotent case in Examples \ref{Ex:classtheta} (ii). For the circle or periodic interval case the operators do even coincide, but in general this cannot be expected here; item (i) in Lemma \ref{L:peradjoint} can give counterexamples.
\item[(ii)] By \cite[Proposition 4.1 and Theorem 4.2]{Arendtetal23} the skew-adjointness of $A^\Theta$ is equivalent to $\Theta$ being unitary. In such cases, including Examples \ref{Ex:SGgraphwellposed} (a) and \ref{Ex:SGwellpos} (iii), the conclusion of Lemma \ref{L:peradjoint} follows from this stronger result. 
\end{enumerate} 
\end{remark}

\begin{proof}
Recall that we have $(A_0,\mathcal{D}(A_0))=(-\star_b^{-1}\partial_B,\mathcal{C}_B)$ and $((A_0)^\ast,\mathcal{D}(A_0)^\ast))=(\partial_B^\bot,\mathcal{D}(\partial_B^\bot))$ in the context of Theorem \ref{T:Arendt}. Since $A_0\subset A^\Theta$, we have $(A^\Theta)^\ast\subset A_0^\ast=\partial_B^\bot$.

To see (i), note that for $f_1,f_2\in \mathcal{D}(A^\Theta)$ we have, in the notation of Theorem \ref{T:IbP}, $z_1=0$ and $z_2=0$. Consequently, by the last equality in (\ref{E:simplequadruple}), 
\[\left\langle A^\Theta f_1,f_2\right\rangle_{L^2(X,\nu_b)}=-\left\langle \partial^\bot f_1,f_2\right\rangle_{L^2(X,\nu_b)}=\left\langle f_1,\partial^\bot f_2\right\rangle_{L^2(X,\nu_b)}.\]
Since this holds for all $f_1\in \mathcal{D}(A^\Theta)$, we have $f_2\in \mathcal{D}((A^\Theta)^\ast)$ and $(A^\Theta)^\ast f_2=\partial^\bot f_2=-A^\Theta f_2$. Items (ii) and (iii) follow similarly using the last equality in (\ref{E:generalquadruple}).
\end{proof}

We conclude that in these special cases, nonzero stationary parts cannot appear suddenly. Although intuitively this is clear, a formal verification needs some arguments.
Let $P_{\ker A^\Theta}$ denote the orthogonal projection in $L^2(X,\nu_b)$ onto $\ker A^\Theta$.

\begin{corollary}\label{C:peradjoint}
Let the assumptions of Lemma \ref{L:peradjoint} (i), (ii) or (iii) be in force. If $\mathring{v}\in\mathcal{D}(A^\Theta)$ is such that $P_{\ker A^\Theta}\mathring{v}=0$, then for any $t>0$ the solution $v$ of (\ref{E:simpleCauchy}) satisfies  $P_{\ker A^\Theta}v(t)=0$.
\end{corollary}

\begin{proof}
The functional $f\mapsto p(f):=\|P_{\ker A^\Theta} f\|_{L^2(X,\nu_b)}$ is sublinear and continuous on $L^2(X,\nu_b)$. Its subdifferential $dp(f)$  at $f\in L^2(X,\nu_b)$ consists of all $\varphi\in L^2(X,\nu_b)$ such that $\left\langle g,\varphi\right\rangle_{L^2(X,\nu_b)}\leq p(g)$ for all $g\in L^2(X,\nu_b)$ and $\left\langle f,\varphi\right\rangle_{L^2(X,\nu_b)}=p(f)$. See \cite[p. 48]{Arendt86}.

Let $f\in\mathcal{D}(A^\Theta)$. If $P_{\ker A^\Theta}f=0$, then $0\in dp(f)$ and trivially $\left\langle A^\Theta f, 0\right\rangle_{L^2(X,\nu_b)}=0$. If $P_{\ker A^\Theta}f\neq 0$, consider the function $\varphi:=\|P_{\ker A^\Theta} f\|_{L^2(X,\nu_b)}^{-1}P_{\ker A^\Theta} f$. By orthogonality and Cauchy-Schwarz we have $\varphi\in dp(f)$. Lemma \ref{L:peradjoint} gives $\varphi\in \mathcal{D}((A^\Theta)^\ast)$ and $(A^\Theta)^\ast \varphi=0$, so that $\left\langle A^\Theta f, \varphi\right\rangle_{L^2(X,\nu_b)}=0$. 

In particular, $A^\Theta$ is $p$-dissipative in the sense of  \cite[Definition 2.1]{Arendt86}. Consequently \cite[Theorems 2.6 and 2.7]{Arendt86} give
$p(T^\Theta(t)\mathring{v})\leq p(\mathring{v})$, $t\geq 0$,
for all $\mathring{v}\in \mathcal{D}(A^\Theta)$.
\end{proof}

We end this section with some comments on positivity.

\begin{remark}
As illustrated in Remark \ref{R:naivepos}, the nonnegativity of the solution can be guaranteed for suitable boundary and loop conditions and certain nonnegative initial conditions. It would be desirable to know under which boundary and loop conditions the contraction semigroup $(T^\Theta(t)_{t>0}$ in Theorem \ref{T:wellposed} is positive in the sense that for any $f\in L^2(X,\nu_b)$ with $f\geq 0$ we have $T^\Theta(t)f\geq 0$, $t>0$. Since in general no explicit representation of the semigroup is available, it is interesting to look for sufficient conditions for positivity in terms of the generator $A^\Theta$. Recall that by the $m$-dissipativity of $A^\Theta$, the range of $\lambda-A^\Theta$ equals $L^2(X,\nu_b)$ for any $\lambda>0$. Therefore, by \cite[Theorem 2.1]{Phillips62}, a sufficient condition for positivity is that $A^\Theta$ is dispersive, that is, $\langle A^\Theta f,f\vee 0\rangle_{L^2(X,\nu_b)}\leq 0$, $f\in \mathcal{D}(A^\Theta)$. See also \cite[Theorem 1.7]{MaR95}.

If $X$ is a finite metric graph with edge set $E$, the boundary $B\subset X$ is finite and $b=(b_e)_{e\in E}$ is in $\ker \partial_B^\ast$, then, arguing similarly as in in \cite[Proposition 18.14]{BFR17}, we can see that it sufficies to have 
\[0\geq -\sum_{e\in E} \int_0^1f_e'(f_e\vee 0)\:dx=-\sum_{e\in E}b_e\big((f_e\vee 0)(1)^2-(f_e\vee 0)(0)^2\big)\]
for any $f=(f_e)_{e\in E}\in \mathcal{D}(A^\Theta)$. This can be written in the form of a matrix, which in some cases can even be used to obtain an explicit representation for $(T^\Theta(t)_{t>0}$, see \cite[Sections 18.2 and 18.3]{BFR17}. It is not difficult to see that for $c=c_1=c_2$  and any $\bar{\theta}\in [-1,1]$ the semigroup in Examples \ref{Ex:circleswellpos} is positive. 

For general $X$, suppose that, as in the situation of Lemma \ref{L:peradjoint} (i), (ii) or (ii), we have $\mathbf{1}\in \ker((A^\Theta)^\ast)$. If in addition 
\begin{equation}\label{E:Kato}
\left\langle \sgn(f)A^\Theta f,\mathbf{1}\right\rangle_{L^2(X,\nu_b)}\leq 0,\quad f\in \mathcal{D}(A^\Theta),
\end{equation}
then by the characterization provided in \cite[Corollary 3.9]{Arendt86b} the semigroup $(T^\Theta(t)_{t>0}$ is positive. This happens in particular if $f\in\mathcal{D}(A^\Theta)$ implies $|f|\in \mathcal{D}(A^\Theta)$. Since $\mathcal{D}(A^\Theta)$ is generally larger than $\mathcal{C}$, these properties do no seem easy to check. 

In the context of (\ref{E:ce})  inequality (\ref{E:Kato}) and the question of domain stability under $|\cdot|$ display an interesting parallel to prominent discussions of essential skew-adjointness, renormalization and chain rule properties in Euclidean situations. See \cite{Aizenman78, Depauw03, Nelson63} and the recent discussions in \cite{Crippaetal, Gusev24}.
\end{remark}

\addtocontents{toc}{\vspace{1\baselineskip}}

\appendix

\section{Remarks on energy measures}

Let Assumption \ref{A:basic} be in force. 

\begin{lemma}\label{L:coincide} Let Assumption \ref{A:basic} be in force.
A $\sigma$-finite Borel measure $\nu$ is energy-dominant for $(\mathcal{E},\mathcal{C})$ if and only if $\nu_\omega\ll\nu$ for all $\omega\in\mathcal{H}$.
\end{lemma}

We use the fact that for any $\omega,\eta\in\mathcal{H}$ and any Borel set $A\subset X$ we have
\begin{equation}\label{E:energymeascont}
|\nu_\omega(A)^{1/2}-\nu_\eta(A)^{1/2}|\leq\nu_{\omega-\eta}(A)^{1/2};
\end{equation}
this is the generalization of a well-known estimate, \cite[p. 111]{FOT94}, and is easily seen using \cite[Corollary 2.1]{HRT13}.

\begin{proof}
Suppose that $\nu$ is energy-dominant, that is, $\nu_f\ll\nu$ for all $f\in\mathcal{C}$. Given a finite linear combination $\sum_ig_i\partial f_i$ with $g_i\in C_c(X)$ and $f_i\in \mathcal{C}$ and a Borel set $A\subset X$ with $\nu(A)=0$, we have $\nu_{f_i,f_j}(A)=0$ by energy dominance and Cauchy-Schwarz, and therefore also
\[\nu_{\sum_i g_i\partial f_i,\sum_j g_j\partial f_j}(A)=\sum_{i,j}\int_A g_ig_j\:d\nu_{f_i,f_j}=0.\]
Given $\omega\in\mathcal{H}$, density and (\ref{E:energymeascont}) imply that $\nu_\omega(A)=0$. The other implication is trivial.
\end{proof}

\begin{lemma}\label{L:Hsep} Let Assumption \ref{A:basic} be in force.
The Hilbert space $\mathcal{H}$ is separable.
\end{lemma}

\begin{proof}
Let $(f_i)_{i=1}^\infty$ be $\mathcal{E}^{1/2}$-dense in $\mathcal{C}$ and $(g_i)_{i=1}^\infty$ dense in $C_c(X)$. Then the space 
\[S_{\mathbb{Q}}=\lin_\mathbb{Q}\big(\big\lbrace g_i\partial f_j:\ i,j=1,2,...\big\rbrace\big)\]
of finite linear combinations of elements $g_i\partial f_j$ with rational coefficients is dense in $\mathcal{H}$. Since by construction the space of finite linear combinations of elements $g\partial f$ with $f\in \mathcal{C}$ and $g\in C_c(X)$ is dense in $\mathcal{H}$, it suffices to show that each such element $g\partial f$ can be approximated by elements of $S_\mathbb{Q}$. Given such $g\partial f$ and $\varepsilon>0$, let $k$ be such that $\mathcal{E}(f-f_k)^{1/2}<\varepsilon/(2\|g\|_{\sup})$ and choose $\ell$ such that $\|g-g_\ell\|_{\sup}<\varepsilon/(2\mathcal{E}(f_k)^{1/2})$. Then 
\[\|g\partial f-g_\ell\partial f_k\|_{\mathcal{H}}\leq \|g\|_{\sup}\mathcal{E}(f-f_k)^{1/2}+\|g-g_\ell\|_{\sup}\mathcal{E}(f_k)^{1/2}<\varepsilon.\]
Using the triangle inequality the claimed density of $S_{\mathbb{Q}}$ now follows. 
\end{proof}

Lemma \ref{L:coincide} and Lemma \ref{L:Hsep} give a variant of \cite[Lemma 2.3]{Hino10}. 
\begin{corollary}\label{C:folklore} Let Assumption \ref{A:basic} be in force.
If $(\omega^{(i)})_{i=1}^\infty\subset \mathcal{H}$ is a sequence whose span is dense in $\mathcal{H}$ and $(a_i)_{i=1}^\infty$ is a sequence of positive real numbers such that $\sum_{i=1}^\infty a_i\|\omega^{(i)}\|_{\mathcal{H}}^2<+\infty$, then the measure 
$\sum_{i=1}^\infty a_i\nu_{\omega^{(i)}}$ is minimal energy-dominant.
\end{corollary}

The preceding gives the following variant of \cite[Proposition 2.7]{Hino10}.

\begin{proposition}\label{P:plenty}
Let Assumption \ref{A:basic} be in force. 
\begin{enumerate}
\item[(i)] The set of minimal energy-dominant $\omega\in\mathcal{H}$ is dense in $\mathcal{H}$. 
\item[(ii)] The set of minimal energy-dominant $f\in\mathcal{C}$ is dense in $\mathcal{C}$ with respect to the seminorm $\mathcal{E}^{1/2}$.
\end{enumerate}
\end{proposition}

The proof of \cite[Proposition 2.7]{Hino10} goes through with minor modifications: Take a complete orthonormal system $(\omega^{(i)})_{i=1}^\infty$ in $\mathcal{H}$; Lemma \ref{L:Hsep} ensures its existence. By Corollary \ref{C:folklore} the measure $\nu=\sum_{i=1}^\infty 2^{-i}\nu_{\omega^{(i)}}$ is minimal energy-dominant. By a martingale convergence argument, \cite[p. 273]{Hino10}, one can find a Borel set $X_0\subset X$ of full $\nu$-measure and Borel versions $Z^{ij}$ of the densities $\frac{d\nu_{\omega^{(i)},\omega^{(j)}}}{d\nu}$  that are defined at each $x\in X_0$ and such that $\sum_{i=1}^\infty 2^{-i}Z^{ii}(x)=1$, $x\in X_0$. The map $\Psi:\ell^2\to\mathcal{H}$, $\Psi(a)=\sum_{i=1}^\infty a_i2^{-i/2}\omega_i$, $a=(a_i)_{i=1}^\infty\in \ell^2$, is a contraction and its range is dense. Now one can follow the proof of \cite[Proposition 2.7]{Hino10} to find a dense subset $S$ of $\ell^2$ such that for any $a\in S$ we have $\frac{d\nu_{\Psi(a)}}{d\nu}>0$ $\nu$-a.e. on $X$. This implies that all elements of $\Psi(S)$ are minimal energy-dominant, and since by construction $\Psi(S)$ is dense in $\mathcal{H}$, this gives (i). Item (ii) can be seen using the same argument with $\mathcal{H}$ replaced by $\overline{\partial(\mathcal{C})}$ as in (\ref{E:Hodge}).

\section{Duality of continuity and transport type equations}\label{S:duality}

Let Assumption \ref{A:Dirichletform} (i) and (ii) be satisfied and recall (\ref{E:partialstar}) and (\ref{E:divnorm}). Given $V\in L^0(0,T;\mathcal{D}(\partial^\ast))$, we define $\partial^\ast V$ to be the $\mathcal{L}^1|_{(0,T)}$-equivalence class of $t\mapsto \partial^\ast (\tilde{V}(t))$, where $\tilde{V}$ is an arbitrary representative of $V$ with respect to $\mathcal{L}^1|_{(0,T)}$-equivalence. This gives a well-defined linear map $\partial^\ast:L^0(0,T;\mathcal{D}(\partial^\ast))\to L^0(0,T;L^2(X,\nu))$.

Now let $b \in L^2(0, T; L^\infty(X,(\mathcal{H}_x)_{x\in X},\nu))$ and $\mathring{U}\in \mathcal{D}(\partial^\ast)$. We call a function $U\in C([0,T];\mathcal{H})\cap L^0(0,T;\mathcal{D}(\partial^\ast))$ a \emph{solution} of 
\begin{equation}\label{E:tteH}
\begin{cases}
\partial_t U = (\partial^\ast U)b  &\text{on } (0,T) \times X,\\
U(0)=\mathring{U} &\text{on } X
\end{cases}
\end{equation}
\emph{in the integral sense} if $\partial^\ast U\in L^2(0,T;L^2(X,\nu))$ and $U(t)=\mathring{U}+\int_0^t\partial^\ast U(s) b(s)ds$, $t\in [0,T]$.

\begin{remark}
Under the stated hypotheses the function $t\mapsto (\partial^\ast U)b$ is an element of $L^1(0,T;\mathcal{H})$ and an integral solution $t\mapsto U(t)$, seen as an $\mathcal{H}$-valued function, is absolutely continuous on $[0,T]$.
\end{remark}
One can pass between problems (\ref{E:cemod}) and (\ref{E:tteH}) by temporal integration and spatial differentiation.

\begin{proposition}\label{P:duality}
Let Assumption \ref{A:Dirichletform} (i) and (ii) be satisfied, let $b\in L^2(0, T; L^\infty(X,(\mathcal{H}_x)_{x\in X},\nu))$, let $\mathring{U}\in \mathcal{D}(\partial^\ast)$ and 
$\mathring{u}:=\partial^\ast \mathring{U}$.
\begin{enumerate}
\item[(i)] If $u$ is a solution of (\ref{E:cemod}) in the weak sense, then 
$U(t):=\mathring{U}+\int_0^t u(s)b(s)ds$, $t\in [0,T]$,
is a solution of (\ref{E:tteH}) in the integral sense. 
\item[(ii)] If $U$ is a solution of (\ref{E:tteH}) in the integral sense, then $u:=\partial^\ast U$
is a solution of (\ref{E:cemod}) in the weak sense.
\end{enumerate}
\end{proposition}

\begin{proof}
If $u$ is a solution of (\ref{E:cemod}) in the weak sense and $U$ is as stated in Proposition \ref{P:duality} (i), then for any
$t\in [0,T]$ we have $U(t)\in\mathcal{H}$ and for any $\varphi\in\mathcal{C}$ the identity $\left\langle U(t),\partial\varphi\right\rangle_{\mathcal{H}}=\big\langle\mathring{U},\partial\varphi\big\rangle_{\mathcal{H}}+\int_0^t \left\langle u(s)b(s),\partial\varphi\right\rangle_{\mathcal{H}}ds$ holds. By Remark \ref{R:derivativeL1} this implies that $\frac{d}{dt}\left\langle U(\cdot),\partial\varphi\right\rangle_{\mathcal{H}}$ exists as an element of $L^1(0,T)$ and equals $t\mapsto \left\langle u(t)b(t),\partial\varphi\right\rangle_{\mathcal{H}}$ in this sense. For arbitrary $\psi\in C_c^1(0,T)$ identity (\ref{E:weaksol}) gives 
\[\int_0^T\psi'(t)\left\langle U(t),\partial\varphi\right\rangle_{\mathcal{H}}dt=-\int_0^T\psi(t)\left\langle u(t)b(t),\partial\varphi\right\rangle_{\mathcal{H}}dt=\int_0^T\psi'(t)\left\langle u(t),\varphi \right\rangle_{L^2(X,\nu)}dt.\]
This implies that there is a $\mathcal{L}^1$-null set $\mathcal{N}_{\varphi}\subset [0,T]$ such that 
\begin{equation}\label{E:coincide}
\left\langle U(t),\partial\varphi\right\rangle_{\mathcal{H}}=\left\langle u(t),\varphi\right\rangle_{L^2(X,\nu)}
\end{equation}
holds for all $t\in [0,T]\setminus \mathcal{N}_\varphi$. Since under the stated assumptions $(\mathcal{D}(\mathcal{E}),\|\cdot\|_{\mathcal{D}(\mathcal{E})})$ is separable, also 
$\mathcal{C}$ is separable with respect to $\|\cdot\|_{\mathcal{D}(\mathcal{E})}$. Let $\{\varphi_k\}_k\subset \mathcal{C}$ be a countable set which is $\|\cdot\|_{\mathcal{D}(\mathcal{E})}$-dense in $\mathcal{C}$ and let $\mathcal{N}:=\bigcup_k \mathcal{N}_{\varphi_k}$. Then also $\mathcal{N}$ is a $\mathcal{L}^1$-null set, and for $t\in [0,T]\setminus \mathcal{N}$ equality (\ref{E:coincide}) holds for all $\varphi\in \mathcal{C}$. Consequently we have $U(t)\in \mathcal{D}(\partial^\ast)$ and $\partial^\ast U(t)=u(t)$ for all such $t$. This implies (i). To see (ii) note that with a similar identification of derivatives integration by parts gives
\begin{multline}
\int_0^T\psi'(t)\left\langle u(t),\varphi\right\rangle_{L^2(X,\nu)}dt=\int_0^T\psi'(t)\left\langle U(t),\partial\varphi\right\rangle_{\mathcal{H}}dt=-\int_0^T\psi(t)\left\langle\partial^\ast U(t)b(t),\partial \varphi\right\rangle_\mathcal{H}dt-\psi(0)\big\langle\mathring{U},\partial\varphi\big\rangle_{\mathcal{H}}\notag\\
=-\int_0 \psi(t)\left\langle u(t)b(t),\partial \varphi\right\rangle_\mathcal{H}dt-\psi(0)\big\langle\mathring{u},\varphi\big\rangle_{L^2(X,\nu)}
\end{multline}
for any $\varphi\in\mathcal{C}$ and $\psi\in C_c^1([0,T))$.
\end{proof}

Now assume that $(\mathcal{E},\mathcal{C})$ satisfies Assumption \ref{A:basic}, the martingale dimension of $(\mathcal{E},\mathcal{C})$ is one, $\omega\in\mathcal{H}$ is minimal energy-dominant and Assumption \ref{A:Dirichletform} (i) holds with $\nu=\nu_\omega$. We prove Theorem \ref{T:duality}.

\begin{proof} Given a solution $u$ of (\ref{E:cemod}) in the weak sense, let $U$ be as in Proposition \ref{P:duality} with $\mathring{U}=\star_\omega \mathring{w}$ and set $w(t):=\star_\omega^{-1} U(t)$, $t\in (0,T]$. Then Proposition \ref{P:duality} (i),
the isometric property of $\star_\omega$, identity (\ref{E:pullstar}) and idempotence imply that 
\[w(t)-\mathring{w}=\star_\omega\int_0^t\partial^\ast U(s)b(s)ds=\int_0^t (\partial^\ast U(s))\star_\omega^{-1} b(s)ds=-\int_0^t(\partial^\bot w(s))\star_\omega^{-1} b(s)ds.\]
This shows (i). Statement (ii) follows from Proposition \ref{P:duality} (ii) since any solution $w$ of (\ref{E:tte}) in the integral sense gives a solution $U:=\star_\omega w$ of (\ref{E:tteH}) in the integral sense.
\end{proof}

\section{Distinguished representations}\label{S:distinguished}

We prove Lemma \ref{L:simpledistinguished}.

\begin{proof} Let $f\in \mathcal{D}(\partial_B^\bot)$ with representation (\ref{E:keydecomp}). By (\ref{E:simplequotient}) there is a unique constant $c\in\mathbb{R}$ such that 
\[G_-f-\frac{1}{\sqrt{2}}(c,-c)=\frac{1}{\sqrt{2}}\Big((g+n_B(fb)),(w+z)\Big).\]
Then the first identity in (\ref{E:simpledistinguished}) holds, and the orthogonality in (\ref{E:simplequotient}) gives
\begin{align}
0&=\sqrt{2}\left\langle G_-f,(1,-1)\right\rangle_{\widetilde{H}_{\pm}}\notag\\
&=\left\langle (g+c+n_B(fb),w-c+z),(\mathbf{1},-\mathbf{1})\right\rangle_{\widetilde{H}_{\pm}}\notag\\
&=\sum_B g+c\#(B)+\sum_Bn_B(fb)-\left\langle w,\mathbf{1}\right\rangle_{L^2(X,\nu_b)}+c\nu_b(X)-\left\langle z,\mathbf{1}\right\rangle_{L^2(X,\nu_b)}.\notag
\end{align}
Since $n_Bb\equiv 0$, we have
\begin{equation}\label{E:simplekey1}
\sum_B n_B(fb)=\sum_B (du)_B=\left\langle z,\mathbf{1}\right\rangle_{L^2(X,\nu_b)}
\end{equation}
by (\ref{E:keyNeumann}). This implies that 
\begin{equation}\label{E:simplekey2}
\sum_B g+c\#(B)-\left\langle w,\mathbf{1}\right\rangle_{L^2(X,\nu_b)}+c\nu_b(X)=0.
\end{equation}
Now 
\begin{align}
\langle (g+c-n_B(fb),w-c-z)&,(\mathbf{1},-\mathbf{1})\rangle_{\widetilde{H}_{\pm}}\notag\\
&=\sum_B g+c\#(B)-\sum_Bn_B(fb)-\left\langle w,\mathbf{1}\right\rangle_{L^2(X,\nu_b)}+c\nu_b(X)+\left\langle z,\mathbf{1}\right\rangle_{L^2(X,\nu_b)}\notag\\
&=0
\end{align}
by (\ref{E:simplekey1}) and (\ref{E:simplekey2}). By (\ref{E:simplequotient}) this shows the second identity in (\ref{E:simpledistinguished}).
\end{proof}

We prove Lemma \ref{L:distinguished}.

\begin{proof}
By (\ref{E:Hminus}) there is a unique number $c\in \mathbb{R}$ such that 
\[G_-f-c\Big(\frac{\mathbf{1}}{\sqrt{2}},\mathbf{1},0,-\frac{1}{\sqrt{2}}\star_b^{-1}v\Big)=\Big(\frac{1}{\sqrt{2}}(g+n_B(fb)), \frac{(du)_B}{|n_Bb|}, \frac{n_B(fb)}{|n_Bb|},\frac{1}{\sqrt{2}}(w+z)\Big).\]
With this very constant $c$ we have (\ref{E:distinguished-}) and 
\begin{align}
0&=\left\langle \Big(\frac{1}{\sqrt{2}}(g+c+n_B(fb)), \frac{(du)_B}{|n_Bb|}+c, \frac{n_B(fb)}{|n_Bb|},\frac{1}{\sqrt{2}}(w-c\star_b^{-1}v+z)\Big), \Big(\frac{\mathbf{1}}{\sqrt{2}},\mathbf{1},0,-\frac{1}{\sqrt{2}}\star_b^{-1}v\Big)\right\rangle_{\widetilde{H}_\pm}\notag\\
&=\frac12\sum_{\mathring{B}}(g+n_B(fb)+c)+\sum_{\check{B}}\big\lbrace (du)_B+c|n_Bb|\big\rbrace-\frac12\left\langle w+z-c\star_b^{-1} v,\star_b^{-1} v\right\rangle_{L^2(X,\nu_b)}\notag\\
&=\frac12\sum_{\mathring{B}}g+\frac12\sum_{\mathring{B}} n_B(fb)+\frac{c}{2}\#(\mathring{B})+\sum_{\check{B}}(du)_B+c\sum_{\check{B}}|n_Bb|\notag\\
&\qquad \qquad \qquad \qquad -\frac12\left\langle w,\star_b^{-1}v\right\rangle_{L^2(X,\nu_b)}-\frac12\left\langle z,\star_b^{-1}v\right\rangle_{L^2(X,\nu_b)}+\frac{c}{2}\|v\|_\mathcal{H}^2. \notag
\end{align}
Now $n_B(fb)=(du)_B$ on $\mathring{B}$ and 
\begin{equation}\label{E:generalkey1}
\sum_B(du)_B=\left\langle z,\mathbf{1}\right\rangle_{L^2(X,\nu_b)}=\left\langle z,\star_b^{-1} v\right\rangle_{L^2(X,\nu_b)}
\end{equation}
by (\ref{E:keyNeumann}). Using this and Remark \ref{R:divfreebalance}, we find that 
\begin{align}\label{E:generalkey2}
\frac{c}{2}\big(\#(\mathring{B})&+\sum_B|n_Bb|+\|v\|_\mathcal{H}^2\big)\notag\\
&=-\frac12\sum_{\mathring{B}}g+\frac12\left\langle w,\star_b^{-1}v\right\rangle_{L^2(X,\nu_b)}-\frac12\sum_{\mathring{B}}(du)_B-\sum_{\check{B}}(du)_B+\frac12\left\langle z,\star_b^{-1}v\right\rangle_{L^2(X,\nu_b)}\notag\\
&=-\frac12\sum_{\mathring{B}} g+\frac12\left\langle w,\star_b^{-1}v\right\rangle_{L^2(X,\nu_b)}-\frac12\sum_{\check{B}}(du)_B+\frac12\sum_{\hat{B}}(du)_B.
\end{align}
Now 
\begin{align}
&\left\langle \Big(\frac{1}{\sqrt{2}}(g+c-n_B(fb)), \frac{n_B(fb)}{|n_Bb|}, \frac{(du)_B}{|n_Bb|}-c,\frac{1}{\sqrt{2}}(w-c\star_b^{-1}v-z)\Big), \Big(\frac{\mathbf{1}}{\sqrt{2}},0,-\mathbf{1},-\frac{1}{\sqrt{2}}\star_b^{-1}v\Big)\right\rangle_{\widetilde{H}_\pm}\notag\\
&=\frac12\sum_{\mathring{B}}(g-n_B(fb)+c)-\sum_{\hat{B}}\big\lbrace (du)_B-c|n_Bb|\big\rbrace-\frac12\left\langle w-z-c\star_b^{-1} v,\star_b^{-1} v\right\rangle_{L^2(X,\nu_b)}\notag\\
&= \frac12\sum_{\mathring{B}}g-\frac12\sum_{\mathring{B}} n_B(fb)+\frac{c}{2}\#(\mathring{B})-\sum_{\hat{B}}(du)_B+c\sum_{\hat{B}}|n_Bb|\notag\\
&\qquad \qquad \qquad \qquad -\frac12\left\langle w,\star_b^{-1}v\right\rangle_{L^2(X,\nu_b)}+\frac12\left\langle z,\star_b^{-1}v\right\rangle_{L^2(X,\nu_b)}+\frac{c}{2}\|v\|_\mathcal{H}^2\notag\\
&=\frac12\sum_{\mathring{B}}g+\frac{c}{2}\#(\mathring{B})-\frac12\sum_{\hat{B}}(du)_B+ \frac12\sum_{\check{B}}(du)_B+c\sum_{\hat{B}}|n_Bb|-\frac12\left\langle w,\star_b^{-1}v\right\rangle_{L^2(X,\nu_b)}+\frac{c}{2}\|v\|_\mathcal{H}^2
\notag\\
&=0\notag
\end{align}
by (\ref{E:generalkey1}) and (\ref{E:generalkey2}), and this confirms (\ref{E:distinguished+}).
\end{proof}


\section{Calculations for the Sierpi\'nski gasket graph}\label{S:repSgraph}

We provide explicit calculations to verify  the identities (\ref{E:zemptysetexposed}), (\ref{E:uprimeexposed}), (\ref{E:w0exposed}) and (\ref{E:wemptysetexposed}) for the Sierpi\'nski gasket graph claimed in Examples \ref{Ex:repSgraph}. We use the setting and notation introduced there.

We first verify (\ref{E:zemptysetexposed}), that is,
\begin{equation}\label{E:zemptyset}
z_\emptyset=0.
\end{equation}
To see this, note first that - since the quantities in (\ref{E:guidequadratic}) are all constants - we have  
\begin{equation}\label{E:guideaffine}
u'_e(x)=u_e''x+u_e'(0)\quad \text{and}\quad  u_e(x)=\frac12 u_e''x^2+u_e'(0)x+u_e(0),\quad x\in [0,1],\quad e\in E_0.
\end{equation}
Combining (\ref{E:keydecomp}) and (\ref{E:Kirchhoff1}) for the first equality and taking into account (\ref{E:guidequadratic}) and (\ref{E:guideaffine}) for the second, we obtain 
\begin{equation}\label{E:uprimeq0}
u_{q_0p_2}'(0)=u_{p_1q_0}'(1)=z_\emptyset+z_0+u_{p_1q_0}'(0)
\end{equation}
at $q_0$ and analogous identities at $q_1$ and $q_2$.
Combining (\ref{E:keydecomp}) and (\ref{E:Kirchhoff2}) for the first equality and using the analogs on $K_2^{(1)}$ of (\ref{E:guidequadratic}) and (\ref{E:guideaffine}) for the second, we find 
\[u_{p_1q_0}'(0)+u_{p_1p_2}'(0)=u_{q_2p_1}'(1)+u_{p_0p_1}'(1)=5z_\emptyset-z_2+u_{q_2p_1}'(0)+u_{p_0p_1}'(0)\]
at $p_1$; again analogous identities hold at $p_0$ and $p_2$. Using the analog of (\ref{E:uprimeq0}) at $q_2$ to replace $u_{q_2p_1}'(0)$, we obtain 
\[u_{p_1q_0}'(0)+u_{p_1p_2}'(0)=6z_\emptyset + u_{p_0q_2}'(0)+u_{p_0p_1}'(0).\]
By symmetry we have the same kind of balance for $u_{p_0q_2}'(0)+u_{p_0p_1}'(0)$ and $u_{p_2q_1}'(0)+u_{p_2p_0}'(0)$ on the left-hand side, and combining these three equations we obtain (\ref{E:zemptyset}).

We now verify that the components of $\star_b^{-1}\partial u$ on the edges of $K_0^{(1)}$ satisfy (\ref{E:uprimeexposed}), that is,
\begin{equation}\label{E:uprime}
u_{p_1q_0}'(x)=z_0x-z_0,\quad u_{q_0p_2}'(x)=z_0x\quad \text{and}\quad \frac12u_{p_1p_2}'(x)=-z_0x+z_0,\quad x\in [0,1].
\end{equation}
By the continuity of $u$, (\ref{E:guideaffine}), (\ref{E:guidequadratic}) and (\ref{E:zemptyset}) we have 
\begin{equation}\label{E:usingcontofu}
u_{q_0p_2}(0)=u_{p_1q_0}(1)=\frac12 z_0+u_{p_1q_0}'(0)+u_{p_1q_0}(0)
\end{equation}
at $q_0$ and similar identities at $q_1$ and $q_2$. At $p_1$ we find, again by (\ref{E:guideaffine}), (\ref{E:guidequadratic}) and (\ref{E:zemptyset}), that $u_{q_2p_1}(1)=\frac12z_2+u_{q_2p_1}'(0)+u_{q_2p_1}(0)$. Using (\ref{E:Kirchhoff1}) at $q_2$ and (\ref{E:zemptyset}) and also the analog of (\ref{E:usingcontofu}) at $q_2$, we obtain  $u_{q_2p_1}'(0)=z_2+u_{p_0q_2}'(0)$ and $u_{q_2p_1}(0)=\frac12z_2+u_{p_0q_2}'(0)+u_{p_0q_2}(0)$. Pluggin in, this yields $u_{q_2p_1}(1)=2z_2+2u_{p_0q_2}'(0)+u_{p_0q_2}(0)$. Since 
\begin{equation}\label{E:innerpath}
u_{p_0p_1}(1)=-z_2+u_{p_0p_1}'(0)=u_{p_0p_1}(0) 
\end{equation}
and the continuity of $u$ gives $u_{p_0p_1}(1)=u_{q_2p_1}(1)$ and $u_{p_0p_1}(0)=u_{p_0q_2}(0)$, we obtain 
\begin{equation}\label{E:threez}
u_{p_0p_1}'(0)=2u_{p_0q_2}'(0)+3z_2.
\end{equation}
at $p_0$. Analogous identities hold at $p_1$ and $p_2$. The $p_1$-variant, together with (\ref{E:Kirchhoff2}) at $p_1$,  (\ref{E:Kirchhoff1}) at $q_2$ and (\ref{E:threez}) at $p_0$ as written, we find
\begin{align}
u_{p_1q_0}'(0)+3z_0+2u_{p_1q_0}'(0)&=u_{p_1q_0}'(0)+u_{p_1p_2}'(0)\notag\\
&=-z_2+u_{q_2p_1}'(0)+u_{p_0p_1}'(0)\notag\\
&=u_{p_0q_2}'(0)+u_{p_0p_1}'(0)\notag\\
&=\frac32 u_{p_0p_1}'(0)-\frac32 z_2.\notag
\end{align}
This gives 
\begin{equation}\label{E:p0p1}
u_{p_0p_1}'(0)=2u_{p_1q_0}'(0)+2z_0+z_2.
\end{equation}
Using (\ref{E:Kirchhoff2}) at $p_2$ and proceeding similarly, we also find
\begin{equation}\label{E:p2p0}
u_{p_2p_0}'(0)=2u_{p_1q_0}'(0)+2z_0+z_1.
\end{equation}
Summing (\ref{E:innerpath}) and its analogs on $p_1p_2$ and $p_2p_0$ along the loop $p_0p_1p_2$ we find
\[z_0+z_1+z_2=u_{p_1p_2}'(0)+u_{p_2p_0}'(0)+u_{p_0p_1}'(0)=6u_{p_1q_0}'(0)+6z_0+z_0+z_1+z_2;\]
the first equality holds since $u$ is continuous, the second follows using the $p_2$-variant of (\ref{E:threez}) and identities (\ref{E:p2p0}) and (\ref{E:p0p1}). Consequently
\[u_{p_0p_1}'(0)=-z_0,\quad u_{q_0p_2}'(0)=0\quad \text{and}\quad u_{p_1p_2}'(0)=z_0;\]
here we have used (\ref{E:Kirchhoff1}) for the second equality and the $p_2$-variant of (\ref{E:threez}) for the third. By (\ref{E:guidequadratic}) and (\ref{E:guideaffine}) this gives (\ref{E:uprime}).

We verify (\ref{E:w0exposed}), that is, 
\begin{equation}\label{E:w0}
w_0=\frac19(2f_{p_1q_0}(0)-2f_{p_1p_2}(0)+4f_{q_0p_2}(1)-4f_{p_1p_2}(1)).
\end{equation}
Note that by (\ref{E:zandw}) we have 
\begin{equation}\label{E:wparts}
w_{q_0p_2}=w_\emptyset+w_0,\quad w_{p_1p_2}=w_\emptyset-\frac12w_0\quad \text{and}\quad w_{p_1q_0}=w_\emptyset+w_0,
\end{equation}
which suggests to calculate
\begin{equation}\label{E:suggestion}
3w_0=w_{q_0p_2}+w_{p_1q_0}-2w_{p_1p_2}.
\end{equation}
The edge-wise evaluation 
\[w_e=f|_e(x)-g|_e(x)-b_e^{-1}u_e'(x),\quad x\in [0,1],\]
of (\ref{E:keydecomp}), together with (\ref{E:uprime}), gives 
\begin{align}
w_{q_0p_2}&=f_{q_0p_2}(0)-g(q_0)\notag\\
w_{p_1q_0}&=f_{p_1q_0}(0)-g(p_1)+z_0=f_{p_1q_0}(1)-g(q_0)\notag\\
w_{p_1p_2}&=f_{p_1p_2}(0)-g(p_1)-z_0.\notag
\end{align}
By (\ref{E:Kirchhoff1}) we have $w_{q_0p_2}=w_{p_1q_0}$. Inserting into (\ref{E:suggestion}) yields
\begin{equation}\label{E:3w4z}
3w_0=2f_{p_1q_0}(0)-2f_{p_1p_2}(0)+4z_0,
\end{equation}
and using (\ref{E:z0}) we arrive at (\ref{E:w0}). 

The remaining identity (\ref{E:wemptysetexposed}), that is, 
\begin{equation}\label{E:wemptyset}
w_\emptyset=\frac13(f_{q_0p_2}(0)-g(q_0)+f_{q_1p_0}(0)-g(q_1)+f_{q_2p_1}(0)-g(q_2))-\frac13\sum_{i=0}^2w_i,
\end{equation}
is now easily seen: The first identity in (\ref{E:wparts}) and its counterparts on $K_1^{(1)}$ and $K_2^{(1)}$ give
\begin{align}
3w_\emptyset&=w_{q_0p_2}-w_0+w_{q_1p_0}-w_1+w_{q_2p_1}-w_2\notag\\
&=(f_{q_0p_2}(0)-g(q_0)+f_{q_1p_0}(0)-g(q_1)+f_{q_2p_1}(0)-g(q_2))-w_0-w_1-w_2.\notag
\end{align}


\begin{thebibliography}
\normalsize

\bibitem{Aizenman78}
M. Aizenman, \emph{On vector fields as generators of flows: A counterexample to Nelson's conjecture}, Ann. Math. {\bf 107} (2),
287--296.
\bibitem{Allain}
G. Allain, \emph{Sur la repr\'esentation des formes de Dirichlet},
Ann. Inst. Fourier {\bf 25} (1975), 1-10.

\bibitem{Alonso18} 
P. Alonso-Ruiz, \emph{Explicit formulas for heat kernels on diamond fractals},
Comm. Math. Phys. {\bf 364} (3) (2018), 1305--1326.

\bibitem{AlonsoFreibergKigami18}
P. Alonso-Ruiz, U. Freiberg, J. Kigami, \emph{Completely symmetric resistance forms on the stretched Sierpinski gasket}, J. Fractal Geom. {\bf 5} (3) (2018), 227--277.

\bibitem{A04}
L. Ambrosio, \emph{Transport equation and Cauchy problem for BV vector fields}, Invent.
Math.{\bf 158} (2004), 227--260.

\bibitem{A08}
L. Ambrosio, \emph{Transport equation and Cauchy problem for non-smooth vector fields}, In: \emph{Calculus of Variations and Nonlinear Partial Differential Equations}, Dacorogna B., Marcellini P. (eds), Lecture Notes in Math. {\bf 1927}, Springer, Berlin, 2008, pp. 1--41.

\bibitem{AT14}
L. Ambrosio, D. Trevisan, \emph{Well posedness of Lagrangian flows and continuity equations in metric measure spaces}, Anal. PDE {\bf 7} (2014), 1179-1234.

\bibitem{Arendt86}
W. Arendt, \emph{Contraction semigroups and dissipative operators},
In: One-parameter semigroups of positive operators, R. Nagel ed., Lecture Notes in Math. \textbf{1184}, Springer, Berlin, 1986, pp. 47--53.

\bibitem{Arendt86b}
W. Arendt, \emph{Characterization of positive semigroups on Banach lattices},
In: One-parameter semigroups of positive operators, R. Nagel ed., Lecture Notes in Math. \textbf{1184}, Springer, Berlin, 1986, pp. 247--291.

\bibitem{Arendtetal23}
W. Arendt, I. Chalendar, R. Eymard, \emph{Extensions of dissipative and symmetric operators}, 
Semigroup Forum {\bf 106} (2023), 339--367.

\bibitem{Banasiak16}
J. Banasiak, A. Falkiewicz, P. Namayanja, \emph{Semigroup approach to diffusion and transport problems on networks}, Semigroup Forum {\bf 93} (3) (2016), 427--443.

\bibitem{Ba98} 
M. T. Barlow, \emph{Diffusions on fractals}, In: Lectures on Probability Theory and Statistics (Saint-Flour, 1995), 
Lecture Notes in Math. \textbf{1690}, Springer, Berlin, 1998, pp. 1--121.


\bibitem{BB89}
M. T. Barlow, R. F. Bass, \emph{The construction of Brownian motion on the Sierpinski carpet}, Ann. Inst. H. Poincaré {\bf 25} (1989) 225–257.



\bibitem{BBKT10}
M. Barlow, R. F. Bass, T. Kumagai, A. Teplyaev, \emph{Uniqueness of Brownian motion on Sierpi\`nski carpets},
J. Eur. Math. Soc. (JEMS) 12 (2010), no. 3, 655--701.

\bibitem{BE04}
M. T. Barlow, S. N. Evans, \emph{Markov processes on vermiculated spaces},
In: Random walks and geometry, Ed. by V. Kaimanovich, pp. 337-348, deGruyter, Berlin, 2004.

\bibitem{BP88}
M. T. Barlow, E. A. Perkins, \emph{Brownian motion on the Sierpinski gasket}, Probab. Theory Relat. Fields {\bf 79} (1988), 543--624.

\bibitem{BFR17}
A. B\'atkai, M. K. Fijavi\v{z}, A. Rhandi, \emph{Positive Operator Semigroups},
vol. 257 of Operator Theory: Advances and Applications, Birkh\"auser/Springer, Cham, 2017.

\bibitem{BK16}
F. Baudoin, D. J. Kelleher, \emph{Differential forms on Dirichlet spaces and Bakry- \'{E}mery estimates on metric graphs},  Trans. Amer. Math. Soc. {\bf 371}(5) (2019), 3145--3178.

\bibitem{BerkolaikoKuchment}
G. Berkolaiko, P. Kuchment, \emph{Introduction to Quantum Graphs}, Math. Surveys and Monographs, vol. 186, Amer. Math. Soc., Providence, 2013.

\bibitem{BeurlingDeny58}
A. Beurling, J. Deny, \emph{Espaces de Dirichlet I, Le cas \'el\'ementaire}, Acta Math. {\bf 99} (1958), 203--224.

\bibitem{Bogachev}
V. I. Bogachev, G. DaPrato, M. R\"ockner, S. V. Shaposhnikov, \emph{On the uniqueness of solutions to continuity equations},
J. Diff. Eq. {\bf 259} (2015), 3854--3873.


\bibitem{BouchutJames98}
F. Bouchut, F. James, \emph{One-dimensional transport equations with discontinuous coefficients}, Nonlinear Analysis, Theory, Methods and Appl. {\bf 32} (7) (1998), 891--933.

\bibitem{BH91}
N. Bouleau, F. Hirsch, \emph{Dirichlet Forms and Analysis on Wiener Space},
deGruyter Studies in Math. 14, deGruyter, Berlin, 1991.

\bibitem{Bressan14}
A. Bressan, S. \v{C}ani\'c, M. Garavello, M. Herty, B. Piccoli, \emph{Flows on networks: recent results and
perspectives}, EMS Surv. Math. Sci. {\bf 1} (2014), 47--111.

\bibitem{Camilli17}
F. Camilli, R. De Maio, A. Tosin, \emph{Transport of measures on networks}, Netw. Heterog. Media {\bf 12} (2) (2017), 191--215.

\bibitem{CaoQiu}
S. Cao, H. Qiu, \emph{Uniqueness and convergence of resistance forms on unconstrained Sierpinski carpets},
preprint (2024), arXiv: 2403:17311v1.

\bibitem{CGIS13}
F. Cipriani, D. Guido, T. Isola, J.-L. Sauvageot, \emph{Integrals and potentials of differential $1$-forms in the Sierpinski gasket}, Adv. Math. {\bf 239} (2013), 128--163.

\bibitem{CS03}
F. Cipriani, J.-L. Sauvageot, \emph{Derivations as square roots of Dirichlet forms},
J. Funct. Anal. {\bf201} (2003), 78-120.

\bibitem{CS09}
F. Cipriani, J.-L. Sauvageot, \emph{Fredholm modules on p.c.f. self-similar fractals and their conformal geometry},
Comm. Math. Phys. {\bf286} (2009), 541-558.


\bibitem{Crippadiss}
G. Crippa, \emph{The flow associated to weakly differentiable vector fields}, PhD thesis, S.N.S. Pisa and Univ. Z\"urich, 2007.


\bibitem{Crippaetal}
G. Crippa, N. Gusev, S. Spirito, E. Wiedemann, \emph{Failure of the chain rule for the divergence of bounded vector fields}, 
Ann. Sc. Norm. Super. Pisa Cl. Sci. (5), Vol. XVII (2017), 1--18.

\bibitem{Depauw03}
N. Depauw, \emph{Non unicit\'e des solutions born\'ees pour un champ de vecteurs BV en dehors d'un hyperplan},
C. R. Math. Acad. Sci. Paris {\bf 337} (2003), 249--252.

\bibitem{DPL89}
R. J. Di Perna, P. L. Lions, \emph{Ordinary differential equations, transport theory and Sobolev
spaces}, Invent. Math. {\bf 98} (1989), 511--547.


\bibitem{Dixmier81}
J. Dixmier, \emph{Von Neumann Algebras},
North-Holland Math. Lib. 27, North-Holland, Amsterdam, 1981.

\bibitem{Eberle99}
A. Eberle, \emph{Uniqueness and non-uniqueness of semigroups generated by singular diffusion operators}, Springer Lect. Notes Math. 1718, Springer, New York, 1999.

\bibitem{EngelKramar22}
K.-J. Engel, M. Kramar Fijav\v{z}, \emph{Flows on metric graphs with general boundary conditions}, J. Math. Anal. Appl. {\bf 513} (2) (2022), 126214.

\bibitem{EngelNagel00}
K.-J. Engel, R. Nagel, \emph{One-Parameter Semigroups for Linear Evolution Equations},
Grad. Texts in Math. 194, Springer, New York, 2000.

\bibitem{Evans}
L. C. Evans, \emph{Partial Differential Equations}, Graduate Studies in Mathematics, Vol. 19, Amer. Math. Soc., Providence, Rhode Island, 1998.

\bibitem{ExnerPost09}
P. Exner, O. Post, \emph{Approximation of quantum graph vertex couplings by scaled Schr\"odinger operators on thin branched manifolds}, J. Phys. A {\bf 42} (2009), 415305, 22.

\bibitem{FreibergHamblyHutchinson17}
U. Freiberg, B. M. Hambly, J. E. Hutchinson, \emph{Spectral asymptotics for  
V-variable Sierpinski gaskets}, Ann. Inst. H. Poincaré Probab. Statist. {\bf 53} (4), 2162--2213.

\bibitem{FOT94}
M. Fukushima, Y. Oshima and M. Takeda, \emph{Dirichlet forms and symmetric Markov processes},
deGruyter, Berlin, New York, 1994.

\bibitem{GK21}
A. Georgakopoulos, K. Kolesko, \emph{Brownian motion on graph-like spaces}, Studia Math. {\bf 261} (2) (2021), 121--155.

\bibitem{Goldberg}
S. I. Goldberg, \emph{Curvature and Homology}, Rev. ed., Dover Publications, Inc., Mineola, 1998.

\bibitem{Gusev19}
N. A. Gusev, \emph{On the one-dimensional continuity equation with a nearly incompressible vector field}, Commun. Pure Appl. Anal. {\bf 18}(2) (2019), 559--568.

\bibitem{Gusev24}
N. A. Gusev, \emph{The Nelson conjecture and chain rule property},
preprint (2024), arXiv:2411.09338v1.

\bibitem{Hambly92}
B. M. Hambly, \emph{Brownian motion on a homogeneous random fractal}, Probab. Theory Rel. Fields {\bf 94} (1992), 1--38.

\bibitem{Hambly97}
B. M. Hambly, \emph{Brownian motion on a random recursive Sierpinski gasket}, Ann. Probab. {\bf 25} (3) (1997), 1059--1102.

\bibitem{HamblyNyberg03}
B. Hambly, S. Nyberg, \emph{Finitely ramified graph-directed fractals, spectral asymptotics and the multidimensional
renewal theorem}, Proc. Edinb. Math. Soc. {\bf 46} (2) (2003), 1--34.

\bibitem{HamblyKumagai10}
B. M. Hambly, T. Kumagai, \emph{Diffusion on the scaling limit of the critical percolation cluster in the diamond hierarchical lattice}, Comm. Math. Phys. {\bf 295} (1) (2010), 29--69.


\bibitem{Hino08}
M. Hino, \emph{Martingale dimension for fractals},
Ann. of Probab. {\bf 36}(3) (2008),  971--991.

\bibitem{Hino10}
M. Hino, \emph{Energy measures and indices of Dirichlet forms, with applications to derivatives on some fractals},
Proc. London Math. Soc. {\bf100}  (2010), 269--302.

\bibitem{Hino13}
M. Hino, \emph{Upper estimate of martingale dimension
for self-similar fractals}, Probab. Theory Relat. Fields {\bf 156} (2013), 739--793.


\bibitem{Hinz16}
M. Hinz, \emph{Sup-norm closable bilinear forms and Lagrangians}, 
Ann. Mat. Pura Appl. {\bf 195} (4) (2016), 1021--1054.

\bibitem{HinzKelleherTeplyaev15}
M. Hinz, D. Kelleher, A. Teplyaev, \emph{Metrics and spectral triples for Dirichlet and forms}, J. Noncommut.
Geom.  {\bf 9} (2) (2015), 359--390.

\bibitem{HinzMeinert20}
M. Hinz, M. Meinert, \emph{On the viscous Burgers equation on metric graphs and fractals}, J. Fractal Geometry {\bf 7} (2) (2020), 137--182. 

\bibitem{HinzMeinert22}
M. Hinz, M. Meinert, \emph{Approximation of partial differential equations on compact resistance spaces}, Calc. Var. PDE {\bf 61} (2022), 19.

\bibitem{HRT13}
M. Hinz, M. R\"ockner, A. Teplyaev, \emph{Vector analysis for  Dirichlet forms and quasilinear PDE and SPDE on metric measure spaces}, Stoch. Proc. Appl. {\bf 123}(12) (2013), 4373--4406.

\bibitem{HR16}
M. Hinz, L. Rogers, \emph{Magnetic fields on resistance spaces},  J. Fractal Geometry {\bf 3} (2016),  75--93.

\bibitem{HT15}
M. Hinz,   A. Teplyaev, \emph{Local Dirichlet forms, Hodge theory, and
the Navier-Stokes equations on topologically one-dimensional
fractals}, Trans. Amer. Math. Soc. {\bf 367}  (2015),
1347--1380, Corrigendum in Trans. Amer. Math. Soc. {\bf 369} (2017), 6777-6778.

\bibitem{IRT12}
M. Ionescu, L. Rogers, A. Teplyaev, \emph{Derivations, Dirichlet forms and spectral analysis},
J. Funct. Anal. {\bf 263} (2012), no. 8, 2141--2169. 

\bibitem{Kaj13}
N. Kajino, \emph{Analysis and Geometry of the Measurable Riemannian Structure on the Sierpinski Gasket}, In: Fractal Geometry and Dynamical Systems in Pure and Applied Mathematics I: Fractals in Pure Mathematics, D. Carfi, M. L. Lapidus, 
E. P. J. Pearse, M. van Frankenhuijse, editors, Contemp. Math. {\bf 600} (2013), Amer. Math. Soc., pp. 91--133.

\bibitem{Ki89}
J. Kigami, \emph{A harmonic calculus on the Sierpinski space}, Japan J. Appl.
Math. {\bf 6} (1989), 259–290.

\bibitem{Ki93}
J. Kigami, \emph{Harmonic calculus on p.c.f. self-similar sets}, Trans. Amer. Math. Soc. {\bf 335} (2) (1993), 721--755.

\bibitem{Ki93b} J. Kigami,
\emph{Harmonic metric and Dirichlet form on the Sierpi\'nski gasket},
Asymptotic problems in probability theory: stochastic models and
diffusions on fractals (Sanda/Kyoto, 1990), 201--218,
Pitman Res. Notes Math. Ser., \textbf{283},
Longman Sci. Tech., Harlow, 1993.

\bibitem{Ki01}
J. Kigami, \emph{Analysis on Fractals}, Cambridge Univ. Press, Cambridge, 2001.

\bibitem{Ki03}
J. Kigami, \emph{Harmonic analysis for resistance forms}, J. Funct. Anal. {\bf 204} (2003), 525--544.


\bibitem{Ki23}
J. Kigami, \emph{Conductive Homogeneity of
Compact Metric Spaces and
Construction of p-Energy}, Mem. Eur. Math. Soc. 5, EMS Press, 2023.

\bibitem{Kostrykin99}
V. Kostrykin, R. Schrader, \emph{Kirchhoff's rule for quantum wires}, J. Phys. A {\bf 32} (1999), 595--630.

\bibitem{Kostrykin00}
V. Kostrykin, R. Schrader, \emph{Kirchhoff's rule for quantum wires. II: The inverse problem with possible applications to
quantum computers}, Fortschr. Phys. {\bf 48} (2000), 703--716.

\bibitem{KramarSikolya05}
M. Kramar, E. Sikolya, \emph{Spectral properties and asymptotic periodicity of flows in networks},
Math. Z. {\bf 249} (1) (2005), 139--162.

\bibitem{Kuchment04}
P. Kuchment, \emph{Quantum graphs I. Some basic structures}, Waves in Random media {\bf 14} (2004), 107--128.

\bibitem{Kuchment05}
P. Kuchment, \emph{Quantum graphs II. Some spectral properties of quantum and combinatorial graphs}, J. Phys. A: Math.
Theory {\bf 38} (2005), 4887--4900.

\bibitem{Kurasov}
P. Kurasov, \emph{Spectral Geometry of Graphs}, Operator Theory: Adv. and Appl., vol. 293, Birkh\"auser / Springer, Berlin, 2024. 

\bibitem{Ku89}
S. Kusuoka, \emph{Dirichlet forms on fractals and products of random matrices}, Publ. Res. Inst. Math. Sci. {\bf 25} (1989), 659--680.

\bibitem{Ku93} S. Kusuoka,
\emph{Lecture on diffusion process on nested fractals.}\
Lecture Notes in Math. \textbf{1567} 39--98,
Springer-Verlag, Berlin, 1993.

\bibitem{KN91}
K. Kuwae, Sh. Nakao, \emph{Time changes in Dirichlet space theory}, Osaka J. Math. {\bf 28} (1991), 847--865.

\bibitem{Laakso00}
T.J. Laakso, \emph{Ahlfors Q-regular spaces with arbitrary $Q>1$ admitting weak Poincar\'e inequality}, Geom. Funct. Anal. {\bf 10} (1) (2000), 111--123.

\bibitem{Landkof72}
N. S. Landkof, \emph{Foundations of Modern Potential Theory}, Grundlehren math. Wiss. 180, 
Springer, New York, 1972. 

\bibitem{LeJan}
Y. LeJan, \emph{Mesures associ\'ees \`a une forme de Dirichlet. Applications}, Bull. S.M.F. {\bf 106} (1978), 61--112.

\bibitem{Li} 
T. Lindstr{\o}m, \emph{Brownian motion on nested fractals.}\
Mem.  Amer.  Math.  Soc. \textbf{420}, 1989.

\bibitem{LQ19}
X. Liu, Zh. Qian, \emph{Parabolic type equations associated with the Dirichlet form on the Sierpinski gasket},
Probab. Theory Relat. Fields {\bf 175} (2019), 1063--1098.

\bibitem{LP61}
G. Lumer, R. S. Phillips, \emph{Dissipative operators in a Banach space}, Pac. J. Math. {\bf 11} (2) (1961), 679--698.

\bibitem{MaR95}
Z.-M. Ma, M. R\"ockner, \emph{Markov processes associated with positivity preserving coercive forms},
 Canad. J. Math. {\bf 47} (4) (1995), 817--840.

\bibitem{Mugnolo}
D. Mugnolo, \emph{Semigroup Methods for Evolution Equations on Networks}, Springer Intl. Switzerland, 2014. 

\bibitem{Nelson63}
E. Nelson, \emph{Les ecoulements incompressibles d'energie finie}, Coll. Int. Rech. Sci. {\bf 117} (1963), 159--165.

\bibitem{OSY02}
A. \"Oberg, R. S. Strichartz, A. Q. Yingst, \emph{Level sets of harmonic functions on the Sierpinski gasket}, Ark. Math. {\bf 40} (2002), 335--362.

\bibitem{Phillips59}
R. S. Phillips, \emph{Dissipative operators and hyperbolic systems of partial differential equation}, Trans. Am. Math. Soc. {\bf 90} (2) (1959), 193--254.

\bibitem{Phillips62}
R. S. Phillips, \emph{Semigroups of positive contraction operators}, Czechoslovak
Math. J. {\bf 12} (1962), 294--313.


\bibitem{Post12}
O. Post, \emph{Spectral Analysis on Graph-like Spaces}, Lect. Notes Math. {\bf 2039}, Springer, Berlin, 2012.

\bibitem{RT10}
L. Rogers, A. Teplyaev, \emph{Laplacians on the basilica Julia set}, Comm. Pure Appl. Anal. {\bf 9} (1) (2010), 211-231. 

\bibitem{ScheferThesis}
W. Schefer, \emph{First order differential operators and equations of continuity and transport-type on metric graphs and fractals}, PhD thesis, Bielefeld University, 2024.

\bibitem{Schefer24}
W. Schefer, \emph{On the domains of first order differential operators on the Sierpinski gasket}, preprint (2024), Bielefeld University.

\bibitem{Steinhurst}
B. Steinhurst, \emph{Uniqueness of locally symmetric Brownian motion on Laakso spaces}, Pot. Anal. {\bf 38} (1) (2013), 281--298.


\bibitem{Str00} R. S. Strichartz,
\emph{Taylor approximations on Sierpinski type fractals},
J. Funct. Anal.
\textbf{174} (2000), 76--127.

\bibitem{Str03}
R. S. Strichartz, \emph{Fractafolds based on the Sierpinski gasket and their spectra}, Trans. Amer. Math. Soc.
{\bf 355} (10) (2003), 4019--4043.

\bibitem{Str06}
R.S. Strichartz, \emph{Differential Equations on Fractals: A Tutorial}, Princeton Univ. Press, Princeton 2006.

\bibitem{Sturm94}
K.-T. Sturm, \emph{Analysis on local Dirichlet spaces. I: Recurrence, conservativeness and
Lp-Liouville properties}, J. Reine Angew. Math. {\bf 456} (1994), 173--196.

\bibitem{Takesaki79}
M. Takesaki, \emph{Theory of Operator Algebras}, Springer, New York, 1979.

\bibitem{T00}
A. Teplyaev, \emph{Gradients on fractals}, Journal Funct. Anal. {\bf 254} (2000), 1188--1216.

\bibitem{T08}
A. Teplyaev, \emph{Harmonic coordinates on fractals with finitely ramified cell structure}, Canad. J. Math. {\bf 60} (2008), 457--480.

\bibitem{Ts19}
K. Tsougkas, \emph{Non-degeneracy of the harmonic structure on Sierpinski gaskets}, J. Fractal Geom. {\bf 6} (2) (2019), 143--156.



\end{thebibliography}
\end{document}